\documentclass[reqno]{amsart}
\usepackage[a4paper,left=3cm,right=3cm,top=2.5cm,bottom=2.5cm]{geometry}
\usepackage[english]{babel}
\usepackage[utf8]{inputenc}
\usepackage[T1]{fontenc}
\usepackage{amsmath, amsthm, amssymb}
\usepackage{graphicx}
\usepackage{subcaption}
\usepackage[hyphens]{url}
\usepackage{booktabs}
\usepackage{lineno}
\usepackage{algorithm}
\usepackage{algpseudocode}
\usepackage{siunitx}
\usepackage{makecell}
\usepackage{colortbl}
\usepackage{xcolor}
\usepackage{tikz}
\usepackage{pgfplots}
\usepackage{pifont}
\usepackage [autostyle, english = american]{csquotes}
\MakeOuterQuote{"}
\usepackage[backend=biber,doi=true, isbn=false, url=false]{biblatex}
\usepackage{xcolor}
\definecolor{darkblue}{rgb}{0.0, 0.0, 0.55}
\definecolor{darkred}{rgb}{0.55, 0.0, 0.0}
\definecolor{darkgreen}{rgb}{0.0, 0.55, 0.0}

\usepackage{hyperref}
\hypersetup{
    colorlinks=true,
    linkcolor=darkblue,
    filecolor=magenta,
    urlcolor=cyan,
    citecolor=darkblue,
    pdftitle={Efficient Numerical Strategies for Regularized Semi-Discrete Optimal Transport},
    pdfauthor={Moaad Khamlich, Francesco Romor, Gianluigi Rozza},
    pdfkeywords={Optimal Transport, Semi-Discrete, Entropic Regularization, Numerical Methods, Computational Efficiency}
}

\newtheorem{theorem}{Theorem}[section]

\theoremstyle{definition}
\newtheorem{definition}[theorem]{Definition}
\theoremstyle{remark}
\newtheorem{remark}[theorem]{Remark}
\theoremstyle{definition}
\newtheorem{problem}[theorem]{Problem}

\newcommand{\R}{\mathbb{R}}
\newcommand{\Prob}{\mathcal{P}}
\newcommand{\X}{\mathcal{X}} %
\newcommand{\Y}{\mathcal{Y}} %
\newcommand{\OmegaDomain}{\Omega} %
\newcommand{\dd}{\,\mathrm{d}} %
\newcommand{\KL}{\mathrm{KL}} %
\newcommand{\argmin}{\operatornamewithlimits{argmin}}
\newcommand{\argmax}{\operatornamewithlimits{argmax}}
\newcommand{\defeq}{\mathrel{:=}} %
\newcommand{\bpsi}{\boldsymbol{\psi}} %
\newcommand{\bnu}{\boldsymbol{\nu}} %
\newcommand{\calT}{\mathcal{T}} %
\newcommand{\W}{\mathcal{W}} %

\pgfplotsset{compat=1.18}

\begin{document}

\title[Efficient RSOT Strategies]{Efficient Numerical Strategies for Entropy-Regularized Semi-Discrete Optimal Transport}

\author[Khamlich]{Moaad Khamlich$^{1}$}
\author[Romor]{Francesco Romor$^{2}$}
\author[Rozza]{Gianluigi Rozza$^{1}$}

\thanks{$^{1}$SISSA, International School for Advanced Studies, Mathematics Area, mathLab, via Bonomea 265, 34136 Trieste, Italy. Email: moaad.khamlich@sissa.it, gianluigi.rozza@sissa.it}
\thanks{$^{2}$WIAS, Weierstrass Institute for Applied Analysis and Stochastics, Mohrenstr. 39, 10117 Berlin, Germany. Email: francesco.romor@wias-berlin.de}

\begin{abstract}
Semi-discrete optimal transport (SOT), which maps a continuous probability measure to a discrete one, is a fundamental problem with wide-ranging applications. Entropic regularization is often employed to solve the SOT problem, leading to a regularized (RSOT) formulation that can be solved efficiently via its convex dual. However, a significant computational challenge emerges when the continuous source measure is discretized via the finite element (FE) method to handle complex geometries or densities, such as those arising from solutions to Partial Differential Equations (PDEs). The evaluation of the dual objective function requires dense interactions between the numerous source quadrature points and all target points, creating a severe bottleneck for large-scale problems. This paper presents a cohesive framework of numerical strategies to overcome this challenge. We accelerate the dual objective and gradient evaluations by combining distance-based truncation with fast spatial queries using R-trees. For overall convergence, we integrate multilevel techniques based on hierarchies of both the FE source mesh and the discrete target measure, alongside a robust scheduling strategy for the regularization parameter. When unified, these methods drastically reduce the computational cost of RSOT, enabling its practical application to complex, large-scale scenarios. We provide an open-source C++ implementation of this framework, built upon the \texttt{deal.II} finite element library, available at \url{https://github.com/SemiDiscreteOT/SemiDiscreteOT}.
\end{abstract}
\maketitle

\section{Introduction}
\label{sec:introduction}

Optimal Transport (OT) stands as a fundamental mathematical framework for analyzing probability measures, determining efficient ways to transform one distribution into another under specified cost constraints \cite{santambrogio2015optimal}. This field has evolved from Monge's pioneering work \cite{Monge1781} to Kantorovich's modern formulation \cite{Kantorovich1942}, establishing itself as a cornerstone of mathematical analysis. The framework's significance emerges particularly through the Wasserstein distance which metrizes the space of probability measures and provide more geometrically meaningful properties than traditional statistical metrics like the Kullback-Leibler divergence \cite{villani2003topics, villani2008optimal}. This geometric perspective has fueled its adoption across diverse fields. These range from its foundational role in understanding certain Partial Differential Equations (PDEs) as gradient flows \cite{Jordan1998, Ambrosio2008} to its use in machine learning \cite{peyre2019computational} (e.g., for generative models like Wasserstein GANs \cite{WGAN, WGAN_training}, domain adaptation \cite{Courty2017}, word embeddings \cite{Yamada2025}, and defining robust loss functions \cite{Frogner2017}). In scientific computing, it provides a powerful tool for comparing complex data, such as density fields arising from solutions to PDEs \cite{Haker2004, Papadakis2014}.

A frequently encountered variant is semi-discrete optimal transport (SOT), where mass from a continuous probability measure $\mu$, defined on a domain $\Omega \subset \R^d$, is transported to a discrete measure $\nu = \sum_{j=1}^N \nu_j \delta_{y_j}$, supported on a finite set $Y = \{y_1, \dots, y_N\}$. This setting arises naturally in problems like point cloud quantization, mesh generation, and matching continuous data, such as images or PDE solutions defined on a mesh, to discrete feature sets \cite{Gu2016, degoes2012}. For the standard quadratic cost $c(x,y) = \frac{1}{2}\|x-y\|^2$, the optimal transport map is piecewise constant, defined by a Laguerre diagram (also known as a power diagram) partition of the source domain \cite{Aurenhammer1991, merigot2011multiscale}. Efficient algorithms for computing this partition and the corresponding transport map, which form the basis of modern numerical SOT, have been developed by leveraging robust computational geometry techniques \cite{merigot2011multiscale, degoes2012, levy2015numerical, levy2018notions}.

Entropic regularization \cite{peyre2019computational,Cuturi2013} offers a popular alternative by adding a Kullback-Leibler divergence term to the OT objective, scaled by a parameter $\varepsilon > 0$. This approach has roots in the Schrödinger bridge problem \cite{Leonard2013} and has revolutionized computational OT. It yields a strictly convex problem with a unique, smoother optimal transport \emph{plan} $\pi_\varepsilon$, which converges to an optimal Kantorovich plan as $\varepsilon \to 0$ \cite{Chizat2021}. Furthermore, entropic regularization unlocks efficient computational methods based on its convex dual formulation. The most prominent of these is the Sinkhorn algorithm \cite{Cuturi2013}, a remarkably efficient block coordinate ascent method. Its stability, particularly when implemented in the log-domain \cite{Schmitzer2019}, and its high amenability to parallelization on GPUs \cite{Feydy2020}, have made regularized OT a practical tool for large-scale problems and enabled new applications like differentiable OT layers in deep learning \cite{Soumava2022, Frogner2017} and the fast computation of Wasserstein barycenters \cite{Puccetti2018, Benamou2015}.

This paper focuses on the efficient numerical solution of the regularized semi-discrete optimal transport (RSOT) problem, specifically addressing the computational bottleneck associated with its dual formulation \cite{Genevay2016}. Our approach leverages a quasi-Newton optimization scheme, specifically the L-BFGS algorithm, which stands in contrast to first-order Sinkhorn-like methods and provides a practical balance between the robustness of entropic regularization and superlinear local convergence of second-order methods \cite{dieci2024solving, bansil2021newton,wu2025pins}. The choice of L-BFGS over simpler first-order algorithms is motivated by its potential for faster convergence, especially for small $\varepsilon$ values, while avoiding the computational overhead of computing and inverting the full Hessian matrix required by Newton's method. While many successful acceleration strategies have been applied to fully discrete-to-discrete optimal transport problems \cite{Schmitzer2019, Solomon2015, Quellmalz2022}, their adaptation to the semi-discrete setting within a finite element (FE) framework presents unique challenges. Such a framework is essential when the source measure $\mu$ is defined on a complex domain or represented by a PDE solution, requiring numerical integration over an underlying computational mesh.

The core difficulty lies in the evaluation step for the OT dual objective and its gradient. In principle, for each source integration point (or quadrature point $x_i$ used to approximate integrals involving $\mu$), one needs to compute interactions with all $N$ target points $y_j$ via the Gibbs kernel $K(x, y_j) = \exp(-c(x,y_j)/\varepsilon)$. When $N$ is large, and/or the number of source quadrature points required for accurate integration is high, the $O(N)$ cost per source point becomes prohibitive, severely limiting scalability. This computational pattern mirrors the challenges faced in classical N-body problems, such as gravitational or electrostatic simulations \cite{Hockney2021}. Consequently, several approaches have explored accelerating this computation, often drawing inspiration from N-body techniques like fast multipole methods (FMM) \cite{Greengard1987, Carrier1988} and tree codes \cite{Barnes1986}, which reduce complexity by grouping distant interactions and using approximations (e.g., multipole expansions or center-of-mass approximations).

Adapting these ideas to the exponentially decaying Gibbs kernel has been explored, primarily in the discrete OT setting, leading to algorithms with near-linear or log-linear complexity in $N$ under certain conditions \cite{Genevay2016, Solomon_convolutional_2015,Schmitzer2019}. However, efficiently integrating these accelerators within a semi-discrete finite element context, handling the continuous source measure $\mu$ and the associated numerical integration, requires careful algorithmic design.

The challenges of entropic regularization extend beyond computational aspects. While it introduces beneficial mathematical properties, it also creates a bias that diffuses the transport plan compared to unregularized solutions. This diffusion might be problematic in applications requiring precise feature preservation \cite{Manole2024, Feydy2019}. Moreover, as $\varepsilon$ approaches zero in pursuit of reduced bias, numerical stability becomes a concern: dual potentials may grow unbounded, creating computational challenges even in log-space implementations \cite{Schmitzer2019, Chizat_2016}. The number of Sinkhorn iterations required for convergence also typically increases as $\varepsilon$ decreases \cite{Genevay2018SampleCO}, highlighting the need for robust methods that remain efficient at small $\varepsilon$ values.

We present a numerical framework integrating several computational strategies to overcome these challenges:

\begin{enumerate}
    \item \textbf{Enhanced Dual Computation:} We combine distance-based interaction truncation with efficient spatial indexing through R-trees \cite{Guttman1984}. This approach exploits the natural decay of the Gibbs kernel for distant pairs, particularly effective at small $\varepsilon$ values \cite{Schmitzer2019}. The integration of R-trees enables rapid neighbor queries, reducing the computational complexity from $O(N)$ toward $O(\log N)$ per evaluation point, with potential for $O(1)$ complexity in favorable scenarios.

    \item \textbf{Hierarchical Processing:} We employ coarse-to-fine strategies using hierarchies of both the source domain mesh and the target point set, accelerating convergence by leveraging solutions from coarser problem representations. These multigrid-inspired approaches have proven effective for various OT settings (both discrete \cite{Liu2021} and continuous \cite{Haker2004, Papadakis2014}).

    \item \textbf{Regularization Scheduling:} We incorporate a method for annealing the regularization parameter $\varepsilon$ (also referred to as $\varepsilon$-scaling), conceptually similar to annealing schedules used in optimization \cite{Kirkpatrick1983} and related to $\varepsilon$-scaling in auction algorithms for assignment problems \cite{Bertsekas1992}. This is often beneficial for robustness, convergence speed, and obtaining solutions closer to the unregularized problem by starting with a larger, more stable $\varepsilon$ and gradually decreasing it.
\end{enumerate}

These components synergize to create a robust and efficient approach for solving RSOT problems, particularly beneficial for scenarios involving large target point sets and small regularization parameters. We have developed a high-performance C++ library, \texttt{SemiDiscreteOT}, which implements these methods using the \texttt{deal.II} finite element library \cite{DEALII} and is publicly available at \url{https://github.com/SemiDiscreteOT/SemiDiscreteOT}.

We demonstrate the efficacy of this framework through a series of demanding computational studies, each designed to highlight a different aspect of its capabilities. These include: a comprehensive benchmarking and scalability analysis to establish a quantitative baseline for performance; the computation of Wasserstein barycenters and the geometric registration of complex 3D vascular geometries, which showcase the solver's utility in nested, iterative algorithms for shape analysis; and finally, an application to blue noise sampling on a Riemannian manifold, which underscores the framework's ability to handle non-Euclidean costs and to solve for the optimal locations of the target measure itself.

The paper is organized as follows: Section \ref{sec:formulation} establishes the mathematical foundations of RSOT and its dual formulation, followed by Section \ref{sec:discretization} which details the finite element discretization. Our core numerical strategies are presented in Section \ref{sec:acceleration} (covering dual evaluation speedups and regularization scheduling) and Section \ref{sec:multilevel} (hierarchical methods). The comprehensive validation of our framework is presented in Section \ref{sec:numerical_experiments}, with the results from the aforementioned applications. We conclude in Section \ref{sec:conclusion} with a discussion of results and future research directions.

\section{Mathematical Formulation of RSOT}
\label{sec:formulation}
Optimal Transport provides a principled framework for comparing probability measures by quantifying the minimal cost required to transport mass from a source distribution to a target distribution. We focus here on the Kantorovich formulation and its computationally advantageous entropy-regularized variant, specifically tailored to the semi-discrete setting, which is common in applications ranging from computational geometry to the analysis of PDE-based data.

Let $\mu$ be a source probability measure defined on a domain $\X = \OmegaDomain \subseteq \R^d$. We assume $\mu$ has a density $\rho(x)$ with respect to the Lebesgue measure $\dd x$. Let $\nu = \sum_{j=1}^N \nu_j \delta_{y_j}$ be a discrete target probability measure supported on a finite set of points $\Y = \{y_1, \dots, y_N\} \subset \R^d$, with given positive weights $\nu_j$ summing to one ($\sum_{j=1}^N \nu_j = 1$). Let $c: \X \times \Y \to \R^+$ be a continuous cost function defining the cost of transporting a unit of mass from a point $x \in \X$ to a point $y_j \in \Y$. A canonical choice is the squared Euclidean distance, $c(x, y_j) = \frac{1}{2}\|x - y_j\|^2$.

The Kantorovich approach relaxes the search for a deterministic transport map to finding an optimal \emph{transport plan}, which is a probability measure on the product space $\X \times \Y$.

\begin{definition}[Admissible Transport Plans $\Pi(\mu, \nu)$]
The set of \emph{admissible transport plans} $\Pi(\mu, \nu)$ consists of probability measures $\pi \in \Prob(\X \times \Y)$ that satisfy the marginal constraints:
\begin{equation}
\pi(A \times \Y) = \mu(A) \quad \text{and} \quad \pi(\X \times \{y_j\}) = \nu_j,
\end{equation}
for all measurable sets $A \subseteq \X$ and for each $j=1, \dots, N$.
\end{definition}

\begin{remark}[Semi-Discrete Plans]
\label{rem:semi_discrete_plans}
In our semi-discrete context, any plan $\pi \in \Pi(\mu, \nu)$ must have the specific structure $\pi = \sum_{j=1}^N \pi_j \otimes \delta_{y_j}$, where each $\pi_j$ is a positive measure on $\X$. The marginal constraint on $\mu$ implies that each $\pi_j$ must be absolutely continuous with respect to $\mu$, i.e., it has a density $p_j(x)$ such that $\dd\pi_j(x) = p_j(x) \dd\mu(x)$. The marginal constraints then require that $\sum_{j=1}^N \pi_j = \mu$ (as measures on $\X$) and $\int_\X \dd\pi_j(x) = \nu_j$ for each $j$.
\end{remark}

The goal of the original OT problem is to find the plan with the lowest total transport cost.

\begin{problem}[Kantorovich OT Problem]
\label{eq:kantorovich_problem}
The Kantorovich formulation seeks an optimal plan $\pi^* \in \Pi(\mu, \nu)$ that minimizes the total transport cost:
\begin{equation}
C_K = \inf_{\pi \in \Pi(\mu, \nu)} \mathcal{C}_K(\pi) \defeq \inf_{\pi \in \Pi(\mu, \nu)} \int_{\X \times \Y} c(x, y) \dd\pi(x, y).
\end{equation}

Under mild assumptions, such an optimal plan $\pi^*$ exists. Furthermore, Kantorovich duality connects this primal minimization problem to an equivalent dual maximization problem involving potential functions.
\end{problem}

The minimal cost $C_K$ defines the optimal transport cost, which we denote by $\W_c(\mu, \nu)$. If the cost is the p-th power of a metric, $c(x,y) = d(x,y)^p$, then $(\W_c(\mu, \nu))^{1/p}$ is the standard p-Wasserstein distance.

While fundamental, solving the exact Kantorovich problem \eqref{eq:kantorovich_problem} can be computationally intensive. Entropy regularization offers a popular and effective alternative by adding a penalty term to the objective function, leading to smoother problems solvable with faster algorithms. This penalty is based on the Kullback-Leibler (KL) divergence relative to a reference measure, typically the product measure $\mu \otimes \nu$.

\begin{definition}[Kullback-Leibler Divergence]
Given two probability measures $\pi, \gamma$ on the same space, where $\pi$ is absolutely continuous with respect to $\gamma$ ($\pi \ll \gamma$), the Kullback-Leibler (KL) divergence is:
\begin{equation}
\KL(\pi \| \gamma) = \int \log\left(\frac{\dd\pi}{\dd\gamma}\right) \dd\pi.
\end{equation}
If $\pi$ is not absolutely continuous w.r.t. $\gamma$, $\KL(\pi \| \gamma) = +\infty$. In our setting, the reference measure is $\mu \otimes \nu = \sum_{j=1}^N \nu_j (\mu \otimes \delta_{y_j})$.
\end{definition}

\begin{problem}[Entropy-Regularized OT (Primal)]
With a regularization parameter $\varepsilon > 0$, the entropy-regularized OT problem seeks the plan $\pi_\varepsilon \in \Pi(\mu, \nu)$ that minimizes the regularized cost:
\begin{equation}
\label{eq:reg_primal_sec2_expanded}
C_\varepsilon = \inf_{\pi \in \Pi(\mu, \nu)} \left\{ \int_{\X \times \Y} c(x, y) \dd\pi(x, y) + \varepsilon \KL(\pi \| \mu \otimes \nu) \right\}.
\end{equation}
\end{problem}
The minimal regularized cost $C_\varepsilon$ is often denoted by $\W_{\varepsilon,c}(\mu,\nu)$, which we will use for the computation of Wasserstein barycenters.

\begin{remark}[Effect of Regularization]
The $\varepsilon \KL(\pi \| \mu \otimes \nu)$ term penalizes plans that deviate significantly from the product measure $\mu \otimes \nu$, which represents statistical independence. This regularization ensures that the optimal plan $\pi_\varepsilon$ is unique and absolutely continuous with respect to $\mu \otimes \nu$ (i.e., $\pi_\varepsilon \ll \mu \otimes \nu$), possessing a density. This smoothness is key to efficient computational methods.
\end{remark}

Similar to the unregularized case, the entropy-regularized problem admits a dual formulation, and strong duality holds, meaning the optimal values of the primal and dual problems coincide. The optimal value $C_\varepsilon$ of the regularized problem \eqref{eq:reg_primal_sec2_expanded} can be obtained by solving the dual problem:
\begin{equation}
\label{eq:reg_dual_general_expanded}
C_\varepsilon = \sup_{\phi, \psi} \left\{ \int_\X \phi \dd\mu + \int_\Y \psi \dd\nu - \varepsilon \int_{\X \times \Y} \exp\left(\frac{\phi(x) + \psi(y) - c(x, y)}{\varepsilon}\right) \dd\mu(x) \dd\nu(y) + \varepsilon \right\}.
\end{equation}
where the supremum is over suitable potential functions $\phi: \X \to \R$ and $\psi: \Y \to \R$. In the semi-discrete setting, $\psi$ is represented by a vector $\bpsi = (\psi_1, \dots, \psi_N) \in \R^N$, where $\psi_j = \psi(y_j)$. The dual problem specializes to:
\begin{multline}
\label{eq:semidisc_dual_expanded}
C_\varepsilon = \sup_{\substack{\phi \in L^1(\mu) \\ \bpsi \in \R^N}} \Bigg\{ \int_\X \phi(x) \rho(x) \dd x + \sum_{j=1}^N \psi_j \nu_j
 - \varepsilon \sum_{j=1}^N \nu_j \int_\X \exp\left(\frac{\phi(x) + \psi_j - c(x, y_j)}{\varepsilon}\right) \rho(x) \dd x + \varepsilon \Bigg\}.
\end{multline}

A cornerstone result relates the optimal primal plan $\pi_\varepsilon$ to the optimal dual potentials $(\phi^*, \bpsi^*)$ that achieve the supremum in \eqref{eq:semidisc_dual_expanded}. Specifically, the density of each component measure $\pi_j$ of the optimal plan is given by a Gibbs-Boltzmann-type formula:
\begin{equation}
\label{eq:reg_plan_uv_expanded}
\frac{d\pi_j}{d\mu}(x) = \nu_j \exp\left( \frac{\phi^*(x) + \psi_j^* - c(x,y_j)}{\varepsilon} \right).
\end{equation}

The marginal constraints impose relationships between the potentials. The source marginal constraint $\sum_{j=1}^N \dd\pi_j(x) = \dd\mu(x)$ implies that the continuous potential factor $\exp(\phi^*(x)/\varepsilon)$ can be expressed solely in terms of the discrete potentials $\bpsi^*$:
\begin{equation}
\label{eq:u_from_v_expanded}
\exp(\phi^*(x)/\varepsilon) = \frac{1}{\sum_{k=1}^N \nu_k \exp\left( (\psi^*_k - c(x, y_k)) / \varepsilon \right)}.
\end{equation}

Substituting this back, we obtain an explicit formula for the density of each component measure $\pi_j$ with respect to $\mu$:
\begin{equation}
\label{eq:final_plan_density}
p_j(x | \bpsi^*) \defeq \frac{d\pi_j}{d\mu}(x) = \frac{\nu_j \exp\left( (\psi^*_j - c(x, y_j)) / \varepsilon \right)}{\sum_{k=1}^N \nu_k \exp\left( (\psi^*_k - c(x, y_k)) / \varepsilon \right)}.
\end{equation}

\begin{remark}[Plan Density Interpretation]
The function $p_j(x | \bpsi^*)$ represents the conditional density (or proportion of mass) at source location $x$ that is transported to the target point $y_j$ under the optimal regularized plan. By construction, $\sum_{j=1}^N p_j(x | \bpsi^*) = 1$ for all $x$, ensuring the source marginal constraint is satisfied locally. The optimal discrete potentials $\bpsi^*$ are implicitly determined by the target marginal constraints: $\int_\X p_j(x | \bpsi^*) \dd\mu(x) = \nu_j$ for all $j=1, \dots, N$.
\end{remark}

Crucially, the dual problem \eqref{eq:semidisc_dual_expanded} can be simplified. The optimization over the infinite-dimensional potential $\phi$ can be carried out analytically for any fixed discrete potential vector $\bpsi$. Plugging the expression for the optimal $\exp(\phi^*(x)/\varepsilon)$ from \eqref{eq:u_from_v_expanded} (which depends only on $\bpsi$) back into the dual objective \eqref{eq:semidisc_dual_expanded} results in a finite-dimensional optimization problem solely over $\bpsi \in \R^N$. This reduction is the key to making the problem computationally tractable, as it transforms an infinite-dimensional optimization problem over functions into a finite-dimensional one over the vector $\bpsi$. The reduced dual problem is equivalent to minimizing the following convex functional $J_\varepsilon(\bpsi)$ over $\R^N$ (see \cite{peyre2019computational} for the derivation):
\begin{equation}
\label{eq:dual_functional_to_minimize_redux}
\inf_{\bpsi \in \R^N} J_\varepsilon(\bpsi) \defeq \inf_{\bpsi \in \R^N} \left\{ \int_{\OmegaDomain} \varepsilon \log\left(\sum_{j=1}^N \nu_j \exp((\psi_j - c(x, y_j))/\varepsilon)\right) \rho(x) \dd x - \sum_{j=1}^N \psi_j \nu_j \right\}.
\end{equation}

The optimal $\bpsi^*$ minimizing $J_\varepsilon$ corresponds to the discrete potentials of the optimal dual solution.

\begin{remark}[Connection to Power Diagrams and Log-Sum-Exp Partitions]
It is crucial to connect our regularized problem to the geometry of the unregularized case. For the squared Euclidean cost, the solution to the unregularized SOT problem is given by a \emph{power diagram}, a partition of the source domain $\Omega$ into cells where each cell consists of points closest to a given target site $y_j$ in a weighted sense \cite{aurenhammer1998minkowski}. The weights of this diagram are precisely the unregularized dual potentials, and finding them is equivalent to a non-smooth convex minimization problem \cite{levy2015numerical}.

Our objective functional $J_\varepsilon(\bpsi)$ can be seen as a smooth approximation of this problem. The term $\varepsilon \log(\sum_{j} \exp(\dots))$ is a well-known smooth approximation of the maximum function (the "log-sum-exp" trick). However, this specific functional form has a deeper interpretation from the perspective of discrete choice theory \cite{Takesen2022}. It can be understood as the expected maximum utility for a population of agents at location $x$ choosing among options $\{y_j\}$, where their perceived utility is perturbed by i.i.d. Gumbel noise. Other choices for this noise distribution, corresponding to different ambiguity sets, yield alternative regularization schemes.

Therefore, minimizing $J_\varepsilon(\bpsi)$ is a smooth counterpart to finding the optimal power diagram weights, where the smoothing has a clear probabilistic motivation. The resulting transport plan density $p_j(x | \bpsi^*)$ from Eq. \eqref{eq:final_plan_density} defines a "soft" or "log-partition" of the domain, which converges to the discontinuous partition of the power diagram as $\varepsilon \to 0$.
\end{remark}

To minimize $J_\varepsilon(\bpsi)$ using gradient-based optimization algorithms, we need its gradient with respect to the components of $\bpsi$. Differentiating the functional $J_\varepsilon(\bpsi)$ defined in \eqref{eq:dual_functional_to_minimize_redux} with respect to a component $\psi_k$ yields:
\begin{align}
\frac{\partial J_\varepsilon}{\partial \psi_k} &= \int_{\OmegaDomain} \frac{\nu_k \exp((\psi_k - c(x, y_k))/\varepsilon)}{\sum_{j=1}^N \nu_j \exp((\psi_j - c(x, y_j))/\varepsilon)} \dd\mu(x) - \nu_k \nonumber \\
&= \int_{\OmegaDomain} p_k(x | \bpsi) \dd\mu(x) - \nu_k \label{eq:gradient_compact_expanded} \\
&= \int_\X \dd\pi_k(x) - \nu_k . \nonumber
\end{align}

The first-order optimality condition for minimizing $J_\varepsilon(\bpsi)$ is $\nabla J_\varepsilon(\bpsi^*) = 0$. Equation \eqref{eq:gradient_compact_expanded} shows that this condition is exactly $\int p_k(x | \bpsi^*) \dd\mu(x) = \nu_k$ for all $k=1, \dots, N$. This means that finding the minimum of the reduced dual functional $J_\varepsilon$ enforces the target marginal constraints for the optimal regularized transport plan $\pi_\varepsilon$. Therefore, the computational task for solving RSOT effectively boils down to minimizing the convex functional $J_\varepsilon(\bpsi)$ over $\R^N$. This is typically achieved using iterative methods (like gradient descent or, as used in this work, L-BFGS) that require repeated evaluation of the gradient \eqref{eq:gradient_compact_expanded}, which in turn involves computing integrals over the source domain $\OmegaDomain$.
\section{Finite Element Discretization and Numerical Integration}
\label{sec:discretization}

The computational core of solving the RSOT problem lies in minimizing the convex functional $J_\varepsilon(\bpsi)$ defined in \eqref{eq:dual_functional_to_minimize_redux} over $\R^N$.
We employ the finite element method (FEM) for spatial discretization and numerical quadrature for integral approximation.

\subsection{Finite Element Approximation}
Let $\OmegaDomain \subset \R^d$ be the source domain. We introduce a computational mesh $\calT_h$ consisting of non-overlapping cells $K$ (e.g., triangles or tetrahedra) such that $\overline{\OmegaDomain} = \cup_{K \in \calT_h} \overline{K}$. Let $h = \max_{K \in \calT_h} \text{diam}(K)$ denote the mesh size parameter.
As previously stated, the source measure $\mu$ has a density $\rho(x)$ with respect to the Lebesgue measure. We approximate this density using a finite element function $\rho_h(x)$ belonging to a suitable finite element space $V_h \subset L^2(\OmegaDomain)$ defined over $\calT_h$. For instance, $V_h$ could be the space of continuous piecewise polynomials of degree $k$ ($P_k$ Lagrange elements):
\begin{equation}
V_h = \{ v \in C^0(\overline{\OmegaDomain}) : v|_K \in P_k(K) \quad \forall K \in \calT_h \}.
\end{equation}
The approximation $\rho_h$ can be obtained, for example, by $L^2$-projection of $\rho$ onto $V_h$ or by interpolation at the finite element nodes if $\rho$ is sufficiently regular. Standard finite element theory \cite{CiarletFEM, BrennerScottFEM} provides error estimates. If $\rho$ belongs to the Sobolev space $H^{k+1}(\OmegaDomain)$, then under suitable mesh regularity assumptions, the $L^2$ error satisfies:
\begin{equation}
\|\rho - \rho_h\|_{L^2(\OmegaDomain)} \le C h^{k+1} |\rho|_{H^{k+1}(\OmegaDomain)},
\end{equation}
where $C$ is a constant independent of $h$ and $\rho$.

\subsection{Numerical Quadrature}
Integrals of the form $\int_{\OmegaDomain} f(x) \dd\mu(x) = \int_{\OmegaDomain} f(x)\rho(x)\dd x$ are approximated by summing contributions from each cell $K \in \calT_h$, using a numerical quadrature rule:
\begin{equation}
\int_{\OmegaDomain} f(x) \rho(x) \dd x \approx \sum_{K \in \calT_h} \int_K f(x) \rho_h(x) \dd x \approx \sum_{K \in \calT_h} \sum_{q=1}^{N_{q,K}} f(x_q^K) \rho_h(x_q^K) w_q^K,
\end{equation}
where $\{x_q^K\}_{q=1}^{N_{q,K}}$ are the quadrature points within cell $K$ and $\{w_q^K\}_{q=1}^{N_{q,K}}$ are the corresponding positive quadrature weights. The choice of quadrature rule (e.g., Gaussian quadrature) is crucial for accuracy. The rule should be chosen to integrate the terms appearing in the functional and gradient sufficiently accurately. Let $N_q = \sum_{K \in \calT_h} N_{q,K}$ be the total number of quadrature points across the mesh; this number directly impacts the computational cost of evaluating the functional and gradient, as discussed in the following section.

\begin{remark}[Computational Scale Implied by Quadrature]
\label{rem:quadrature_scale}
It is instructive to consider the computational scale implied by the numerical quadrature. Discretizing the integral over the source domain $\X$ using the quadrature rule $\int_{\OmegaDomain} f(x) \rho(x) \dd x \approx \sum_{K \in \calT_h} \sum_{q=1}^{N_{q,K}} f(x_q^K) \rho_h(x_q^K) w_q^K$ effectively transforms the continuous source measure $\mu = \rho(x) \dd x$ into an approximate discrete measure $\tilde{\mu}_h = \sum_{q=1}^{N_q} \mu_q \delta_{x_q}$, where $x_q$ represents the $q$-th global quadrature point and $\mu_q = \rho_h(x_q) w_q$ is its associated mass, where $w_q$ is the global quadrature weight, incorporating the cell-specific weight from the original rule.

Consequently, the semi-discrete problem $\mu \to \nu$ (where $\nu = \sum_{j=1}^N \nu_j \delta_{y_j}$) is recast, through this naive discretization approach, into a fully discrete optimal transport problem $\tilde{\mu}_h \to \nu$. This resulting discrete problem involves transporting mass between $N_q$ source points $\{x_q\}_{q=1}^{N_q}$ and $N$ target points $\{y_j\}_{j=1}^N$. The crucial point is that the total number of quadrature points, $N_q = |\calT_h| \times (\text{average points per cell})$, can become extremely large. For fine meshes $|\calT_h|$ or higher-order quadrature rules needed for accuracy, $N_q$ can easily reach millions, often vastly exceeding the number of target points $N$. Solving such a massive discrete-discrete OT problem directly using standard algorithms is computationally demanding.
\end{remark}

\subsection{Discretized Functional and Gradient}
Applying both the finite element approximation $\rho_h$ for the density and numerical quadrature for the integration, we obtain the fully discretized dual functional $J_\varepsilon^h(\bpsi)$:
\begin{equation}
\label{eq:discretized_functional}
J_\varepsilon^h(\bpsi) = \sum_{K \in \calT_h} \sum_{q=1}^{N_{q,K}} \varepsilon \log\left(\sum_{j=1}^N \nu_j \exp\left(\frac{\psi_j - c(x_q^K, y_j)}{\varepsilon}\right)\right) \rho_h(x_q^K) w_q^K - \sum_{j=1}^N \psi_j \nu_j.
\end{equation}
Similarly, the $k$-th component of the gradient $\nabla J_\varepsilon^h(\bpsi)$ is approximated as:
\begin{equation}
\label{eq:discretized_gradient}
\frac{\partial J_\varepsilon^h}{\partial \psi_k} = \sum_{K \in \calT_h} \sum_{q=1}^{N_{q,K}} p_k(x_q^K | \bpsi) \rho_h(x_q^K) w_q^K - \nu_k,
\end{equation}
where $p_k(x_q^K | \bpsi)$ is the plan density factor from \eqref{eq:final_plan_density}, evaluated at the quadrature point $x_q^K$ using the current potential vector $\bpsi$.

\subsection{Convergence and Optimization}
The minimization of the discretized functional $J_\varepsilon^h(\bpsi)$ yields an approximate potential vector $\bpsi^*_h$. The accuracy of this approximation depends on both the finite element error and the quadrature error.
The total error $|J_\varepsilon(\bpsi) - J_\varepsilon^h(\bpsi)|$ arises from these two sources. The quadrature error depends on the degree of polynomial precision of the quadrature rule and the smoothness of the integrand, which involves $\rho_h(x)$ and the $\log(\sum \dots)$ term. Typically, if the quadrature rule is chosen to integrate products of polynomials arising from $\rho_h$ and the basis functions exactly up to a sufficient degree, the overall error is dominated by the finite element approximation error.
Standard error analysis for FEM \cite{BrennerScottFEM} suggests that, under appropriate smoothness assumptions on $\rho$ and the domain $\OmegaDomain$, and for a sufficiently accurate quadrature rule, the error in the functional and its gradient converges to zero as $h \to 0$:
\begin{theorem}(Convergence of minimizers of dual regularized SOT)
\label{theo:conv_minimizers}
    Let the previous assumptions on the source $\mu$ and target $\nu$ measures be valid. We consider the dual functionals $J_{\varepsilon}:\mathbb{R}^N\rightarrow\mathbb{R}$ and $J^{h, k, r}_{\varepsilon}:\mathbb{R}^N\rightarrow\mathbb{R}$, for which the dependence on the discretization step $h$, polynomial order of discretization of $\rho_h$ with finite elements $k$, and order of the Gaussian quadrature rule $r$ are displayed. Let us additionally assume that $c:\Omega\times\Omega\rightarrow\mathbb{R}$ is sufficiently regular, such that the following integrands, for all $i\in\{1,\dots N\}$,
    \begin{equation}
    \label{eq:fandg}
        f(x)=\varepsilon\log\left(\sum_{j=1}^N \nu_j \exp\left(\psi_j - c(x, y_j)/\varepsilon\right)\right),\quad g_i(x) = \frac{\nu_k \exp((\psi_i - c(x, y_i))/\varepsilon)}{\sum_{j=1}^N \nu_j \exp((\psi_j - c(x, y_j))/\varepsilon)},
    \end{equation}
    are at least $\max(k+1, 2r-k)$ differentiable, $f,g_i\in\mathcal{C}^{\max(k+1, 2r-k)}$, $\forall i\in\{1,\dots N\}$. Then, the pointwise convergence of the functionals
    \begin{equation}
|J_\varepsilon(\bpsi) - J_\varepsilon^{h,k,r}(\bpsi)| \overset{h \to 0}{\rightarrow} 0, \qquad
\|\nabla J_\varepsilon(\bpsi) - \nabla J_\varepsilon^{h,k,r}(\bpsi)\|_2 \overset{h \to 0}{\rightarrow} 0,\quad\forall\bpsi\in\mathbb{R}^N,
\end{equation}
implies that the unique minimizer (up to constants) $\bpsi_{\varepsilon}^{h,k,r}$ of $J^{h,k,r}$ converges to the unique minimizer $\bpsi_{\varepsilon}$ of $J_{\varepsilon}$ in Euclidean norm,$$\bpsi^{h,k,r}_{\varepsilon}\overset{h \to 0}{\rightarrow}\bpsi_{\varepsilon}.$$
\end{theorem}
\begin{proof}
    The proof is reported in the Appendix~\ref{appendix:proof}.
\end{proof}
The auxiliary functions $f,g$ defined in Equation~\eqref{eq:fandg} are $\mathcal{C}^{\infty}$ for $c(x, y)=\tfrac{1}{p}\lVert x-y\rVert^p_p$ with $p\in\mathbb{N}$ even integer.
In practice, we solve the finite-dimensional convex optimization problem $\min_{\bpsi \in \R^N} J_\varepsilon^h(\bpsi)$ using an iterative quasi-Newton method, specifically the L-BFGS algorithm. The choice of a second-order-like method over simpler first-order algorithms (like Sinkhorn) is motivated by the potential for faster convergence, especially for small $\varepsilon$. L-BFGS offers a practical compromise, avoiding the need to compute and invert the full Hessian, which can be a significant challenge for Newton's method in this context \cite{kitagawa2019convergence,wu2025pins}. The efficiency of each L-BFGS iteration, however, depends critically on the fast evaluation of the functional and its gradient, making this the primary computational challenge that we will address in the following sections.

\begin{remark}[Choice of Stopping Criterion for Optimization]
\label{rem:stopping_criterion}
The L-BFGS algorithm iteratively minimizes the discretized dual functional $J_\varepsilon^h(\bpsi)$. The stopping criterion is typically based on the norm of the gradient $\nabla J_\varepsilon^h(\bpsi)$.

Recall the $k$-th component of the gradient (Eq. \eqref{eq:discretized_gradient}):
\[
\frac{\partial J_\varepsilon^h}{\partial \psi_k} = \left( \sum_{K \in \calT_h} \sum_{q=1}^{N_{q,K}} p_k(x_q^K | \bpsi) \rho_h(x_q^K) w_q^K \right) - \nu_k = \tilde{\nu}_k(\bpsi) - \nu_k.
\]
Here, $\tilde{\nu}_k(\bpsi)$ represents the total mass computed via quadrature that is assigned to the target point $y_k$ given the current potential vector $\bpsi$, while $\nu_k$ is the prescribed target mass. The gradient component $\frac{\partial J_\varepsilon^h}{\partial \psi_k}$ thus measures the violation of the $k$-th target marginal constraint in the discretized sense. One possible criterion involves the $L^\infty$ norm (maximum absolute value):
\[
\left\| \nabla J_\varepsilon^h(\bpsi) \right\|_\infty = \max_{k} |\tilde{\nu}_k(\bpsi) - \nu_k|.
\]
Terminating when $\left\| \nabla J_\varepsilon^h(\bpsi) \right\|_\infty \le \delta_{\text{tol},\infty}$ means stopping when the \emph{maximum absolute error} in satisfying any single target marginal constraint falls below the tolerance $\delta_{\text{tol},\infty}$. This criterion provides strong, uniform control over the absolute mass conservation error for each individual target point.

The criterion used in this work employs the $L^1$ norm of the gradient. We terminate the iterations when the $L^1$ norm falls below a predefined tolerance $\delta_{\text{tol},1}$:
\[
\left\| \nabla J_\varepsilon^h(\bpsi) \right\|_1 = \sum_{k=1}^N \left| \frac{\partial J_\varepsilon^h}{\partial \psi_k} \right| = \sum_{k=1}^N |\tilde{\nu}_k(\bpsi) - \nu_k| \le \delta_{\text{tol},1}.
\]
The $L^1$ stopping criterion measures the \emph{total absolute error} in satisfying the target marginal constraints across all target points combined. This provides a robust and global measure of convergence.
\end{remark} %
\section{Accelerating the Dual Evaluation}
\label{sec:acceleration}

A direct evaluation of the discretized functional $J_\varepsilon^h(\bpsi)$ \eqref{eq:discretized_functional} and gradient $\nabla J_\varepsilon^h(\bpsi)$ \eqref{eq:discretized_gradient} presents a significant computational bottleneck. For each of the $N_q$ quadrature points $x_q^K$, computing the sums $S(x_q^K, \bpsi) = \sum_{j=1}^N \nu_j \exp\left(\frac{\psi_j - c(x_q^K, y_j)}{\varepsilon}\right)$ requires $O(N)$ operations (evaluating the exponential term and summing). The total cost per evaluation is therefore $O(N_q N)$, which becomes prohibitive for large numbers of target points $N$ or high-accuracy quadrature (large $N_q$).

To overcome this challenge, we employ acceleration techniques that exploit the structure of the problem, particularly the rapid decay of the Gibbs kernel $K_\varepsilon(x, y) = \exp\left(-\frac{c(x,y)}{\varepsilon}\right)$ for points $x, y$ that are far apart in terms of the cost function $c(\cdot, \cdot)$ relative to $\varepsilon$. This kernel exhibits exponential decay as $c(x,y)$ increases, with the rate of decay controlled by the regularization parameter $\varepsilon$.

\subsection{Adaptive Truncation Strategies}
\label{subsubsec:truncation_revised}

The core idea is to approximate the full sum $S(x, \bpsi)$ by truncating it, including only contributions from target points $y_j$ that are sufficiently close to the source point $x$ in terms of the cost function $c(x, y_j)$. We neglect terms corresponding to distant $y_j$ whose contribution $\nu_j \exp\left(\frac{\psi_j - c(x, y_j)}{\varepsilon}\right)$ is below a certain threshold.

One simple approach defines a pointwise contribution threshold $\delta_{\text{thr}} > 0$. A term is neglected if:
\begin{equation} \label{eq:pointwise_criterion_revised}
\nu_j \exp\left(\frac{\psi_j - c(x, y_j)}{\varepsilon}\right) < \delta_{\text{thr}}.
\end{equation}
This condition can be used to define a truncation threshold. To facilitate fast spatial searching, we can use a single cost cutoff $C_{\text{pw}}$. Letting $M = \max_k \psi_k$ and $\overline{\nu} = \max_k \nu_k$, we define $C_{\text{pw}}$ such that $c(x, y_j) \ge C_{\text{pw}}$ guarantees condition \eqref{eq:pointwise_criterion_revised} holds:
\begin{equation} \label{eq:cost_cutoff_delta_revised}
C_{\text{pw}}(\bpsi, \delta_{\text{thr}}) = M + \varepsilon \ln\left( \frac{\overline{\nu}}{\delta_{\text{thr}}} \right).
\end{equation}
However, choosing an appropriate value for the hyperparameter $\delta_{\text{thr}}$ requires careful tuning, as its direct relationship to the overall accuracy of the computed functional $J_\varepsilon^h(\bpsi)$ is not immediately obvious and may depend strongly on problem parameters ($\varepsilon$, distribution of $\nu_j$, range of $\bpsi$).

To achieve a more predictable and theoretically grounded control over the approximation error, we can determine the truncation cost cutoff $C$ to ensure that the \emph{global relative error} introduced in the functional $J_\varepsilon(\bpsi)$ remains below a user-defined tolerance $\tau > 0$. This avoids the ambiguity of choosing $\delta_{\text{thr}}$ and directly controls the error in the quantity being minimized. As derived in Appendix~\ref{app:truncation_derivation}, the condition $\|J_\varepsilon(\bpsi) - \widetilde{J}_\varepsilon(\bpsi)\| \le \tau \|J_\varepsilon(\bpsi)\|$ can be guaranteed by choosing $C$ sufficiently large. Two main bounds for the required cost cutoff can be derived.

First, the \emph{integrated bound} provides a cost cutoff based on the truncated sum structure. The required truncation threshold $C_{\text{int}}$ must satisfy:
\begin{equation}
\label{eq:truncation_cutoff_int}
C_{\text{int}}(\bpsi, \varepsilon, \tau, D) \ge M + \varepsilon \ln\left( \frac{\varepsilon D(\bpsi, C)}{\tau |J_\varepsilon(\bpsi)|} \right),
\end{equation}
where $M = \max_k \psi_k$ represents the maximum potential, $D(\bpsi, C)$ is the integrated inverse of the truncated sum, and $S_{trunc}(x, \bpsi; C)$ denotes the truncated sum itself:
\begin{equation}
\begin{aligned}
S_{trunc}(x, \bpsi; C) &= \sum_{j: c(x, y_j) < C} \nu_j \exp\left(\frac{\psi_j - c(x, y_j)}{\varepsilon}\right), \\
D(\bpsi, C) &= \int_{\OmegaDomain} [S_{trunc}(x, \bpsi; C)]^{-1} d\mu(x).
\end{aligned}
\end{equation}
Note that equation \eqref{eq:truncation_cutoff_int} defines an implicit nonlinear equation for $C_{\text{int}}$, since $D(\bpsi, C)$ itself depends on $C$ through the truncated sum $S_{trunc}(x, \bpsi; C)$. In practice, $D(\bpsi, C)$ can be estimated during the optimization process by using its value computed with the cost cutoff from the previous iteration.

Second, by employing a worst-case estimate for the truncated sum, we obtain the \emph{geometric bound} $C_{\text{geom}}$, which does not require estimating $D$. Let $C_0 = \max_{x \in \OmegaDomain} \min_{j} c(x, y_j)$ be the covering cost, the minimum potential $m = \min_k \psi_k$, and the minimum weight $\underline{\nu} = \min_j \nu_j$. Assuming normalized measures ($\int_\OmegaDomain d\mu = 1$, $\sum \nu_j = 1$), this truncation threshold must satisfy (see Appendix~\ref{app:truncation_derivation}):
\begin{equation}
\label{eq:truncation_cutoff_geom_revised}
C_{\text{geom}}(\bpsi, \varepsilon, \tau) \ge C_0 + \Gamma(\bpsi) + \varepsilon \ln\left( \frac{\varepsilon}{\underline{\nu} \tau |J_\varepsilon(\bpsi)|} \right),
\end{equation}
where $\Gamma(\bpsi) = M-m$ is the range of the potentials. This truncation bound $C_{\text{geom}}$ adapts during the optimization as $\bpsi$ (and thus $M$, $m$, $\Gamma$, and $J_\varepsilon(\bpsi)$) evolves. It provides a robust way to control the global relative error without needing runtime estimates of integrals.

To efficiently find the set of target points $I_C(x) = \{j : c(x, y_j) < C_{\text{trunc}}\}$ for each quadrature point $x=x_q^K$, we can use a spatial data structure. If the cost $c(x,y)$ is a metric (e.g., Euclidean distance), spatial trees like R-trees \cite{Guttman1984} allow range queries to be performed much faster than a linear scan. When using spatial trees with metric costs, the query time is typically $O(\log N + N_{eff}(x))$, where $N_{eff}(x) = |I_C(x)|$ is the number of points found. When $N_{eff}(x)$ is small, this significantly reduces the cost compared to $O(N)$.

\begin{algorithm}[ht!]
    \caption{Accelerated Iterative Solver for the RSOT Dual Problem}
    \label{alg:accelerated_rsot_solver}
    \begin{algorithmic}[1]
    \Require Source mesh $\calT_h$, density $\rho_h$, quadrature $\{x_q, w_q\}_{q=1}^{N_q}$, target $\{y_j, \nu_j\}_{j=1}^N$, cost function $c(\cdot, \cdot)$, $\varepsilon > 0$, tolerance $\delta_{\text{tol}}$.
    \Ensure Optimal potential vector $\bpsi^* \in \R^N$.

    \Repeat \ While $\|\nabla J_{\text{current}}\|_1 > \delta_{\text{tol}}$.

        \State Perform L-BFGS step that include a line search procedure requiring multiple functional $J_{\text{current}}$ and gradient $\nabla J_{\text{current}}$ evaluations at current potential $\bpsi^{(k)}$ via Algorithm~\ref{alg:gradeval}: $$\bpsi^{(k+1)} \leftarrow \text{L-BFGSstep}(\bpsi^{(k)}, \nabla J_{\text{current}}, J_{\text{current}}).$$
        \State $k \leftarrow k + 1$.
    \Until{convergence}
    \end{algorithmic}
\end{algorithm}

\begin{algorithm}[ht!]
\caption{Truncated RSOT gradient and functional evaluation}
    \label{alg:gradeval}
    \begin{algorithmic}[1]
        \Require Source mesh $\calT_h$, density $\rho_h$, quadrature $\{x_q, w_q\}_{q=1}^{N_q}$, target $\{y_j, \nu_j\}_{j=1}^N$, cost function $c(\cdot, \cdot)$, $\varepsilon > 0$, potential $\bpsi$.
        \Ensure Functional $J_{\text{current}}$ and gradient $\nabla J_{\text{current}}$ evaluations.
        \ForAll{quadrature points $x_q$ with weight $w_q$}
            \State $I_q \leftarrow \{j : c(x_q, y_j) < C_{\text{trunc}}\}$ via spatial index query.
            \State $S_q \leftarrow \sum_{j \in I_q} \nu_j \exp((\psi_j - c(x_q, y_j))/\varepsilon)$.

            \If{$S_q > 0$}
                \State $J_{\text{current}} \leftarrow J_{\text{current}} + \varepsilon \ln(S_q) \rho_h(x_q) w_q$.
                \ForAll{$j \in I_q$}
                     \State $T_{qj} \leftarrow \nu_j \exp((\psi_j - c(x_q, y_j))/\varepsilon)$.
                     \State $\nabla J_{\text{current}}[j] \leftarrow \nabla J_{\text{current}}[j] + \frac{T_{qj}}{S_q} \rho_h(x_q) w_q$.
                \EndFor
            \EndIf
        \EndFor

        \State $J_{\text{current}} \leftarrow J_{\text{current}} - \sum_{j=1}^N \psi_j \nu_j$.
        \State $\nabla J_{\text{current}} \leftarrow \nabla J_{\text{current}} - \bnu$.
    \end{algorithmic}
\end{algorithm}
\begin{remark}[Parallelization Strategy]
\label{rem:parallelization}
We employ a two-level parallelization scheme for efficient computation. Firstly, using domain decomposition via MPI, the source domain $\OmegaDomain$ and its corresponding mesh $\calT_h$ are partitioned among multiple MPI processes. Each process is responsible for computing the contributions to $J_\varepsilon^h$ and $\nabla J_\varepsilon^h$ arising from its assigned cells. Secondly, within each MPI process, shared-memory parallelism via OpenMP is used to parallelize the loop over the locally owned cells and their quadrature points. To facilitate this, the R-tree index and the target point data ($\{y_j, \nu_j\}$), as well as the potential vector $\bpsi$, are replicated on each MPI process, allowing concurrent, communication-free spatial queries by the OpenMP threads. Partial gradient contributions computed by the threads are accumulated locally within each MPI process. The final global gradient and functional value are then obtained via MPI reduction operations (summation) across all processes.
\end{remark}

\subsection{Regularization Parameter Scheduling}
\label{sec:epsilon_scheduling}

The choice of the regularization parameter $\varepsilon$ significantly influences both the smoothness of the resulting transport plan and the numerical conditioning of the optimization problem \eqref{eq:dual_functional_to_minimize_redux}. A large value of $\varepsilon$ leads to a smoother objective function that is generally easier to minimize but yields a solution further from the original unregularized OT problem. Conversely, a small $\varepsilon$ allows the solution to better approximate the unregularized problem but can introduce numerical instability and potentially slow down the convergence of optimization algorithms like BFGS.

To balance these competing effects, a common and effective strategy, often referred to as $\varepsilon$-scaling or annealing, involves starting the optimization process with a relatively large initial value $\varepsilon_0$ and gradually decreasing it through a sequence $\varepsilon_0 > \varepsilon_1 > \dots > \varepsilon_M = \varepsilon_{final}$ until the desired target value $\varepsilon_{final}$ is reached. The core idea of this procedure is to solve the regularized semi-discrete OT problem for each $\varepsilon_m$ in the sequence, using the optimal potential $\bpsi^*(\varepsilon_m)$ obtained at one stage as a warm start (initial guess) for the optimization problem at the next stage with $\varepsilon_{m+1}$. This continuation approach often improves the overall robustness of the solution process and can accelerate convergence towards the final solution corresponding to $\varepsilon_{final}$, primarily because the optimizer begins each subsequent stage closer to the minimum.
\section{Multilevel Acceleration Strategies}
\label{sec:multilevel}

Solving the RSOT problem directly on a fine discretization can be slow, especially if the initial guess for $\bpsi$ is poor. Multigrid and multilevel methods are a powerful paradigm for accelerating solvers in numerical analysis. The core idea is to address this by starting with a coarser version of the problem and progressively refining it, using the solution from the coarse level to initialize the fine level. This paradigm was successfully applied to unregularized SOT by Mérigot \cite{merigot2011multiscale}, who introduced a coarse-to-fine hierarchy for the \textit{target measure}, using Lloyd's algorithm to generate the coarser levels. In this work, we adapt and extend this concept in two ways for the regularized FE setting: by building hierarchies for the \textit{source measure} as well, and by developing a combined strategy that refines both simultaneously.
\subsection{Source Measure Coarsening}
\label{sec:source_coarsening}

To mitigate the computational expense associated with solving the RSOT problem directly on the fine target mesh $\calT_L$, we employ a multilevel strategy based on a hierarchy of source domain discretizations. This approach constructs a sequence of meshes $\calT_0, \calT_1, \dots, \calT_L$, where $\calT_0$ is the coarsest mesh and $\calT_L$ is the finest. The core idea is to leverage solutions computed on coarser, computationally less demanding levels to provide effective initial guesses for the optimization on finer levels, thereby accelerating overall convergence.

The generation of this mesh hierarchy proceeds from fine to coarse. Starting with the finest mesh $\calT_L$, coarser meshes $\calT_l$ (for $l = L-1, \dots, 0$) are created iteratively using surface-based mesh simplification techniques, as implemented in the Geogram library \cite{Levy2012Geogram}. Specifically, we employ algorithms founded on Centroidal Voronoi Tesselation (CVT) principles \cite{Du1999CVT, Liu2009CVTComputation} to optimize point placement on the mesh surface. A key aspect of our hierarchy generation is controlling the complexity at each level by specifying the desired number of \textit{surface vertices}. For each level $l < L$, we calculate a target number of surface vertices, $N_{\text{surf}}^{(l)}$, often decreasing exponentially from the number of surface vertices in $\calT_{l+1}$ (or from a defined maximum) down to a specified minimum for $\calT_0$. The surface remeshing algorithm is then invoked, targeting $N_{\text{surf}}^{(l)}$ points on the surface derived from the mesh of level $l+1$. After remeshing the surface, the interior volume is re-tetrahedralized using volumetric meshing as implemented in Tetgen \cite{Si2015Tetgen}.

A crucial aspect of this multilevel strategy is the consistent representation of the source density $\rho(x)$ across all levels of the hierarchy. The density is initially represented by a FE function $\rho_h^{(L)}$ on the finest mesh $\calT_L$. However, the meshes $\calT_l$ generated by the remeshing process are generally \emph{non-nested} (i.e., a coarse cell is not simply a union of fine cells), preventing direct application of standard FE interpolation (like nodal or $L^2$ projection) between levels. To address this, we define the density $\rho_h^{(l)}$ on each coarser mesh $\calT_l$ using a custom interpolation approach. For each degree of freedom on $\calT_l$, we identify the nearest cell in the \textit{finest} mesh $\calT_L$ and evaluate the original high-resolution density function $\rho_h^{(L)}$ within that cell at the corresponding location. This nearest-neighbor search can be efficiently performed using spatial indexing structures like R-trees built on the cell centers of $\calT_L$. The projection ensures that the density representations on coarser levels remain consistent with the fine-scale information, despite the non-nested nature of the mesh hierarchy.

The multilevel solution process then proceeds from coarsest to finest. First, the RSOT problem (e.g., minimize \eqref{eq:dual_functional_to_minimize_redux}) is solved on the coarsest mesh $\calT_0$ using the density $\rho_h^{(0)}$ to obtain an optimal potential vector $\bpsi^{(0)}$. Subsequently, for each level $l$ from $0$ to $L-1$, the computed potential $\bpsi^{(l)}$ serves as the initial guess for the optimization problem on the next finer mesh $\calT_{l+1}$, which utilizes the density $\rho_h^{(l+1)}$. This iterative process continues until the final potential $\bpsi^{(L)}$ is obtained on the target fine mesh $\calT_L$. This coarse-to-fine strategy typically requires significantly fewer iterations on the finer, more expensive levels compared to solving directly on $\calT_L$ from a default initial guess.

\begin{remark}[Generality of the Source Hierarchy]
It is worth noting that this fine-to-coarse generation, which assumes a single high-resolution source representation is available, is only one possible approach. The multilevel framework is general and only requires a sequence of meshes with decreasing complexity. If the problem setting naturally provides a hierarchy of source geometries constructed through progressive refinement (i.e., from coarse to fine), such as in adaptive mesh refinement or shape generation, this sequence can be directly employed as the levels $\{\calT_l\}_{l=0}^L$.
\end{remark}

\subsection{Target Set Coarsening}
\label{sec:target_coarsening}

Alternatively, or in conjunction with source mesh coarsening, acceleration can be achieved by constructing a hierarchy of the discrete target measure $\nu = \sum_{j=1}^N \nu_j \delta_{y_j}$. This strategy focuses on reducing the number of target points $N$ involved in the computationally intensive dual objective evaluation \eqref{eq:dual_functional_to_minimize_redux} at coarser stages of the solution process.

The hierarchy generation involves creating a sequence of discrete measures $\nu^{(0)}, \nu^{(1)}, \dots, \nu^{(L)}$, where $\nu^{(0)}$ is the coarsest measure and $\nu^{(L)} = \nu$ is the original (finest) target measure. The coarser approximations $\nu^{(l)}$ are supported on fewer points, $N_l < N_{l+1}$. This hierarchy is constructed iteratively from the finest level $L$ down to the coarsest level $0$. At each step, from level $l$ to $l-1$, a k-means clustering is applied to the support points $Y_l = \{y_j^{(l)}\}_{j=1}^{N_l}$. Each cluster of points from level $l$ defines a single point $y_k^{(l-1)}$ at the coarser level $l-1$, chosen as the cluster's centroid. Crucially, the measure associated with a coarse point, $\nu_k^{(l-1)}$, is defined as the sum of the measures of the fine points belonging to its cluster, ensuring mass conservation across levels: $\nu_k^{(l-1)} = \sum_{y_j^{(l)} \in \text{cluster}_k} \nu_j^{(l)}$.

A critical component of this multilevel strategy is the transfer of potential information from a coarse level $l$ to the next finer level $l+1$. While simple injection (assigning a child point the potential of its parent) is possible, it often provides a poor initial guess. We employ a more sophisticated approach motivated by the structure of the dual problem and its solution via Sinkhorn-like iterations. Recall that the optimal dual potentials $(\phi^*, \bpsi^*)$ are related through the expression for $\exp{(\phi^*(x)/\varepsilon)}$ in \eqref{eq:u_from_v_expanded}. Given the optimal potential $\bpsi^{(l)}$ from the coarse level $l$, we can define the corresponding optimal continuous potential $\phi^{(l)}(x)$ via:
\begin{equation}
\label{eq:phi_from_psi_coarse}
\exp{(\frac{\phi^{(l)}(x)}{\varepsilon})} = \frac{1}{\sum_{j=1}^{N_l} \nu_j^{(l)} \exp\left( (\psi_j^{(l)} - c(x, y_j^{(l)})) / \varepsilon \right)}.
\end{equation}
This $\phi^{(l)}(x)$ represents the optimal dual field adjusted to the coarse target measure $\nu^{(l)}$. The softmax refinement scheme uses this $\phi^{(l)}(x)$ to compute an initial guess $\psi_k^{(l+1), \text{init}}$ for the potential at a fine point $y_k^{(l+1)}$ on level $l+1$. The formula is derived by considering the relationship that must hold at optimality for the target marginals and is given by:

\begin{equation}
\label{eq:softmax_refinement_formula}
\psi_k^{(l+1), \text{init}} = -\varepsilon \log \left( \int_{\OmegaDomain} \frac{e^{-c(x, y_k^{(l+1)})/\varepsilon}}{\sum_{j=1}^{N_l} \nu_j^{(l)} \exp\left( (\psi_j^{(l)} - c(x, y_j^{(l)})) / \varepsilon \right)} \dd\mu(x) \right).
\end{equation}

Intuitively, this computes the (negative) $\varepsilon$-smoothed $c$-transform of the potential $\phi^{(l)}$ with respect to the cost $c(x, y_k^{(l+1)})$, averaged over the source measure $\mu$. This operation is analogous to performing one half-step of a Sinkhorn-like iteration: starting with the pair $(\phi^{(l)}, \bpsi^{(l)})$ which satisfies the source marginal constraint approximately for the coarse problem, we update the discrete potential to $\bpsi^{(l+1), \text{init}}$ to better approximate the target marginal constraint for the fine measure $\nu^{(l+1)}$ using the existing continuous potential $\phi^{(l)}$. This provides a significantly more informed initialization for the optimization at level $l+1$ compared to simple injection.

The overall multilevel algorithm for target set coarsening proceeds as follows: First, the RSOT problem is solved for the coarsest target measure $\nu^{(0)}$, yielding the potential vector $\bpsi^{(0)}$. Then, for $l = 0$ to $L-1$, the potential vector $\bpsi^{(l)}$ is used to compute $\phi^{(l)}(x)$ via \eqref{eq:phi_from_psi_coarse}, and subsequently an initial guess $\bpsi^{(l+1), \text{init}}$ is generated for the finer level $l+1$ via the softmax-based refinement \eqref{eq:softmax_refinement_formula}. The RSOT problem is then solved for the target measure $\nu^{(l+1)}$ starting from $\bpsi^{(l+1), \text{init}}$, resulting in the optimal potential $\bpsi^{(l+1)}$. This process continues until the solution $\bpsi^{(L)}$ on the original fine target measure $\nu$ is obtained. Similar to source mesh coarsening, this provides significant acceleration by reducing the computational effort required on the levels with a large number of target points.

\subsection{Combined Multilevel Strategy}
\label{subsec:combined_multilevel}

For problems characterized by both complex source domain geometry, addressed by the source mesh hierarchy described in Section~\ref{sec:source_coarsening}, and a large number of target points, handled by the target set hierarchy of Section~\ref{sec:target_coarsening}, we integrate these into a unified multilevel approach. This strategy simultaneously manages the progression through both hierarchies within a single iterative framework, designed to gracefully handle potentially different numbers of levels in the source mesh and target measure hierarchies (see Figure~\ref{fig:combined_multilevel_scheme}).

\begin{figure}[htbp]
\centering
\includegraphics[width=0.9\textwidth]{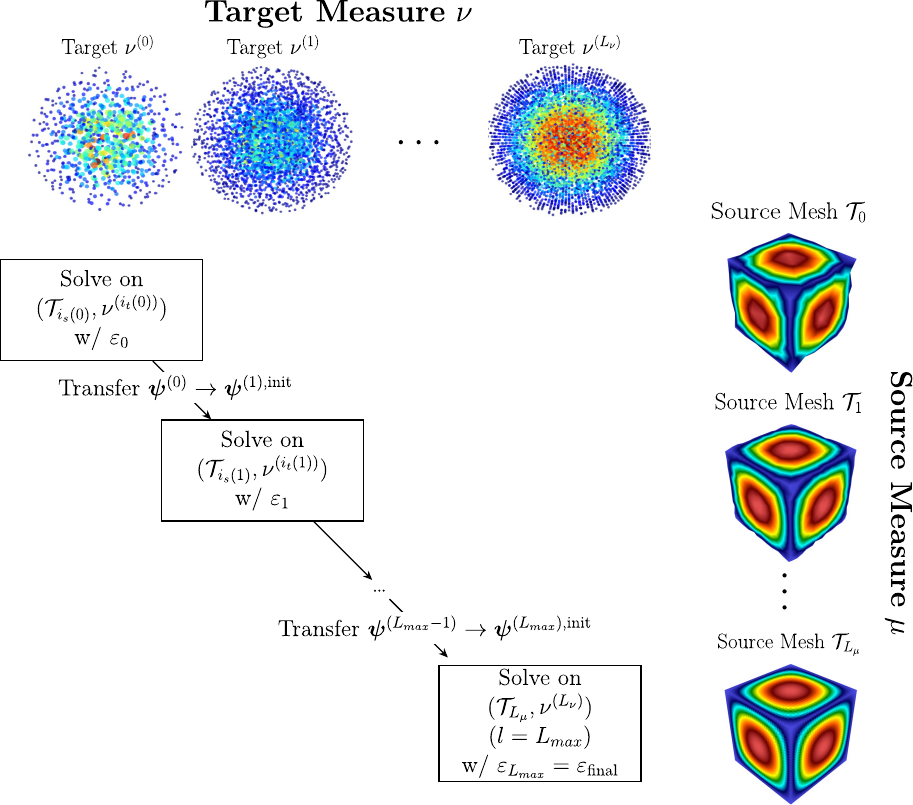}
\caption{Illustration of the unified combined multilevel strategy. The process iterates sequentially from $l=0$ to $l=L_{max}$. At each iteration $l$, the source mesh $\mathcal{T}_{i_s(l)}$ and target measure $\nu^{(i_t(l))}$ are selected based on the index logic. The coarser hierarchy remains fixed until the main iteration index $l$ allows it to start refining. Potential transfer uses $\bpsi^{(l-1)}$ to initialize $\bpsi^{(l)}$, employing either softmax refinement or injection based on whether the target level index $i_t$ increased.}
\label{fig:combined_multilevel_scheme}
\end{figure}

Let the source mesh hierarchy be $\{\mathcal{T}_l\}_{l=0}^{L_\mu}$ and the target measure hierarchy be $\{\nu^{(l)}\}_{l=0}^{L_\nu}$, as defined previously. We define the total number of refinement steps in the combined strategy as $L_{max} = \max(L_\mu, L_\nu)$.

The combined multilevel process iterates through a main loop indexed by $l$ from $0$ (coarsest overall step) to $L_{max}$ (finest overall step). At each iteration $l$, the specific source mesh and target measure used are determined by their respective level indices, $i_s$ and $i_t$. The source level index is $i_s = \max(0, l - (L_{max} - L_\mu))$, and the target level index is $i_t = \max(0, l - (L_{max} - L_\nu))$. This indexing scheme ensures that the hierarchy with fewer levels (e.g., source, if $L_\mu < L_\nu$) remains fixed at its coarsest representation ($\mathcal{T}_0$) for the initial $L_\nu - L_\mu$ iterations. After this initial phase, both the source level index $i_s$ and the target level index $i_t$ increment together at each subsequent iteration $l$, until the hierarchy with more levels reaches its finest representation at $l = L_{max}$. The optimization problem (minimizing the dual functional $J_\varepsilon(\bpsi)$, Eq.~\eqref{eq:dual_functional_to_minimize_redux}) is solved at each iteration $l$ using the selected pair $(\mathcal{T}_{i_s}, \nu^{(i_t)})$.

Crucial to the efficiency of this approach is the transfer of potential information between iterations. Let $\bpsi^{(l-1)}$ be the optimal potential vector obtained at iteration $l-1$ (corresponding to source mesh $\mathcal{T}_{i_s(l-1)}$ and target measure $\nu^{(i_t(l-1))}$). To obtain an initial guess $\bpsi^{(l), \text{init}}$ for the optimization at iteration $l$ (on $\mathcal{T}_{i_s(l)}$ with $\nu^{(i_t(l))}$), we adopt a strategy based on how the levels change.

If the target measure level index increased sequentially from the previous iteration (i.e., $i_t(l) = i_t(l-1) + 1$), we use the softmax refinement scheme described in Section~\ref{sec:target_coarsening} to compute the initial guess. The initial potential for a point $y_k^{(i_t(l))}$ at the finer target level $i_t(l)$ is computed based on the potential $\bpsi^{(l-1)}$ from the coarser target level $i_t(l-1)$, using the formula:
\begin{equation}
\label{eq:softmax_refinement_formula_combined}
\psi_k^{(l), \text{init}} = -\varepsilon \log \left( \int_{\OmegaDomain} \frac{e^{-c(x, y_k^{(i_t(l))})/\varepsilon}}{\sum_{j=1}^{N_{i_t(l-1)}} \nu_j^{(i_t(l-1))} \exp\left( (\psi_j^{(l-1)} - c(x, y_j^{(i_t(l-1))})) / \varepsilon \right)} \dd\mu^{(i_s(l))}(x) \right).
\end{equation}
Note that the integral in \eqref{eq:softmax_refinement_formula_combined} is evaluated numerically using the source measure $\dd\mu^{(i_s(l))}(x) = \rho_h^{(i_s(l))}(x) \dd x$ defined on the \emph{current} source mesh $\mathcal{T}_{i_s(l)}$, while the summation in the denominator uses the potentials $\bpsi^{(l-1)}$ and points $\{y_j^{(i_t(l-1))}\}$ from the \emph{previous} (coarser) target level $i_t(l-1)$.

If the target level index did \emph{not} increase (i.e., $i_t(l) = i_t(l-1)$), which occurs when only the source mesh refines or when the target hierarchy remains at its coarsest level during the initial phase, we use simple injection, as the dimension of the potential vector does not change. The potential vector $\bpsi^{(l-1)}$ is thus used directly as the initial guess: $\bpsi^{(l), \text{init}} = \bpsi^{(l-1)}$.

This combined strategy naturally integrates with regularization parameter scheduling (discussed in Section \ref{sec:epsilon_scheduling}). The sequence of $\varepsilon$ values can be synchronized with the main iteration loop $l=0, \dots, L_{max}$, starting with a larger $\varepsilon_0$ at the coarsest level ($l=0$) and decreasing towards the final target value $\varepsilon_{final}$ as $l$ approaches $L_{max}$ (see Figure \ref{fig:combined_multilevel_scheme}).

This unified approach avoids distinct phases, allowing both source and target complexities to be addressed synergistically throughout the computation. It provides a flexible framework for handling hierarchies of different depths, ensuring that computations on coarser representations inform the solution process on progressively finer levels. %
\section{Numerical Experiments}
\label{sec:numerical_experiments}

In this section, we present numerical results to demonstrate the capabilities of our computational framework. The core of our work is encapsulated in the \texttt{SemiDiscreteOT} C++ library, an open-source implementation of the methods discussed, built upon the \texttt{deal.II} finite element library~\cite{DEALII} and available at \url{https://github.com/SemiDiscreteOT/SemiDiscreteOT}. This library integrates our acceleration techniques, multilevel solvers, and a hybrid MPI+OpenMP parallelization strategy for execution on HPC clusters. These techniques are also amenable to GPU acceleration, preliminary studies in this regard are reported in the Appendix~\ref{appendix:gpu}. For specific tasks, we also rely on external libraries such as Geogram~\cite{Levy2012Geogram} for mesh operations and Faiss~\cite{Faiss} for k-means clustering. All simulations were performed on the Marconi A3 partition (CINECA), featuring nodes with dual 24-core Intel Xeon 8160 CPUs and 192 GB RAM, interconnected via Intel OmniPath.

We present three distinct computational studies to validate the robustness, efficiency, and versatility of our framework:
\begin{enumerate}
    \item \textbf{Benchmarking and Scalability:} A thorough HPC benchmarking analysis to assess the performance, accuracy, and scalability of our acceleration and multilevel strategies using a non-trivial source density derived from a PDE solution.
    \item \textbf{Wasserstein Barycenter Computation:} An application to compute the Wasserstein barycenter of complex 3D vascular geometries, which showcases the use of our solver within a nested, iterative Lloyd-type algorithm.
    \item \textbf{Geometric Registration:} An application using the same vascular geometries to demonstrate pairwise registration, showing the recovery and analysis of a transport map between complex shapes.
    \item \textbf{Blue noise sampling} An application for blue noise sampling via the Lloyd algorithm with non-uniform measures and Riemannian barycenters. A classical problem that naturally requires the semi-discrete formulation of optimal transport and highlights the framework's ability to handle non-Euclidean costs.
\end{enumerate}
These experiments collectively validate the developed \texttt{SemiDiscreteOT} library and the underlying numerical methods.

\subsection{Benchmarking and Scalability Analysis}
\label{subsec:benchmarking}

The primary goal of this experiment is to evaluate the computational performance of the \texttt{SemiDiscreteOT} framework. We investigate the efficiency gains from the proposed acceleration techniques, the convergence improvements offered by multilevel strategies, the impact of regularization parameter scheduling, and the parallel scalability of the implementation.

The benchmark transport problem involves measures derived from Darcy flow simulations in $d=3$ dimensions, defined on two distinct domains and computed using related PDE systems. The source domain is the cube $\Omega_1 = [-1, 1]^3$, discretized by a tetrahedral mesh $\mathcal{T}_h^1$. The target domain is the unit ball $\Omega_2 = \{ \mathbf{x} \in \mathbb{R}^3 \mid \|\mathbf{x}\| \le 1 \}$, discretized by a tetrahedral mesh $\mathcal{T}_h^2$. On each domain $\Omega_i$ ($i=1,2$), we first solve the Darcy flow equations to find velocity $\mathbf{u}_i$ and pressure $p_i$:
\begin{align*}
    \left\{
    \begin{aligned}
        \mathbf{u}_i + \nabla p_i &= \mathbf{0} \quad \text{in } \Omega_i, \\
        \nabla \cdot \mathbf{u}_i &= f_i \quad \text{in } \Omega_i, \\
        \mathbf{u}_i \cdot \mathbf{n} &= 0 \quad \text{on } \partial\Omega_i,
    \end{aligned}
    \right.
    \label{eq:darcy_bc}
\end{align*}
where $\mathbf{n}$ is the outward unit normal, and the source terms are:
\begin{align*}
    f_1(x,y,z) &= \exp\left(-\frac{1}{1 - \max(x^2, y^2, z^2)}\right) \quad \text{for } \Omega_1, \\
    f_2(x,y,z) &= \exp\left(-\frac{1}{1 - (x^2+y^2+z^2)}\right) \quad \text{for } \Omega_2.
\end{align*}
These systems are discretized using continuous piecewise linear ($P_1$) elements for velocity and discontinuous piecewise constant ($P_0$) elements for pressure.

The \textit{source measure} $\mu$ is derived from the pressure solution $p_{1,h}$ on the cube $\Omega_1$. The piecewise constant pressure field $p_{1,h}$ is interpolated into the space of continuous $P_1$ functions on $\mathcal{T}_h^1$, resulting in the field $\tilde{p}_{1,h}(x)$. This $P_1$ pressure field is then $L^1$-normalized over the domain to get the source density: $\rho_h^1(x) = \tilde{p}_{1,h}(x) / \int_{\Omega_1} \tilde{p}_{1,h}(x') dx'$. The source measure is thus defined as $\dd\mu(x) = \rho_h^1(x) \dd x$.

The \textit{target measure} $\nu$ is constructed from the pressure solution $p_{2,h}$ on the ball $\Omega_2$. Similarly, $p_{2,h}$ is interpolated to yield a continuous $P_1$ field $\tilde{p}_{2,h}(x)$. The target measure $\nu$ is discrete and supported on the set $Y = \{y_j\}_{j=1}^N$, which consists of the vertices of the target mesh $\mathcal{T}_h^2$. The weight $\nu_j$ associated with vertex $y_j$ is determined by evaluating the interpolated $P_1$ pressure field $\tilde{p}_{2,h}$ at that vertex, followed by normalization across all vertices:
\begin{equation*}
    \nu_j = \frac{\tilde{p}_{2,h}(y_j)}{\sum_{k=1}^N \tilde{p}_{2,h}(y_k)}.
\end{equation*}
This procedure defines the discrete target measure $\nu = \sum_{j=1}^N \nu_j \delta_{y_j}$, where $N$ is the number of vertices in $\mathcal{T}_h^2$. Note that this construction ensures $\sum \nu_j = 1$. The resulting source density and target measure are visualized in Figure~\ref{fig:benchmark_densities}.

For all benchmark experiments, we employ the Euclidean cost function $c(x,y) = \frac{\|x - y\|^2}{2}$, which is a natural choice for transport problems in geometric settings and allows for efficient spatial indexing using R-trees for the truncation strategies described in Section~\ref{sec:acceleration}. The Euclidean cost also ensures that the truncation cutoffs have a direct geometric interpretation as spatial neighborhoods around each quadrature point.

\begin{figure}[htbp]
    \centering
    \begin{subfigure}[b]{0.4\textwidth}
        \centering
        \includegraphics[width=\textwidth, clip, trim=500 0 30 100]{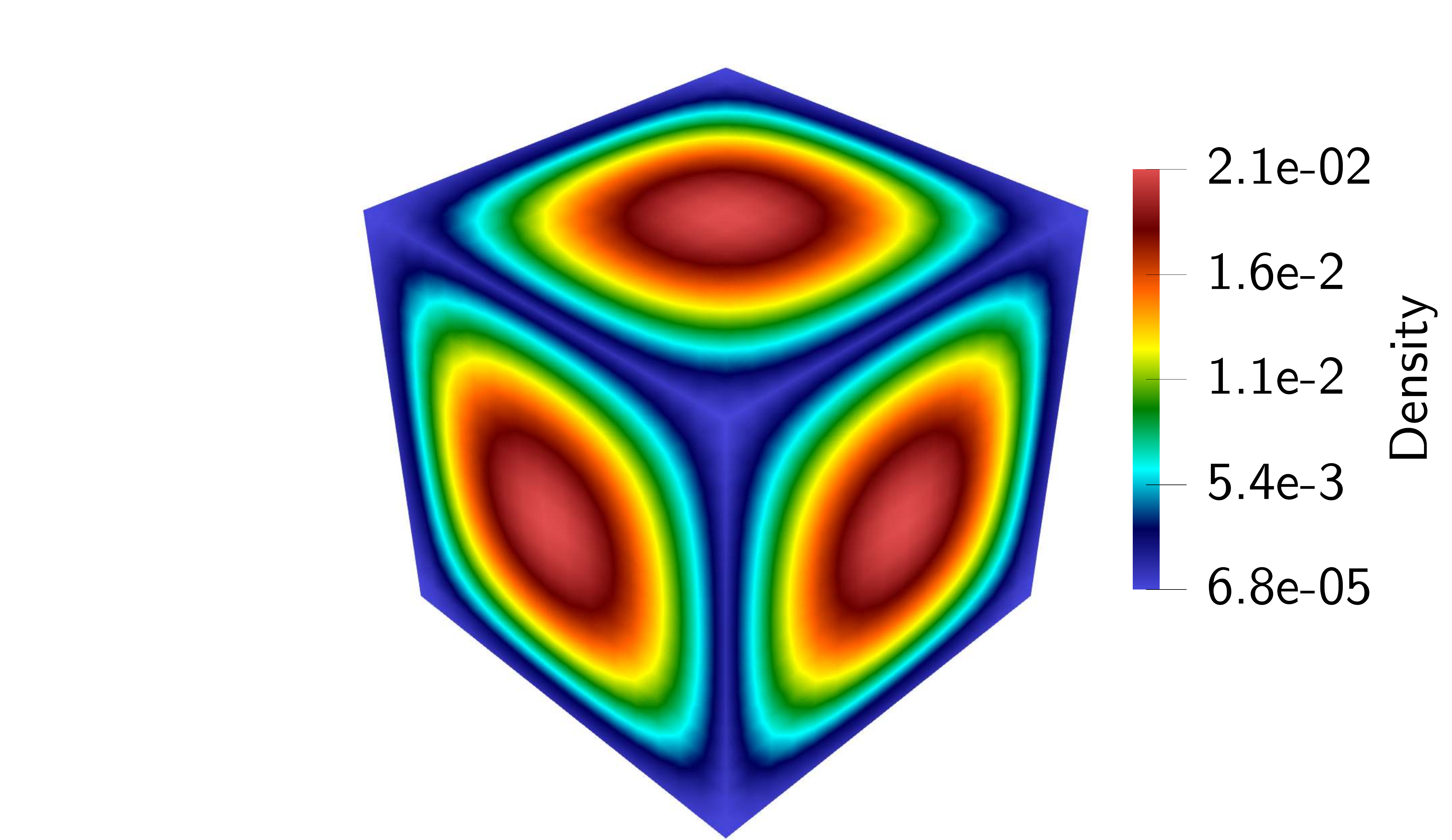}
        \caption{Source Measure $\mu$ (3D View)}
        \label{fig:bench_source_3d}
    \end{subfigure}
    \hfill %
    \begin{subfigure}[b]{0.4\textwidth}
        \centering
        \includegraphics[width=\textwidth, clip, trim=500 0 30 100]{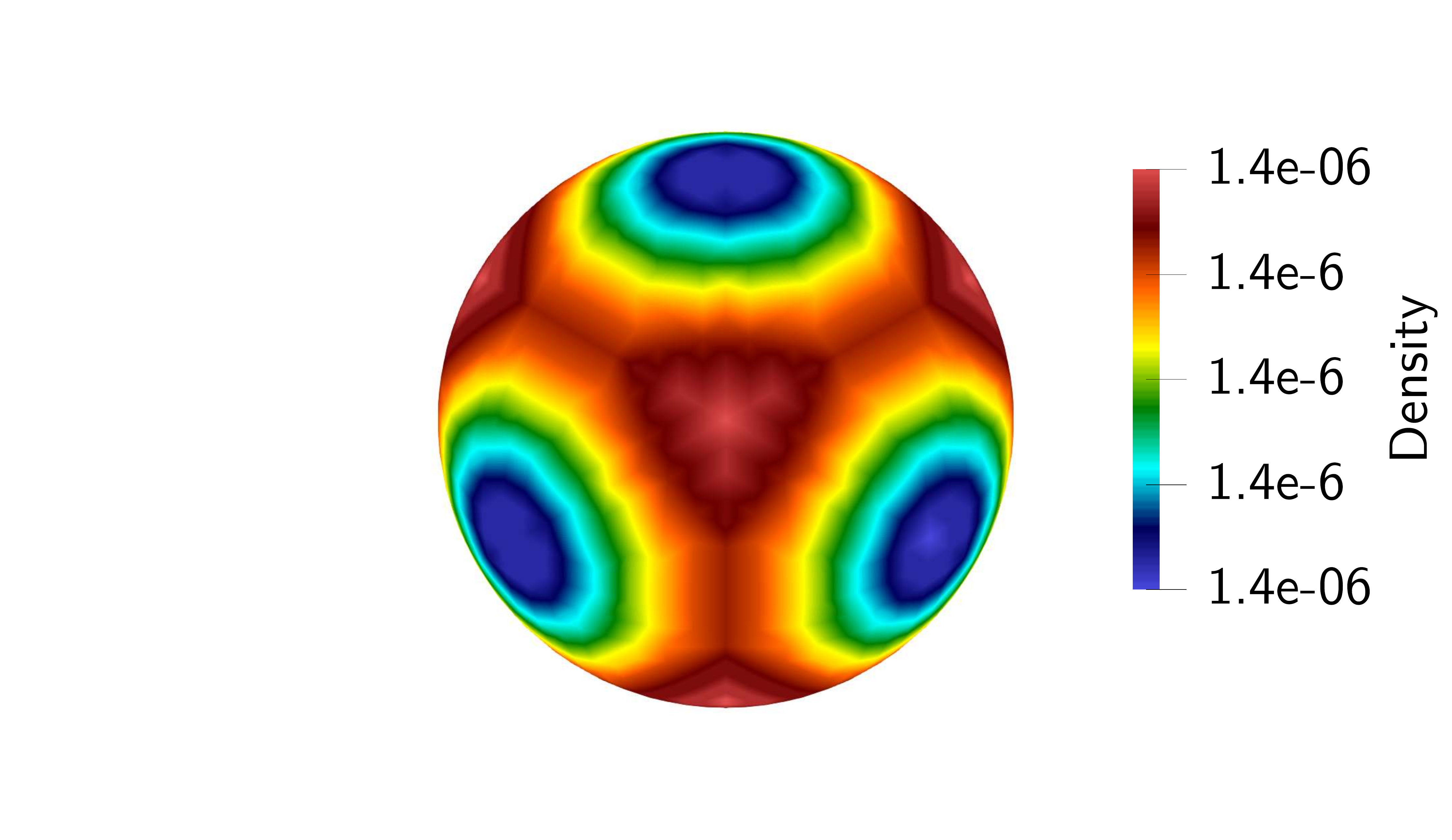}
        \caption{Target Measure $\nu$ (3D View)}
        \label{fig:bench_target_3d}
    \end{subfigure}

    \vspace{0.5cm} %

    \begin{subfigure}[b]{0.4\textwidth}
        \centering
        \includegraphics[width=\textwidth, clip, trim=500 0 30 100]{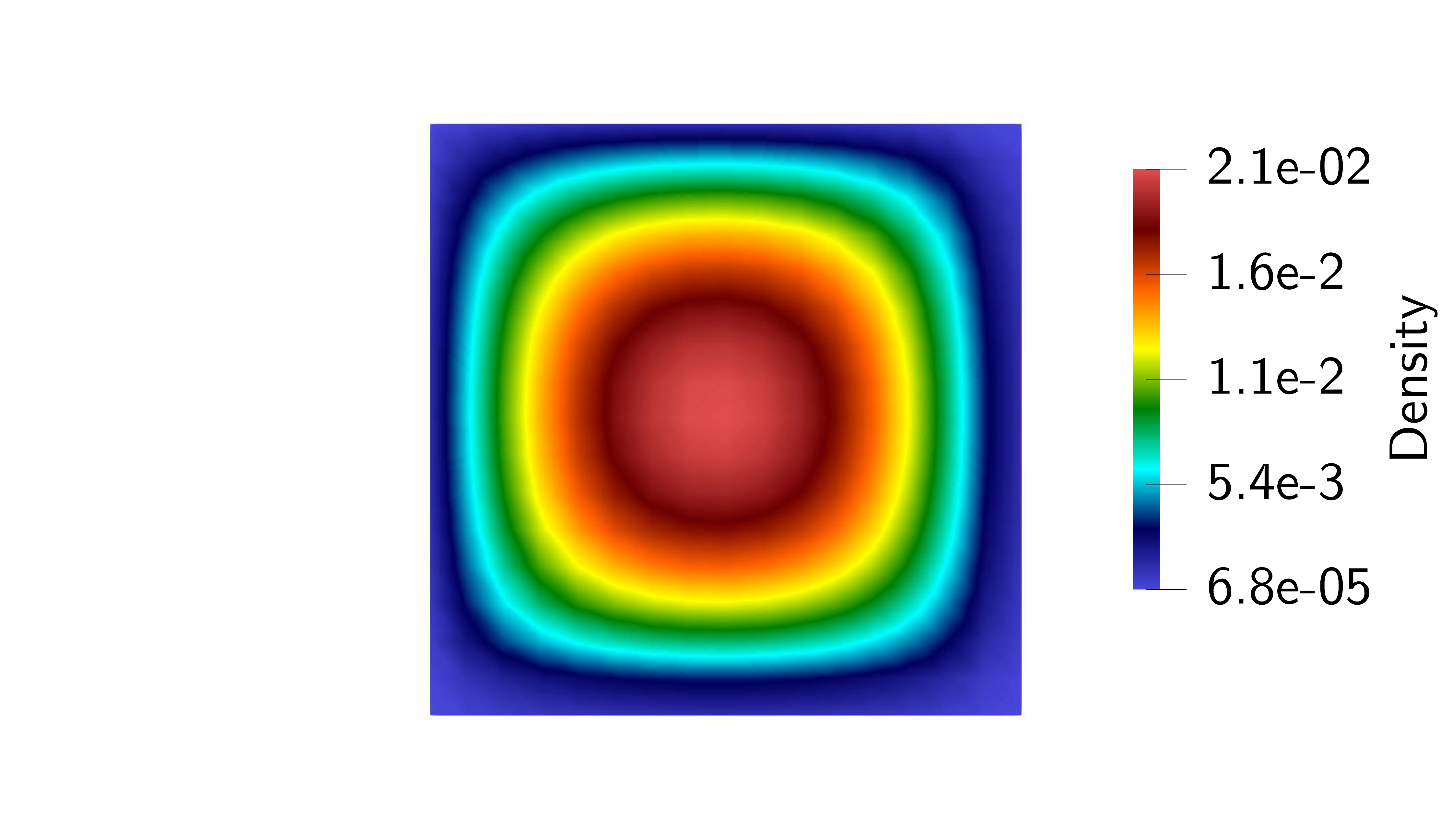}
        \caption{Source Measure $\mu$ ($yz$ Plane Clip)}
        \label{fig:bench_source_clip}
    \end{subfigure}
    \hfill %
    \begin{subfigure}[b]{0.4\textwidth}
        \centering
        \includegraphics[width=\textwidth, clip, trim=500 0 30 100]{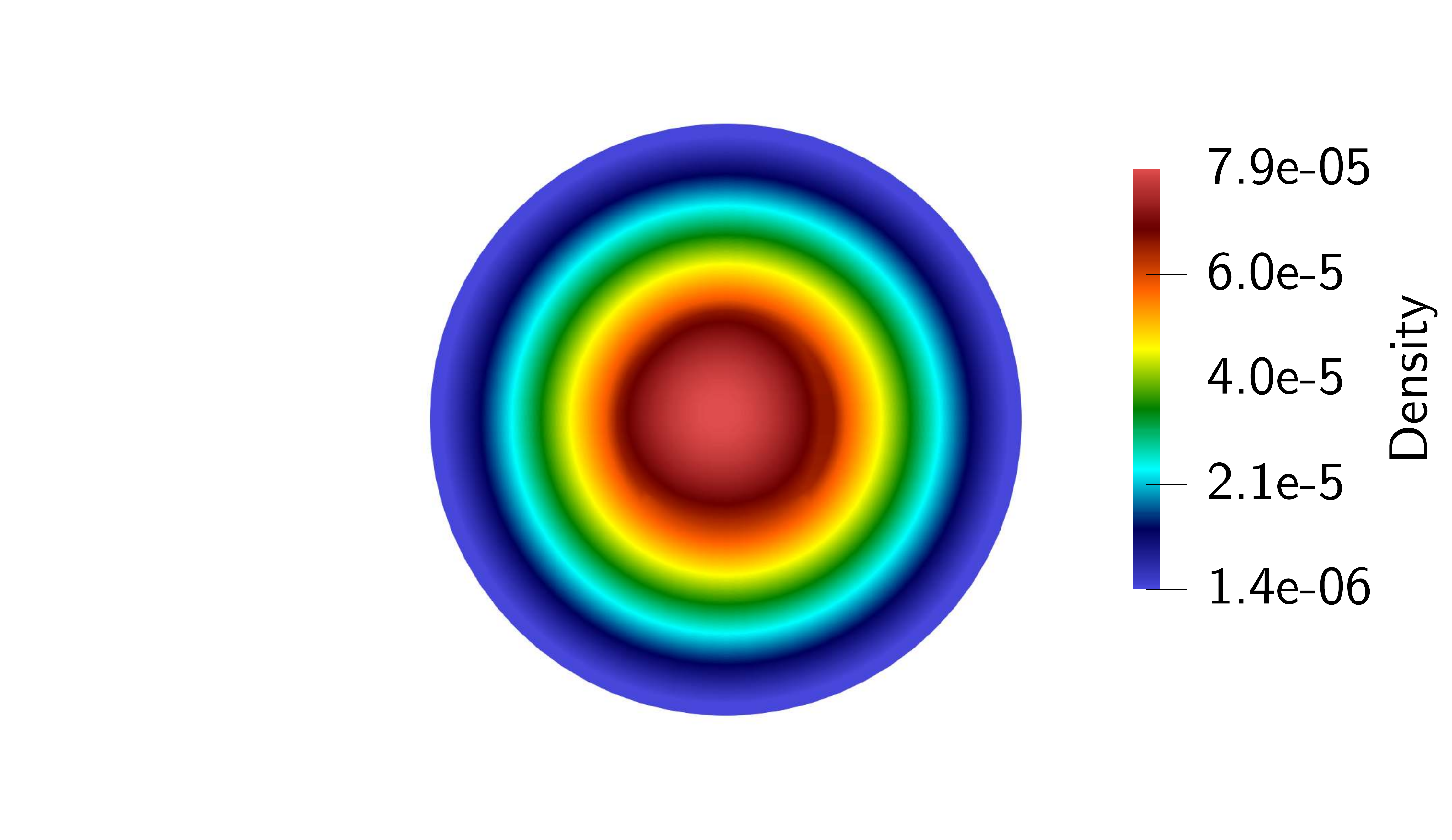}
        \caption{Target Measure $\nu$ ($yz$ Plane Clip)}
        \label{fig:bench_target_clip}
    \end{subfigure}
    \caption{Visualization of the source measure $\mu$ derived from the pressure solution on the cube domain $\Omega_1$, and the discrete target measure $\nu$ derived from the pressure solution on the ball domain $\Omega_2$. Both 3D perspective views and cross-sections through the $yz$ plane ($x=0$) are shown.}
    \label{fig:benchmark_densities}
\end{figure}

In addition to these PDE-derived densities, which introduce spatial non-uniformity and represent our primary benchmark case (referred to as Custom Density), we also consider a Uniform Density scenario where a constant source density $\rho^1(x)$ over $\Omega_1$ and uniform target weights $\nu_j = 1/N$ are used. This simpler case allows us to isolate the impact of density distribution complexity on the performance of our numerical methods.

The RSOT problem was solved by minimizing the dual functional $J_\varepsilon(\bpsi)$ (Eq.~\eqref{eq:dual_functional_to_minimize_redux}) using the L-BFGS algorithm. Numerical integration utilized Gaussian quadrature of order 3.

Unless otherwise specified, all subsequent experiments employ a common set of numerical parameters. To accelerate the evaluation of the dual functional and its gradient, we utilize the Geometric Bound truncation strategy (see Section~\ref{subsubsec:truncation_revised}) with a tolerance of $\tau=\SI{1e-4}{}$. The L-BFGS optimization is terminated when the $L^1$ norm of the gradient falls below a tolerance, i.e., $\|\nabla J_\varepsilon^h(\bpsi) \|_1 \le \delta_{\text{tol}}$, with $\delta_{\text{tol}} = \SI{1e-3}{}$. As detailed in Remark~\ref{rem:stopping_criterion}, this stopping criterion directly relates to the satisfaction of the target marginal constraints.

\subsubsection{Analysis of Truncation Strategies}
\label{subsubsec:param_study_trunc_reg}

This subsection evaluates the performance of different truncation strategies (Pointwise Bound $C_{\text{pw}}$, Integral Bound $C_{\text{int}}$, and Geometric Bound $C_{\text{geom}}$), detailed in Section~\ref{subsubsec:truncation_revised}. The study is conducted for a fixed regularization parameter $\varepsilon=\SI{1e-2}{}$, using the specific control parameters ($\delta_{\text{thr}}$ or $\tau$) indicated in Table~\ref{tab:small_case_trunc_comparison_fixed_eps}. For these tests, the source mesh consists of \num{98304} cells (yielding \num{17969} $P_1$ degrees of freedom for the source density) and the discrete target measure is supported on $N=\num{14761}$ points. The experiments are performed for both the Custom Density and Uniform Density scenarios, as previously described in Section~\ref{subsec:benchmarking}.

\begin{figure}[htbp]
    \centering
    \includegraphics[width=0.9\textwidth]{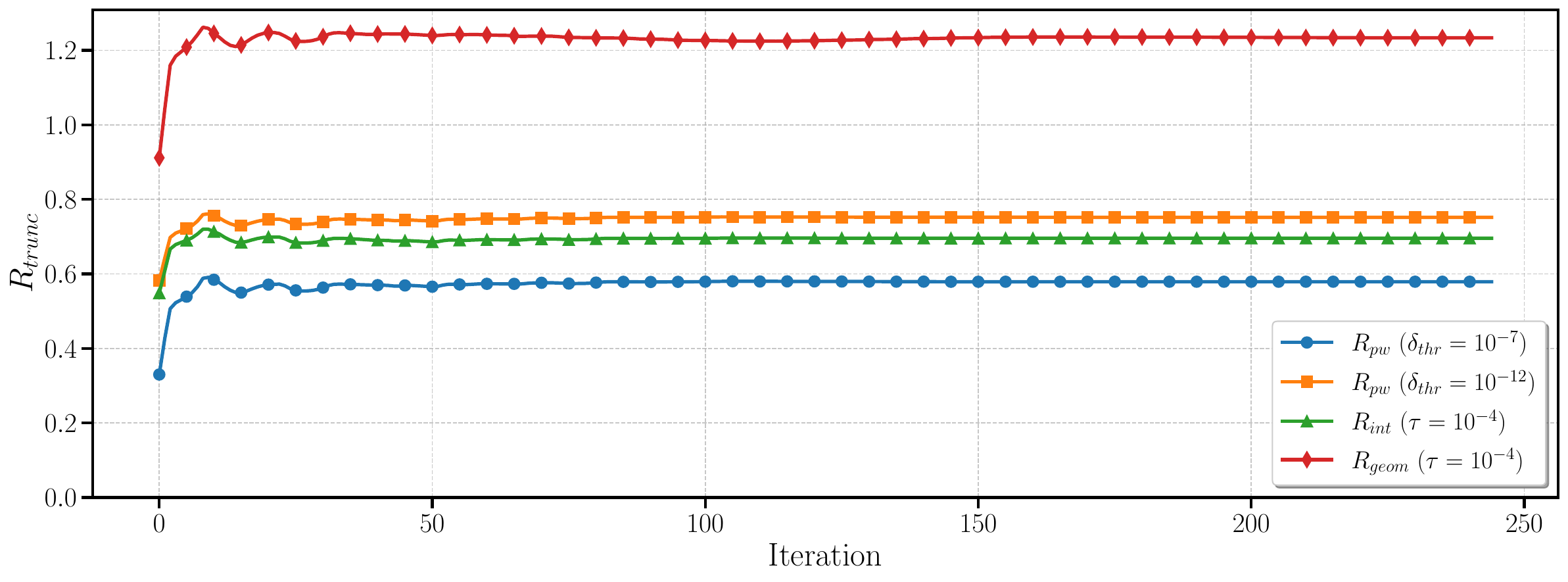}
    \caption{Visualization of truncation cutoffs for different strategies under the custom density. The figure compares the Pointwise bound $C_{\text{pw}}$ with thresholds $\delta_{\text{thr}}=\SI{1e-12}{}$ and $\delta_{\text{thr}}=\SI{1e-7}{}$ against the Geometric bound $C_{\text{geom}}$ and Integral bound $C_{\text{int}}$ with tolerance $\tau=\SI{1e-4}{}$. Each strategy produces a different effective cutoff, determining how many target points are included in the summation at each source quadrature point, which directly impacts both computational cost and numerical accuracy of $J_{\varepsilon}(\bpsi)$.}
    \label{fig:radius_comparison}
\end{figure}

\begin{table}[htbp]
\caption{Performance of truncation strategies.}
  \centering
  \setlength{\tabcolsep}{5.5pt}
  \begin{tabular}{@{}ll|rrcc|rrcc@{}}
      \toprule
      & & \multicolumn{4}{c|}{Custom Density} & \multicolumn{4}{c}{Uniform Density} \\
      \cmidrule(lr){3-6} \cmidrule(l){7-10}
      Strategy & Param. & Iter & Time & $E_{\text{rel}}$ & Conv & Iter & Time & $E_{\text{rel}}$ & Conv \\
      \midrule
      Pointwise & $\delta_{\text{thr}}=\SI{1e-12}{}$ & 244 & 49.1 & \num{1.7e-5} & \checkmark & 119 & 26.8 & \num{1.1e-5} & \checkmark \\
                & $\delta_{\text{thr}}=\SI{1e-10}{}$ & 174 & 33.4 & \num{1.2e-4} & \ding{55} & 127 & 23.8 & \num{5.3e-5} & \checkmark \\
                & $\delta_{\text{thr}}=\SI{1e-7}{}$  & 28  & 5.2  & \num{7.4e-3} & \ding{55} & 5   & 2.0  & \num{1.6e-1} & \ding{55} \\
      \midrule
      Integral  & $\tau=\SI{1e-6}{}$ & 244 & 48.4 & \num{1.4e-5} & \checkmark & 121 & 25.6 & \num{2.1e-5} & \checkmark \\
                & $\tau=\SI{1e-5}{}$ & 244 & 37.6 & \num{3.7e-5} & \checkmark & 128 & 24.7 & \num{5.2e-5} & \checkmark \\
                & $\tau=\SI{1e-4}{}$ & 244 & 34.7 & \num{9.2e-5} & \checkmark & 87  & 20.7 & \num{1.7e-4} & \ding{55} \\
      \midrule
      Geometric & $\tau=\SI{1e-6}{}$ & 244 & 146.8& \num{1.1e-15}& \checkmark & 97  & 57.6 & \num{9.2e-14}& \checkmark \\
                & $\tau=\SI{1e-5}{}$ & 244 & 142.7& \num{2.0e-15}& \checkmark & 97  & 55.9 & \num{1.2e-14}& \checkmark \\
                & $\tau=\SI{1e-4}{}$ & 244 & 139.5& \num{8.9e-15}& \checkmark & 97  & 55.7 & \num{2.0e-13}& \checkmark \\
      \midrule
      None      & ---             & 244 & 236.1& 0      & \checkmark & 97  & 97.3 & 0      & \checkmark \\
      \bottomrule
  \end{tabular}
  \par\medskip
  \label{tab:small_case_trunc_comparison_fixed_eps}
  \footnotesize Note: $E_{\text{rel}}$ is the relative error with respect to the baseline case, computed using converged dual functional values.
\end{table}

The results in Table~\ref{tab:small_case_trunc_comparison_fixed_eps} reveal significant performance differences among the three truncation strategies when applied to both custom and uniform density distributions. All truncation methods substantially outperform the baseline approach without truncation (None). For the custom density, computation times for convergent strategies range from \SI{34.7}{\second} to \SI{167.3}{\second}, compared to \SI{236.1}{\second} for the baseline. For the uniform density, convergent strategies take between \SI{20.7}{\second} and \SI{66.1}{\second}, versus \SI{97.3}{\second} for the baseline.

Key to this comparison is the relative error ($E_{\text{rel}}$), calculated with respect to the baseline case using the converged dual functional values, denoted $J_{\varepsilon}(\bpsi)$. The formula is:
\begin{equation}
  E_{\text{rel}} = \frac{\lvert J_{\varepsilon, \text{trunc}}(\bpsi) - J_{\varepsilon, \text{baseline}}(\bpsi) \rvert}{\lvert J_{\varepsilon, \text{baseline}}(\bpsi) \rvert}.
  \label{eq:relative_error_definition}
\end{equation}

The Pointwise strategy yields the fastest computation times with aggressive (smaller) $\delta_{\text{thr}}$ values but tends to fail convergence with large relative errors. Convergence is typically achieved only with strict threshold settings (e.g., $\delta_{\text{thr}}=\SI{1e-12}{}$), which still provide valuable reductions in computation time. As illustrated in Figure~\ref{fig:radius_comparison}, aggressive settings ($\delta_{\text{thr}}=\SI{1e-7}{}$) produce cutoffs $C_{\text{pw}}$ consistently smaller than $C_{\text{int}}$, resulting in excessive truncation that compromises solution accuracy. Conservative settings ($\delta_{\text{thr}}=\SI{1e-12}{}$) generally exceed $C_{\text{int}}$, providing sufficient approximation quality to ensure convergence.

The Integral bound strategy demonstrates mixed performance. For custom density, all tested $\tau$ values converge, with the least strict tolerance ($\tau=\SI{1e-4}{}$) providing the fastest time while maintaining small relative error. For uniform density, convergence is less consistent, with some configurations failing despite achieving low $E_{\text{rel}}$ values in other cases. The prescribed tolerance $\tau$ is not always a direct predictor of the final observed $E_{\text{rel}}$, because its calculation relies on an approximation of the integral term ($D(\bpsi, C)$ in Eq.~\eqref{eq:truncation_cutoff_int}) as detailed in Section~\ref{sec:acceleration}.

The Geometric bound strategy stands out for exceptional robustness, achieving convergence for all tested parameters and densities while maintaining extremely high solution fidelity ($E_{\text{rel}} \approx \num{1e-13}$ to $\num{1e-15}$). This reliability stems from its conservative formulation (Appendix~\ref{app:truncation_derivation}). Computation times are generally higher than other strategies when they converge, though still significantly faster than the baseline.

A notable observation is that for the custom density case all convergent strategies require the same number of optimization steps (244), but computation times vary significantly (\SI{34.7}{\second} to \SI{236.1}{\second}). This directly implies that truncation strategy choice alters the computational cost per iteration: stricter truncation yields cheaper iterations due to fewer points in the sums, while conservative or no truncation leads to expensive iterations due to the $O(N_q N)$ complexity.

The results suggest a critical relationship between relative error and convergence likelihood. For uniform density, strategies with $E_{\text{rel}} \gtrsim \num{1e-4}$ tend to fail convergence, while those with $E_{\text{rel}} \lesssim \num{1e-5}$ typically converge. A similar but potentially different threshold applies to custom density. This pattern indicates that successful convergence is contingent on the truncation-induced $E_{\text{rel}}$ remaining below a problem-dependent threshold, which varies based on density distribution and other problem characteristics.

This investigation quantifies the practical trade-offs inherent in choosing a truncation strategy. The results from this analysis on a moderately sized problem serve as a valuable guide for selecting hyperparameters in large-scale computations. Given its robustness and high fidelity, the Geometric bound is a reliable choice. We therefore select it with a tolerance of $\tau=\SI{1e-4}{}$ as the default strategy for all subsequent experiments.

\subsubsection{Multilevel Strategies}
\label{subsubsec:multilevel_performance}
To evaluate the multilevel strategies from Section~\ref{sec:multilevel}, we constructed 5-level hierarchies for both source and target measures (details in Table~\ref{tab:hierarchies_details}, visualized in Figures~\ref{fig:source_hierarchy_visualization} and \ref{fig:target_hierarchy_visualization}). The finest level comprises \num{137313} source DoFs and $N_0=\num{116305}$ target points.

We compare three multilevel procedures against a standard single-level solver: Source-only, Target-only, and Combined multilevel, which are detailed in Section~\ref{sec:multilevel}. The Target-only and Combined strategies use the softmax refinement scheme (Eq.~\eqref{eq:softmax_refinement_formula}) to transfer potentials between levels. All experiments were conducted on \num{480} cores with $\varepsilon=\SI{1e-2}{}$.

The results, summarized in Table~\ref{tab:multilevel_comparison_results_detailed}, reveal a transformative performance improvement. The standard solver required \SI{65369}{\second} (over \SI{18}{\hour}) and more than \num{1100} L-BFGS iterations for the custom density case, and failed to converge on the uniform case within a \SI{24}{\hour} time limit. In stark contrast, all multilevel approaches provided substantial acceleration. The most effective strategies, Target-only and Combined, converged after only one or two iterations on the finest, most expensive level. For the complex custom density case, the Target-only strategy finished in just \SI{373}{\second}, achieving a remarkable \SI{175}{\times} speedup over the standard solver. This dramatic improvement stems from the multilevel paradigm's core benefit: using solutions from coarser, computationally cheaper levels provides excellent initial guesses for finer levels, as visually corroborated by the rapid convergence shown in Figure~\ref{fig:functional_evolution}.

While universally beneficial, the strategies showed different performance profiles. The \textbf{Source-only} strategy proved least effective, especially for the custom density, where it took \SI{10450}{\second} (nearly \SI{3}{\hour}) to complete. Its poor performance is explained by two factors: it must solve against the full, high-resolution target measure at every level, and the initial guess it provides for the final level is less effective, a difficulty underscored by the 65 L-BFGS iterations it required on the finest level, compared to just one or two for the other methods. In contrast, the \textbf{Target-only} strategy, which keeps the source mesh fixed while refining the target measure, was the most efficient for the complex custom density (\SI{373}{\second}). Its effectiveness stems from the high-quality warm start provided by the softmax potential refinement (Eq.~\eqref{eq:softmax_refinement_formula}), which is computed on the full-resolution source mesh. As shown in Table~\ref{tab:multilevel_time_breakdown}, the cost for this strategy is dominated by the softmax refinement step and the final level solve.

The \textbf{Combined} strategy, which coarsens both measures, yielded excellent results, proving to be the fastest for the uniform density case (\SI{239}{\second}) and highly competitive for the custom density (\SI{785}{\second}). Its performance advantage in the uniform case stems from the very low cost of its intermediate solves, which use both coarsened meshes and target sets. For the more complex custom density, it required two fine-level iterations versus just one for the Target-only method, accounting for its slightly higher runtime.

In conclusion, this analysis highlights a critical trade-off: the Target-only strategy provides a high-quality warm start via softmax refinement on the full source mesh, while the Combined strategy benefits from cheaper intermediate solves. For problems with complex densities, the quality of the initial guess appears paramount, making the Target-only approach a consistently robust and efficient choice. Overall, the multilevel framework is an indispensable tool for making large-scale RSOT problems computationally tractable.

\begin{figure}[htbp]
    \centering
    \begin{subfigure}[b]{0.32\textwidth}
        \centering
        \includegraphics[width=4.8cm,clip, trim=550 0 30 100]{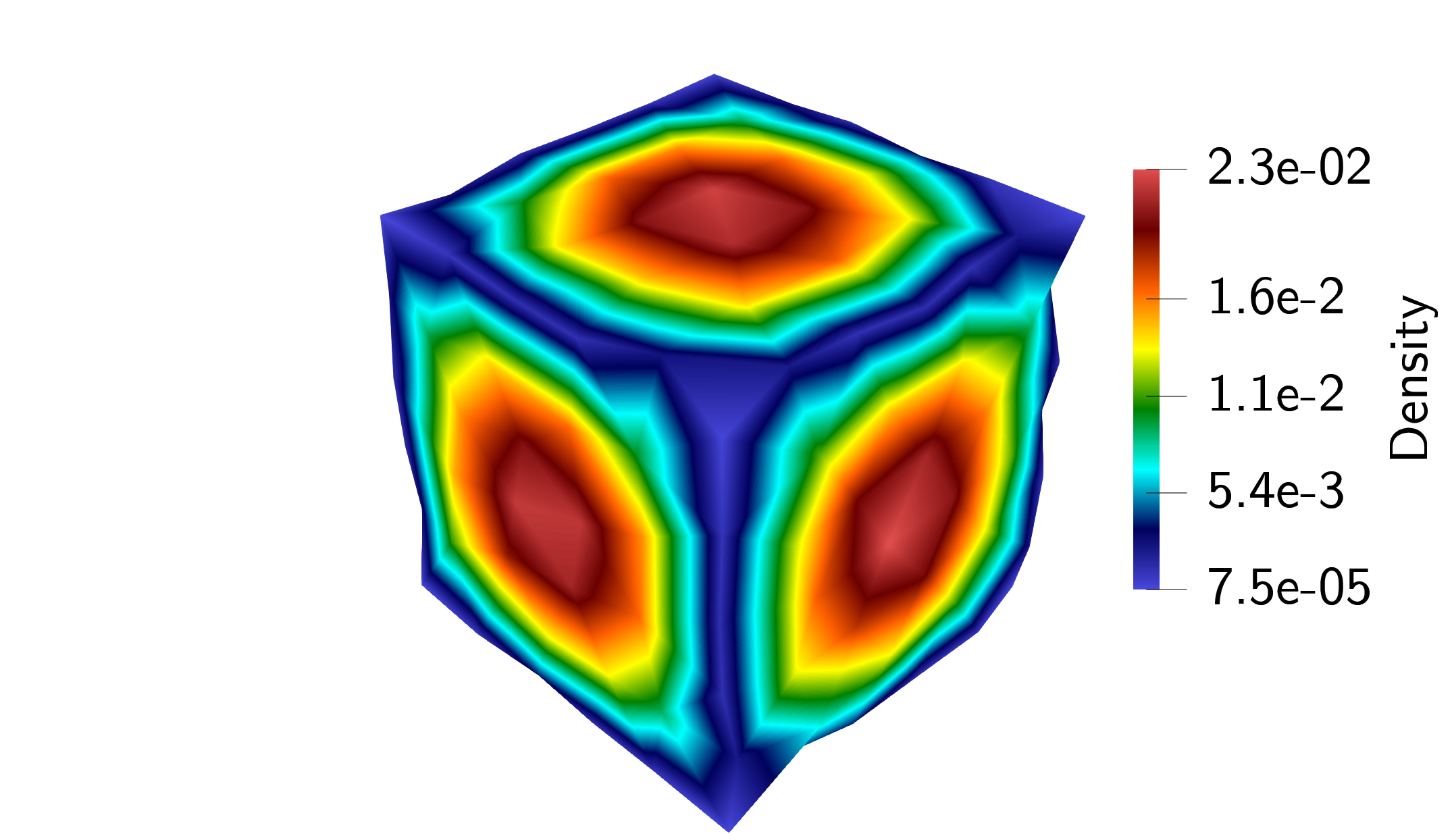}
        \caption{Level 4 (199 DoFs)}
        \label{fig:source_level_4}
    \end{subfigure}
    \hfill
    \begin{subfigure}[b]{0.32\textwidth}
        \centering
        \includegraphics[width=4.8cm,clip, trim=550 0 30 100]{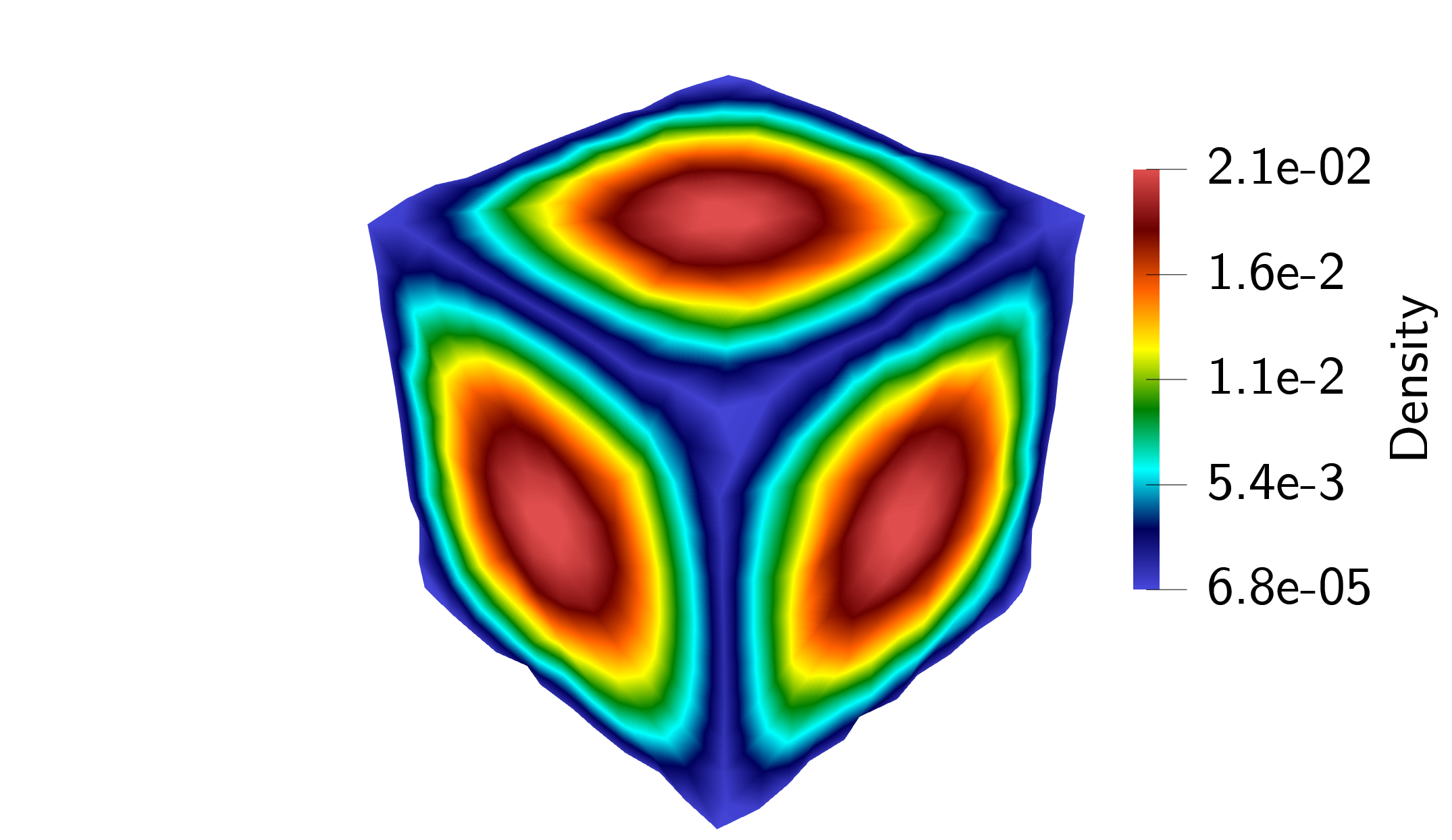}
        \caption{Level 3 (\num{1073} DoFs)}
        \label{fig:source_level_3}
    \end{subfigure}
    \hfill
    \begin{subfigure}[b]{0.32\textwidth}
        \centering
        \includegraphics[width=4.8cm,clip, trim=550 0 30 100]{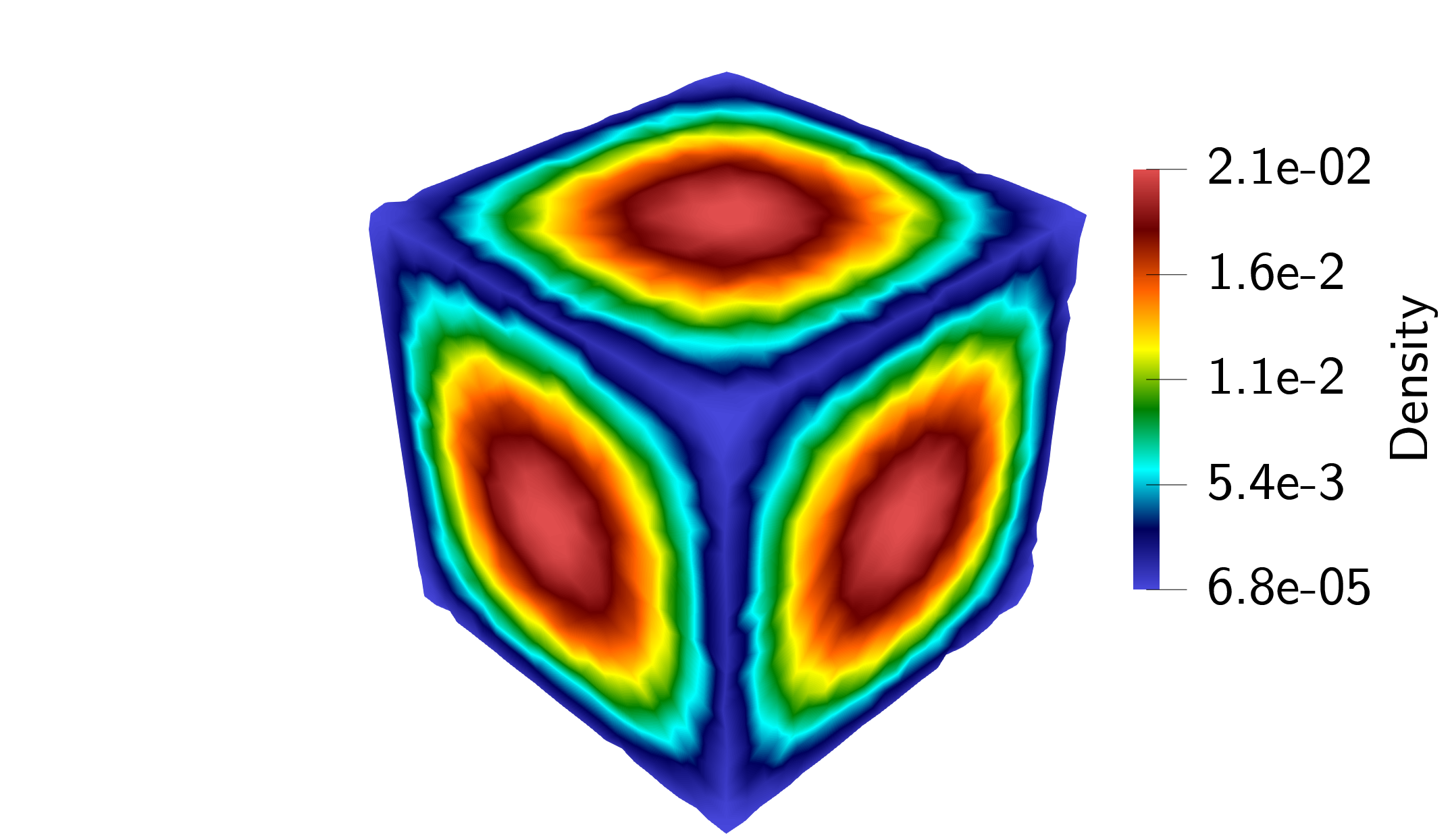}
        \caption{Level 2 (\num{5943} DoFs)}
        \label{fig:source_level_2}
    \end{subfigure}

    \vspace{0.5cm} %

    \begin{center}
    \begin{subfigure}[b]{0.40\textwidth}
        \centering
        \includegraphics[width=4.8cm,clip, trim=550 0 30 100]{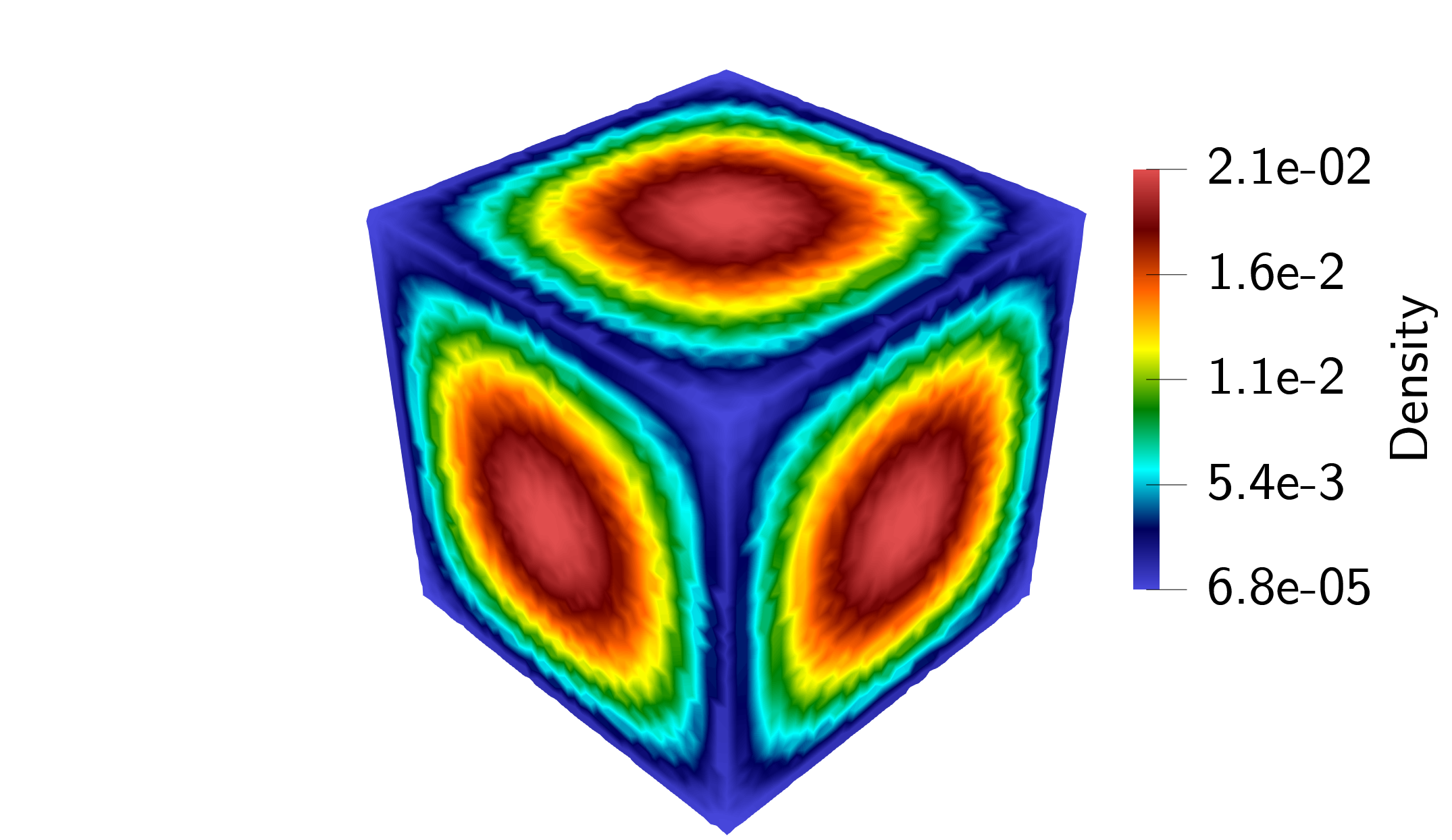}
        \caption{Level 1 (\num{32720} DoFs)}
        \label{fig:source_level_1}
    \end{subfigure}
    \hspace{-1cm} %
    \begin{subfigure}[b]{0.40\textwidth}
        \centering
        \includegraphics[width=4.8cm,clip, trim=550 0 30 100]{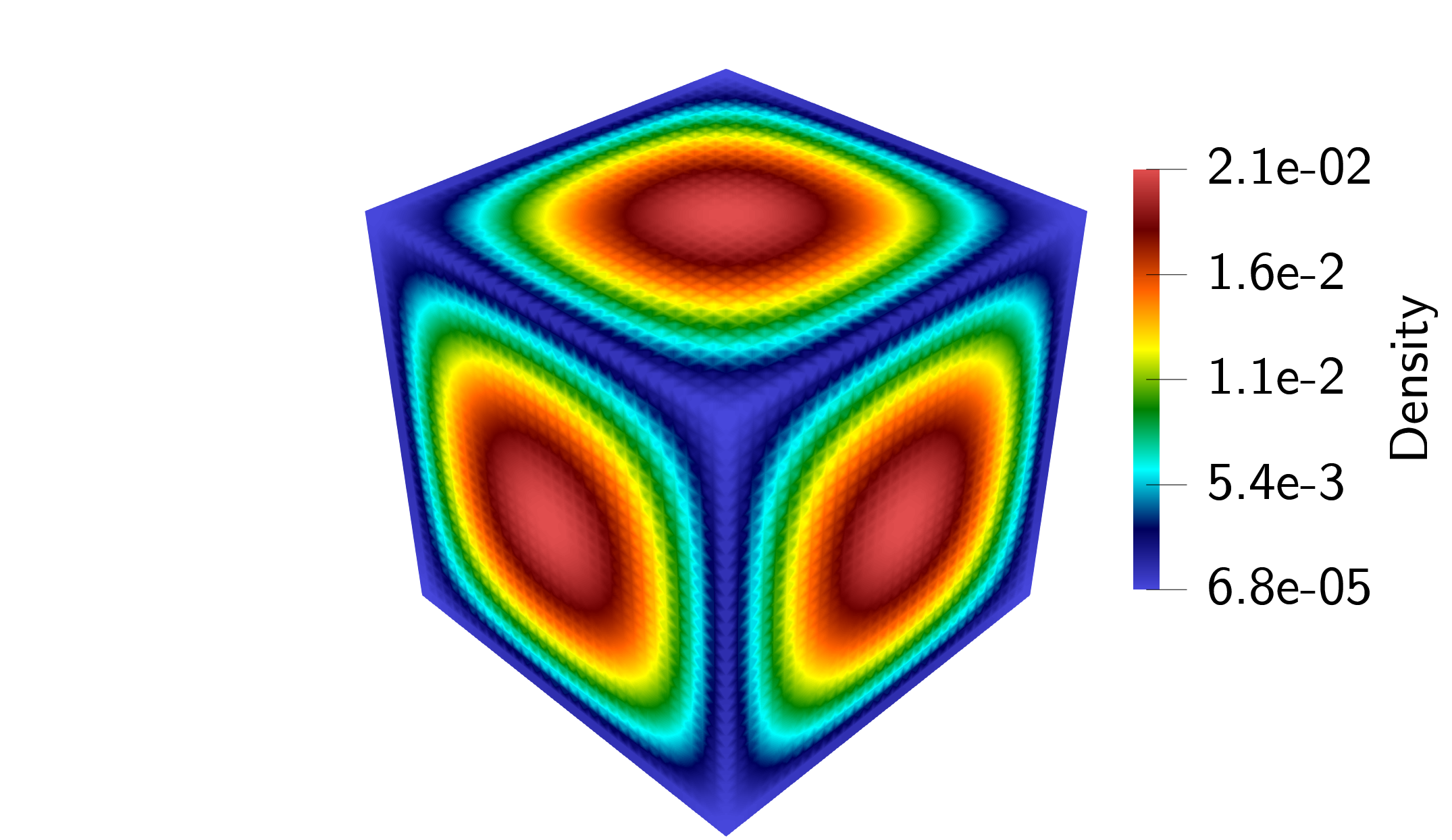}
        \caption{Level 0 (\num{137313} DoFs)}
        \label{fig:source_level_0}
    \end{subfigure}
    \end{center}
    \caption{Visualization of the source mesh hierarchy $\{\mathcal{T}_l\}_{l=0}^4$. Each panel shows a clip through the mesh structure at the corresponding level, indicating the progressive refinement and the associated number of degrees of freedom (vertices) for a $P_1$ discretization.}
    \label{fig:source_hierarchy_visualization}
\end{figure}

\begin{figure}[htbp]
    \centering
    \begin{subfigure}[b]{0.32\textwidth}
        \centering
        \includegraphics[width=4.8cm,clip, trim=16 12 5 0]{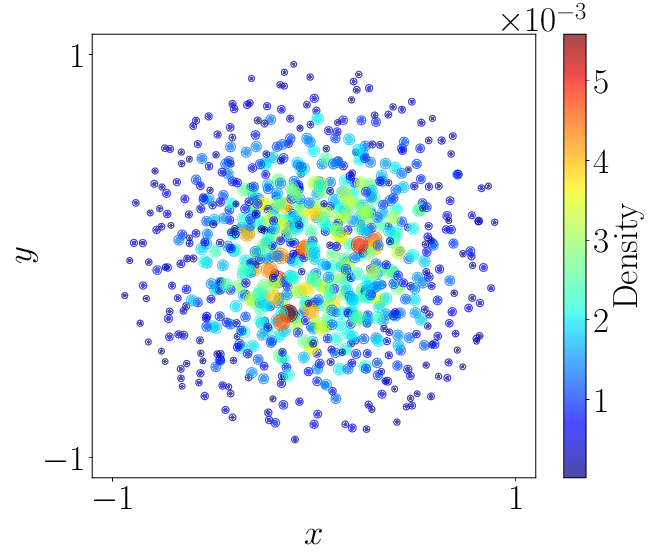}
        \caption{Level 4 (Coarsest, $N_4=781$)}
        \label{fig:target_level_4}
    \end{subfigure}
    \hfill
    \begin{subfigure}[b]{0.32\textwidth}
        \centering
        \includegraphics[width=4.8cm,clip, trim=16 12 5 0]{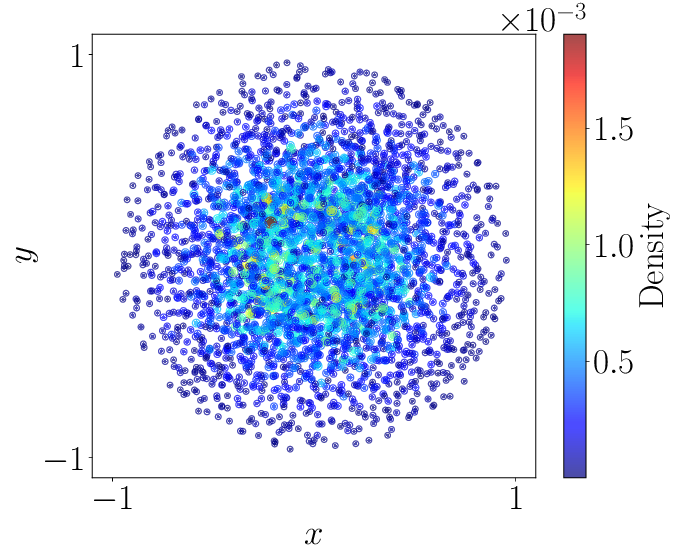}
        \caption{Level 3 ($N_3=\num{3125}$)}
        \label{fig:target_level_3}
    \end{subfigure}
    \hfill
    \begin{subfigure}[b]{0.32\textwidth}
        \centering
        \includegraphics[width=4.8cm,clip, trim=16 12 5 0]{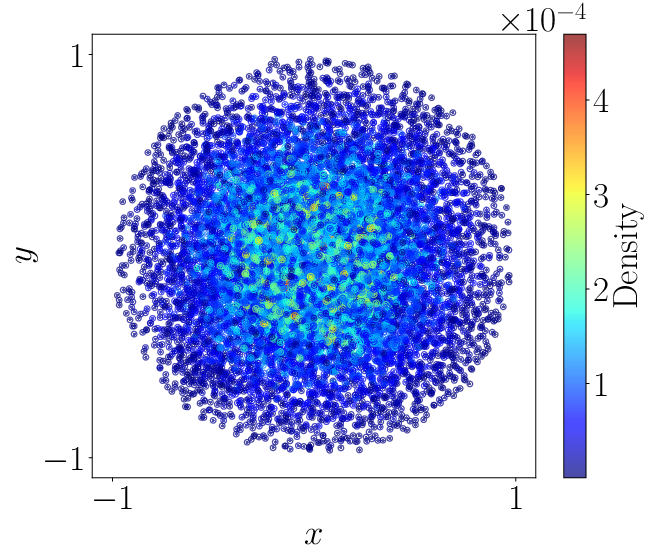}
        \caption{Level 2 ($N_2=\num{12500}$)}
        \label{fig:target_level_2}
    \end{subfigure}

    \vspace{0.5cm} %

    \begin{center}
    \begin{subfigure}[b]{0.40\textwidth}
        \centering
        \includegraphics[width=4.8cm,clip, trim=16 12 5 0]{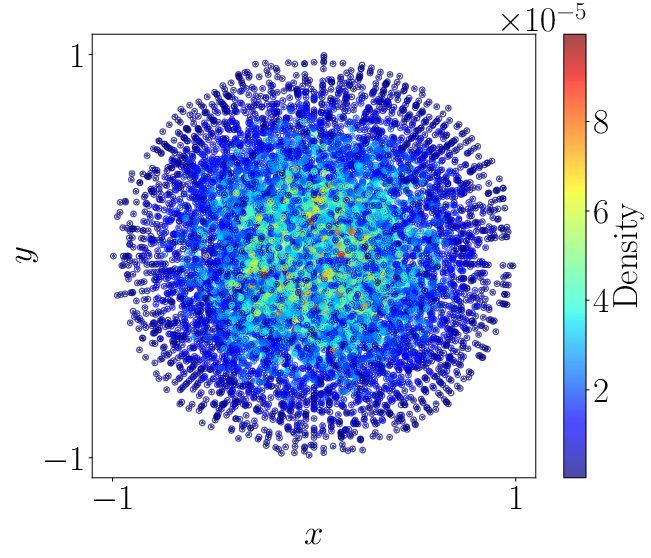}
        \caption{Level 1 ($N_1=\num{50000}$)}
        \label{fig:target_level_1}
    \end{subfigure}
    \hspace{-1cm} %
    \begin{subfigure}[b]{0.40\textwidth}
        \centering
        \includegraphics[width=4.8cm,clip, trim=16 12 5 0]{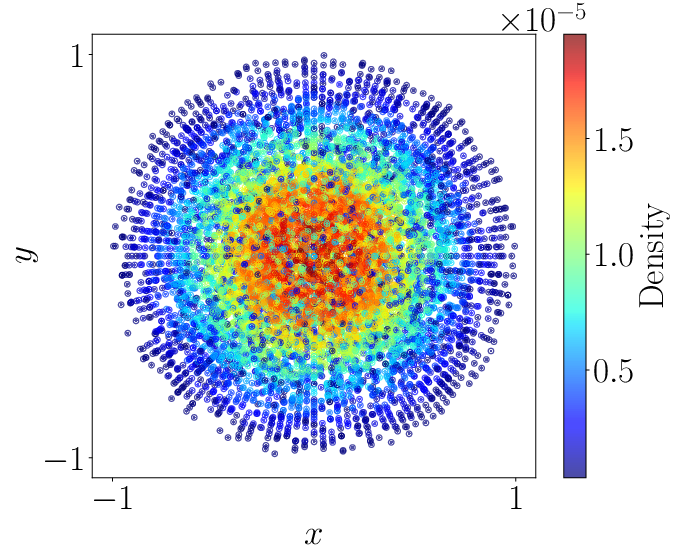}
        \caption{Level 0 (Finest, $N_0=\num{116305}$)}
        \label{fig:target_level_0}
    \end{subfigure}
    \end{center}
    \caption{Visualization of the discrete target measure hierarchy $\{\nu^{(l)}\}_{l=0}^4$ generated via k-means clustering. Each panel shows a scatter plot of the support points $\{y_j^{(l)}\}$ for the corresponding level, projected onto the $xy$ plane.}
    \label{fig:target_hierarchy_visualization}
\end{figure}

\begin{table}[htbp]
    \caption{Details of the 5-level Source Mesh and Target Measure Hierarchies.}
    \label{tab:hierarchies_details}
    \centering
    \begin{tabular}{c c r r c r}
        \toprule
        Level & \multicolumn{3}{c}{Source Mesh Hierarchy ($\mathcal{T}_l$)} & \phantom{abc} & \multicolumn{1}{c}{Target Measure Hierarchy ($\nu^{(l)}$)} \\
        \cmidrule(lr){2-4} \cmidrule(lr){6-6}
        Index ($l$) & & Cells & Vertices / DoFs ($P_1$) && Number of Points ($N_l$) \\
        \midrule
        4 (Coarsest) & & 671     & 199     && 781 \\
        3            & & \num{4559}   & \num{1073}   && \num{3125} \\
        2            & & \num{28962}  & \num{5943}   && \num{12500} \\
        1            & & \num{172933} & \num{32720}  && \num{50000} \\
        0 (Finest)   & & \num{786432} & \num{137313} && \num{116305} \\
        \bottomrule
    \end{tabular}
\end{table}

\begin{table}[htbp]
    \centering
    \resizebox{\textwidth}{!}{%
        \begin{tabular}{l c c c c}
            \toprule
            Strategy         & \multicolumn{2}{c}{Custom Density} & \multicolumn{2}{c}{Uniform Density}                               \\
            \cmidrule(lr){2-3} \cmidrule(lr){4-5}
                                                 & Time (s)                                    & Fine Level Iters (L0)         & Time (s) & Fine Level Iters (L0) \\
            \midrule
            Target-only Multilevel               & \SI{373}{\second}                           & 1                             & \SI{396}{\second}  & 1                \\
            Source-only Multilevel               & \SI{10450}{\second}                         & 65                            & \SI{876}{\second}  & 1                \\
            Combined Multilevel                  & \SI{785}{\second}                           & 2                             & \SI{239}{\second}  & 1                \\
            \midrule
            Standard Solver (Baseline)           & \SI{65369}{\second}                         & \num{1142}                    & $> \SI{24}{h}$    & \num{1372}       \\
                                                 &                                             &                               & (timeout)        & (not converged)  \\
            \bottomrule
        \end{tabular}%
    }
    \caption{Performance comparison of multilevel strategies. Times are total wall-clock seconds. Iterations refer to the count on the finest level (Level 0).}
    \label{tab:multilevel_comparison_results_detailed}
\end{table}

\begin{table}[htbp]
    \centering
    \definecolor{lowpct}{rgb}{0.95,0.95,1.0}
    \definecolor{midpct}{rgb}{0.8,0.8,1.0}
    \definecolor{highpct}{rgb}{0.4,0.4,0.9}
    \newcommand{\cellcol}[1]{%
        \ifdim#1pt<10pt\cellcolor{lowpct}#1\%
        \else\ifdim#1pt<30pt\cellcolor{midpct}#1\%
        \else\cellcolor{highpct}\textcolor{white}{#1\%}
        \fi\fi
    }
    \resizebox{\textwidth}{!}{%
        \begin{tabular}{l rrrrrr rrrrrr}
            \toprule
            & \multicolumn{6}{c}{Custom Density} & \multicolumn{6}{c}{Uniform Density} \\
            \cmidrule(lr){2-7} \cmidrule(lr){8-13}
            Strategy & L4 & L3 & L2 & L1 & L0 & Softmax & L4 & L3 & L2 & L1 & L0 & Softmax \\
            \midrule
            Source   & \cellcol{2.5} & \cellcol{2.2} & \cellcol{1.1} & \cellcol{21.1} & \cellcol{73.1} & \cellcol{0.0} & \cellcol{64.7} & \cellcol{10.1} & \cellcol{8.4} & \cellcol{3.5} & \cellcol{13.4} & \cellcol{0.0} \\
            Target   & \cellcol{5.8} & \cellcol{4.5} & \cellcol{1.9} & \cellcol{14.1} & \cellcol{28.1} & \cellcol{45.5} & \cellcol{4.8} & \cellcol{4.5} & \cellcol{6.5} & \cellcol{10.5} & \cellcol{32.2} & \cellcol{41.6} \\
            Combined & \cellcol{0.0} & \cellcol{0.1} & \cellcol{0.4} & \cellcol{16.4} & \cellcol{66.8} & \cellcol{16.3} & \cellcol{0.0} & \cellcol{0.1} & \cellcol{0.1} & \cellcol{3.7} & \cellcol{43.9} & \cellcol{52.1} \\
            \bottomrule
        \end{tabular}%
    }
    \par\medskip
    \caption{Percentage breakdown of wall-clock time spent per level and for Softmax refinement across different multilevel strategies and density types.}
    \label{tab:multilevel_time_breakdown}
    \footnotesize Note: Level indices L4..L0 represent levels from coarsest to finest. Softmax time is associated with potential refinement steps using Eq.~\eqref{eq:softmax_refinement_formula} in Target-only and Combined strategies. For Combined, percentages per level reflect solving on $(\mathcal{T}_l, \nu^{(l)})$.
\end{table}

\begin{figure}[htbp]
    \centering
    \begin{subfigure}[b]{0.48\textwidth}
        \centering
        \includegraphics[width=\textwidth]{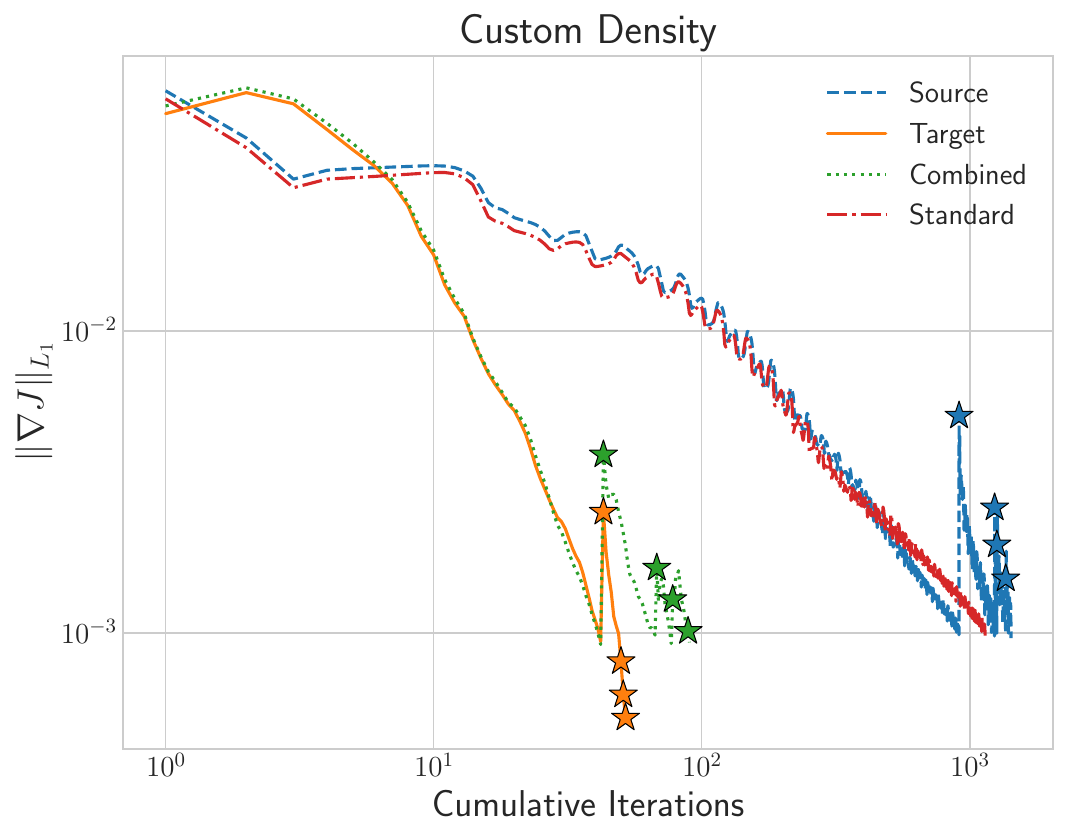}
        \caption{Custom Density}
        \label{fig:functional_evolution_custom}
    \end{subfigure}
    \hfill %
    \begin{subfigure}[b]{0.48\textwidth}
        \centering
        \includegraphics[width=\textwidth]{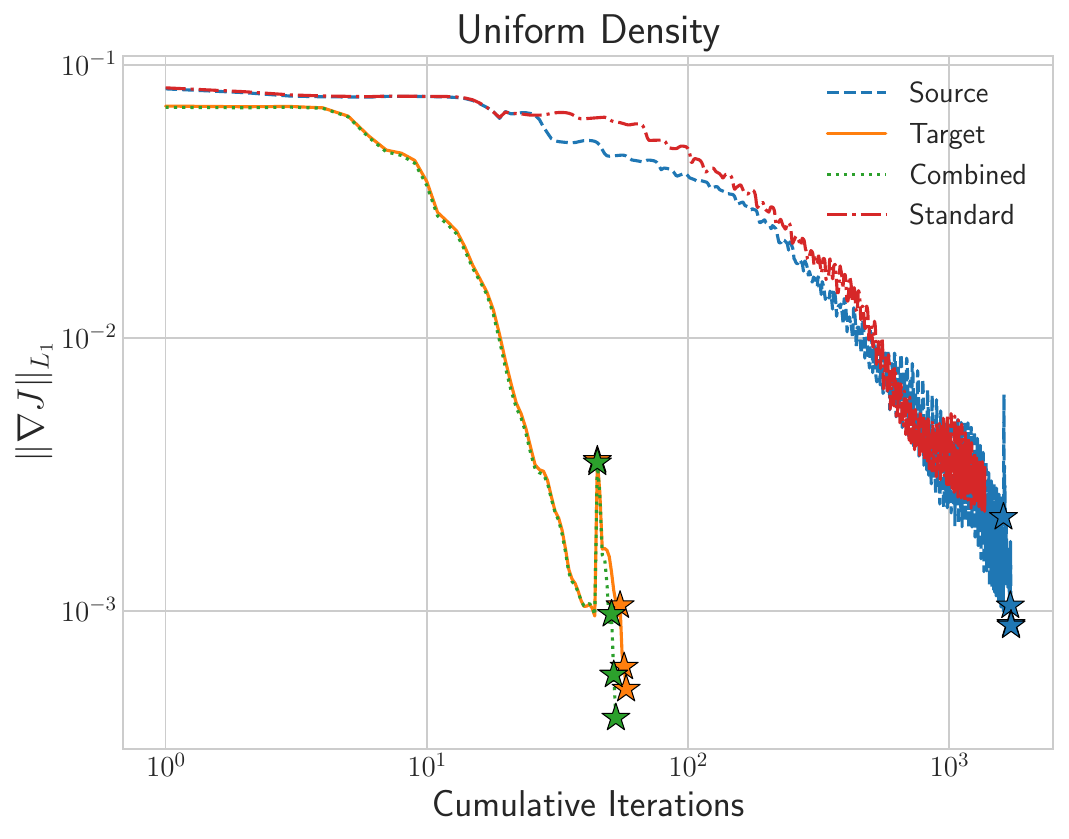}
        \caption{Uniform Density}
        \label{fig:functional_evolution_uniform}
    \end{subfigure}
    \caption{Convergence behavior of different multilevel strategies. The plots show the $L^1$ norm of the dual functional gradient $\|\nabla J_\varepsilon^h(\bpsi) \|_1$ as a function of the cumulative number of L-BFGS iterations across all levels. Vertical dashed lines and star markers indicate transitions between refinement levels, showcasing the rapid convergence upon refinement.}
    \label{fig:functional_evolution}
\end{figure}

\subsubsection{Epsilon Scaling and Convergence to Unregularized OT}
\label{subsubsec:epsilon_convergence}

This section investigates two aspects of the regularization parameter $\varepsilon$: the practical performance of the $\varepsilon$-scaling strategy (Section~\ref{sec:epsilon_scheduling}) and the convergence of the regularized solution to the unregularized one as $\varepsilon \to 0$.

First, we evaluated the effectiveness of $\varepsilon$-scaling for reaching small target regularization parameters ($\varepsilon_{\text{target}} \in \{\SI{1e-6}{}, \SI{1e-7}{}\}$). These tests were run on a problem with a source mesh of \num{98304} cells (\num{17969} DoFs) and a target measure of $N=\num{14761}$ points. We compared a direct solve against a continuation method involving 1, 2, or 3 intermediate solves. This method starts with a larger $\varepsilon$ and reduces it by a factor of 10 at each step, using the previous solution as a warm start. As shown in Table~\ref{tab:bench_epsilon_scaling_iters}, this warm-start strategy drastically reduces the number of L-BFGS iterations required on the final, most challenging solve. However, as Figure~\ref{fig:bench_epsilon_scaling_time_plot} illustrates, total computation time presents a trade-off due to the overhead of the intermediate solves. For the complex custom density at $\varepsilon=\SI{1e-7}{}$, a 2-step strategy proved optimal, yielding a \SI{29.5}{\percent} speedup over the direct solve. For the simpler uniform density, a 1-step strategy was consistently fastest, providing up to a \SI{21.3}{\percent} speedup. These results confirm that $\varepsilon$-scaling is a valuable tool for improving both robustness and efficiency, where the optimal number of steps depends on problem complexity and the target $\varepsilon$.

\begin{table}[htbp]
    \caption{Epsilon Scaling Performance: Final Level L-BFGS Iterations. "$N_s$ scaling steps" denotes a strategy with $N_s$ intermediate solves (total $N_s+1$ solves), starting from an initial $\varepsilon = \varepsilon_{\text{target}} \cdot 10^{N_s}$ and reducing $\varepsilon$ by a factor of 10 at each step until $\varepsilon_{\text{target}}$ is reached.}
    \label{tab:bench_epsilon_scaling_iters}
    \centering
    \begin{tabular}{l S[table-format=3] S[table-format=4] S[table-format=3] S[table-format=3]}
        \toprule
        Scaling Strategy & \multicolumn{2}{c}{Custom Density} & \multicolumn{2}{c}{Uniform Density} \\
        \cmidrule(lr){2-3} \cmidrule(lr){4-5}
        & {$\varepsilon=\SI{1e-6}{}$} & {$\varepsilon=\SI{1e-7}{}$} & {$\varepsilon=\SI{1e-6}{}$} & {$\varepsilon=\SI{1e-7}{}$} \\
        \midrule
        No scaling        & 961 & 1364 & 398 & 623 \\
        1 scaling step    & 79  & 283  & 42  & 88  \\
        2 scaling steps   & 65  & 99   & 39  & 143 \\
        3 scaling steps   & 48  & 435  & 47  & 113 \\
        \bottomrule
    \end{tabular}
\end{table}

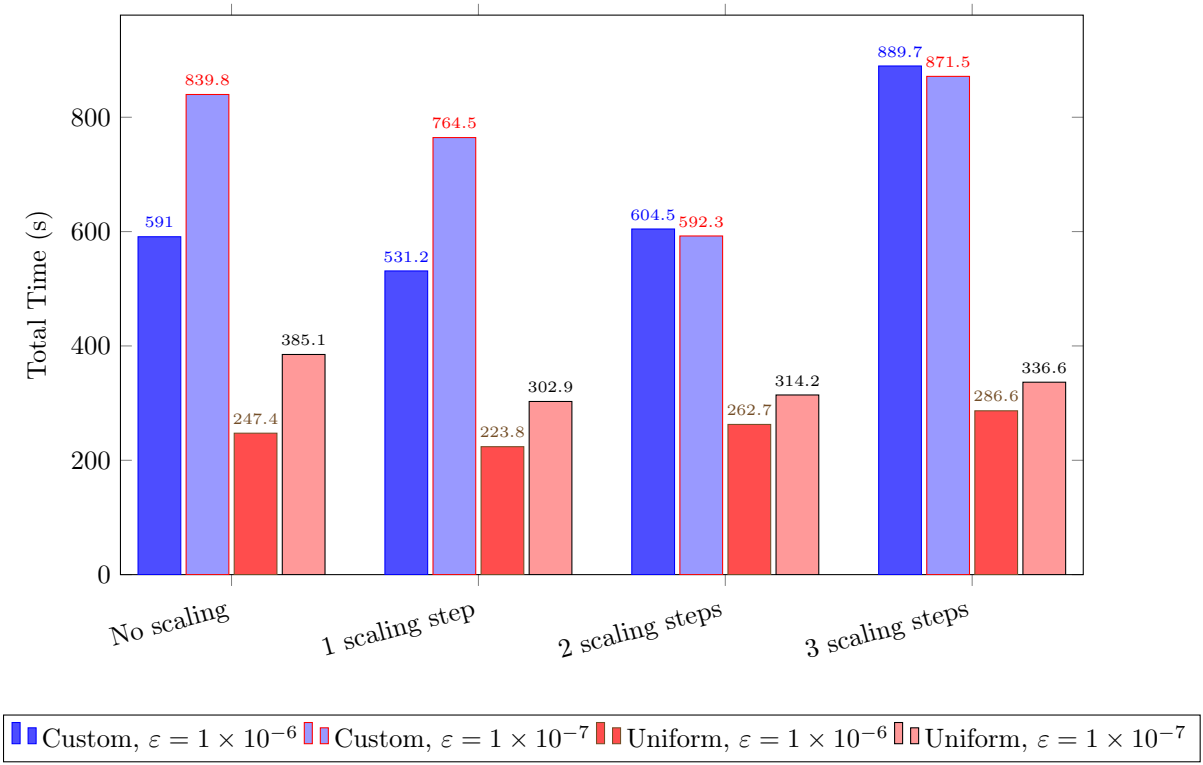
\begin{figure}[htbp]
    \centering
    \begin{tikzpicture}
    \begin{axis}[
        width=0.95\textwidth, %
        height=0.6\textwidth, %
        ybar, %
        bar width=16pt, %
        enlarge x limits=0.15, %
        ylabel={Total Time (s)},
        symbolic x coords={No scaling, 1 scaling step, 2 scaling steps, 3 scaling steps},
        xtick=data,
        xticklabel style={rotate=15, anchor=north east}, %
        nodes near coords, %
        nodes near coords align={vertical},
        nodes near coords style={font=\tiny, /pgf/number format/.cd,fixed,precision=1}, %
        legend style={at={(0.5,-0.25)}, anchor=north, legend columns=-1}, %
        ymin=0, %
    ]
    \addplot+[fill=blue!70!white] coordinates {
        (No scaling, 590.99)
        (1 scaling step, 531.17)
        (2 scaling steps, 604.54)
        (3 scaling steps, 889.66)
    };
    \addlegendentry{Custom, $\varepsilon=\SI{1e-6}{}$}

    \addplot+[fill=blue!40!white] coordinates {
        (No scaling, 839.80)
        (1 scaling step, 764.45)
        (2 scaling steps, 592.33)
        (3 scaling steps, 871.51)
    };
    \addlegendentry{Custom, $\varepsilon=\SI{1e-7}{}$}

    \addplot+[fill=red!70!white] coordinates {
        (No scaling, 247.41)
        (1 scaling step, 223.80)
        (2 scaling steps, 262.72)
        (3 scaling steps, 286.55)
    };
    \addlegendentry{Uniform, $\varepsilon=\SI{1e-6}{}$}

    \addplot+[fill=red!40!white] coordinates {
        (No scaling, 385.06)
        (1 scaling step, 302.87)
        (2 scaling steps, 314.23)
        (3 scaling steps, 336.58)
    };
    \addlegendentry{Uniform, $\varepsilon=\SI{1e-7}{}$}
    \end{axis}
    \end{tikzpicture}
    \caption{Total wall-clock time for different $\varepsilon$-scaling strategies, density types, and target $\varepsilon$ values.}
    \label{fig:bench_epsilon_scaling_time_plot}
\end{figure}

Second, we verified the convergence of the regularized potential $\bpsi^*_\varepsilon$ to the unregularized potential $\bpsi^*_0$. For this test, we computed solutions for the uniform density case for $\varepsilon$ from $10^1$ down to $10^{-8}$, leveraging the $\varepsilon$-scaling strategy to reach the smallest values. A key advantage of this continuation method is that it efficiently yields the entire sequence of solutions for all intermediate $\varepsilon$ values. These potentials were compared against a ground-truth solution $\bpsi^*_0$ obtained from an exact SOT solver based on power diagrams~\cite{Aurenhammer1991, levy2015numerical}, computed using the Geogram library~\cite{Levy2012Geogram}. This analysis was limited to the uniform density case due to the exact solver's incompatibility with our custom density's FE representation.

Figure~\ref{fig:bench_eps_convergence} plots the relative $L^2$ error $\|\bpsi^*_\varepsilon - \bpsi^*_0\|_2 / \|\bpsi^*_0\|_2$ versus $\varepsilon$. The results clearly show that the regularized solution converges to the unregularized one as $\varepsilon \to 0$. The error reaches its minimum of \num{8.0e-4} at $\varepsilon \approx \SI{1e-4}{}$. For smaller $\varepsilon$ values, the error plateaus around \num{1.1e-3}, a behavior likely caused by numerical precision limits and the increasing ill-conditioning of the problem. This confirms that our RSOT solver effectively approximates the true SOT solution, with the closest agreement found near $\varepsilon=\SI{1e-4}{}$ for this problem.

\begin{figure}[htbp]
    \centering
    \includegraphics[width=0.8\textwidth]{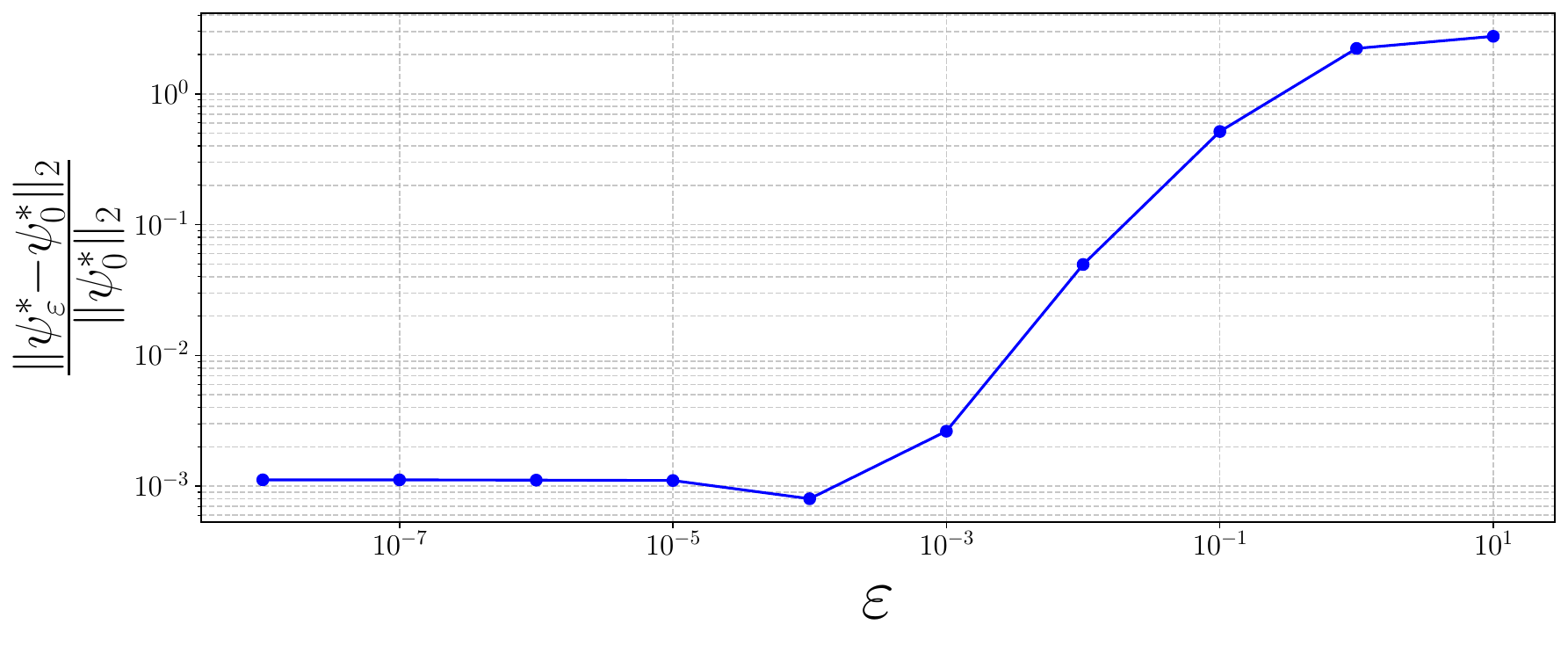} %
    \caption{Relative $L^2$ error between the regularized potential $\bpsi^*_\varepsilon$ and unregularized potential $\bpsi^*_0$ vs. $\varepsilon$.}
    \label{fig:bench_eps_convergence}
\end{figure}

\subsubsection{HPC Scalability}
\label{subsubsec:scalability}

This section assesses the parallel performance of our framework using the combined multilevel strategy and acceleration techniques. All HPC experiments were conducted on the Cineca G100 cluster, where each compute node has 48 cores, using a regularization parameter of $\varepsilon = \SI{1e-2}{}$. We analyze both strong and weak scaling to evaluate efficiency and scalability.

For strong scaling, we fixed the problem size (\num{17969} source DoFs, $N=\num{14761}$ target points) and measured the wall time $T(P)$ while scaling the number of compute nodes $P$ from 1 to 15. From this, we calculated the speedup, $S(P) = T(1)/T(P)$, and parallel efficiency, $E(P) = S(P)/P \times 100\%$. The results, presented in Figure~\ref{fig:bench_strong_scaling_plot}, show the framework achieves a speedup of \num{9.27} on 15 nodes, corresponding to a parallel efficiency of \SI{61.8}{\percent}. This efficiency demonstrates robust strong scaling, with the expected performance drop attributed to communication overhead and serial L-BFGS components becoming more prominent at scale.

\begin{figure}[htbp]
    \centering
    \includegraphics[width=1\textwidth]{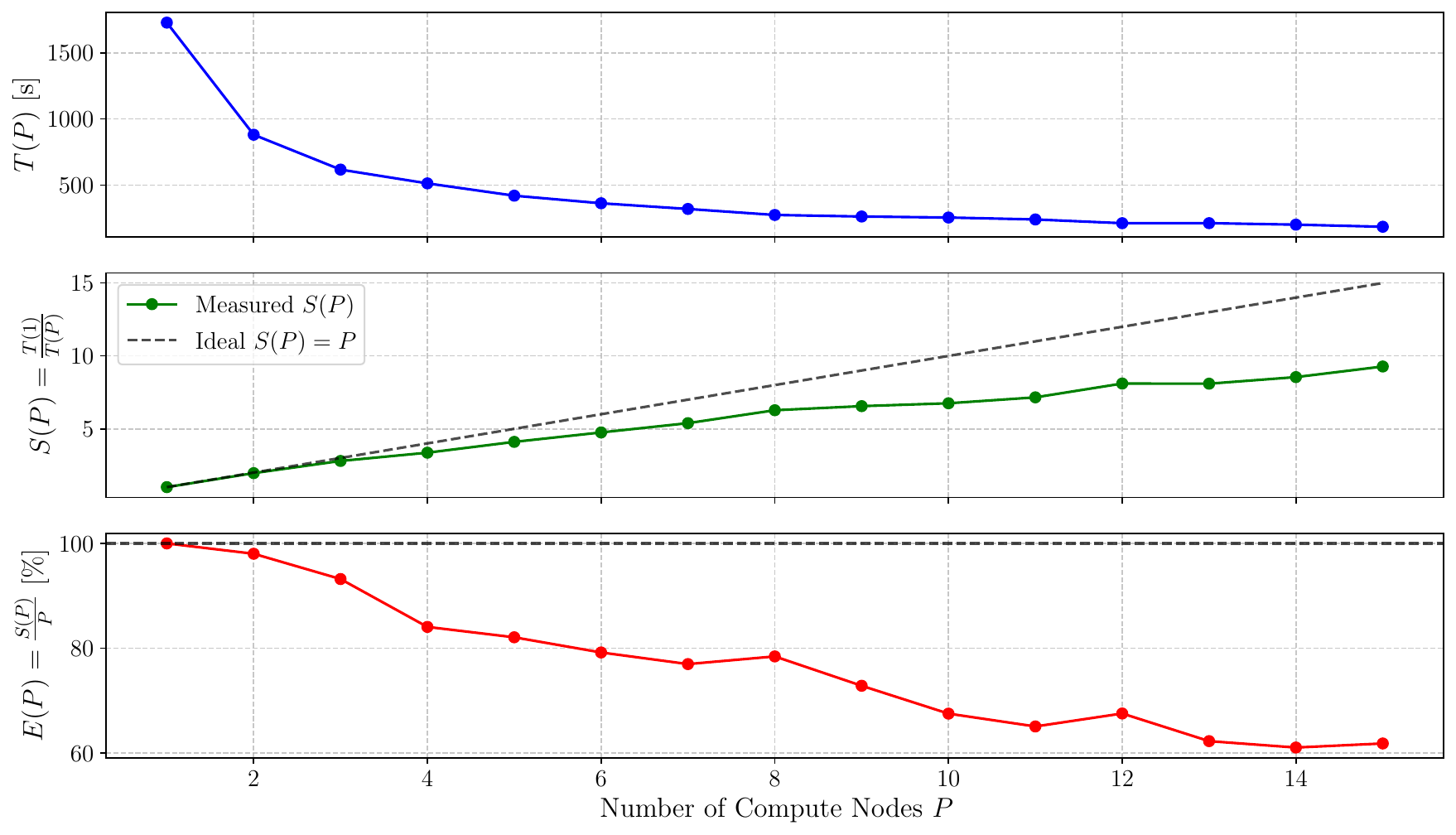}
    \caption{Strong scaling performance for a fixed problem size. Plots show Wall Time $T(P)$, Speedup $S(P)$, and Efficiency $E(P)$ versus the number of compute nodes $P$.}
    \label{fig:bench_strong_scaling_plot}
\end{figure}

For weak scaling, we increased the problem size proportionally with the number of cores to maintain a constant workload per core (see Table~\ref{tab:weak_scaling_setup}). We tested two scenarios: a pure multithreaded (OpenMP) setup on a single node, and a hybrid MPI+OpenMP setup across multiple nodes. In our parallel model, the source mesh is partitioned across MPI processes, while the target point set is replicated on each node.

The results are presented in Figure~\ref{fig:bench_weak_scaling_plot}. On a single node, the pure OpenMP implementation exhibits nearly ideal weak scaling, demonstrating the efficiency of shared-memory parallelism with dynamic load balancing. The hybrid MPI+OpenMP version shows a slight performance degradation (a \SI{13}{\percent} increase in normalized time on 4 nodes). This is attributed to work imbalance from the static MPI-level domain decomposition; since the entire non-uniform target point set is replicated on each node, a process assigned a source subdomain that interacts with a denser region of the target measure has a disproportionately higher computational load, forcing other processes to wait at global synchronization points. Despite this, the performance is strong, confirming the framework's capability to efficiently solve large-scale problems on distributed-memory systems.

\begin{table}[htbp]
    \centering
    \caption{Weak scaling experiment configurations, showing the per-core workload for each parallel strategy.}
    \label{tab:weak_scaling_setup}
    \begin{tabular}{l r r r r r r}
        \toprule
        & & & \multicolumn{2}{c}{Multithread (Single-Node)} & \multicolumn{2}{c}{Multiprocess (Hybrid)} \\
        \cmidrule(lr){4-5} \cmidrule(lr){6-7}
        Mesh & Cells & Vertices & Cores & Cells/Core & Cores & Cells/Core \\
        \midrule
        1 & \num{41672}  & \num{8376}  & 12 & $\approx\num{3470}$ & 48  & $\approx\num{870}$ \\
        2 & \num{89125}  & \num{17321} & 24 & $\approx\num{3710}$ & 96  & $\approx\num{930}$ \\
        3 & \num{130228} & \num{24902} & 36 & $\approx\num{3620}$ & 144 & $\approx\num{900}$ \\
        4 & \num{171884} & \num{32541} & 48 & $\approx\num{3580}$ & 192 & $\approx\num{900}$ \\
        \bottomrule
    \end{tabular}
\end{table}

\begin{figure}[htbp]
    \centering
    \begin{tikzpicture}
    \begin{axis}[
        width=0.8\textwidth,
        height=0.5\textwidth,
        xlabel={Total Number of Cores},
        ylabel={Normalized Execution Time},
        xmin=0, xmax=200,
        ymin=0.9, ymax=1.2,
        xtick={12, 24, 36, 48, 96, 144, 192},
        xticklabels={12, 24, 36, 48, 96, 144, 192},
        ytick={0.9, 0.95, 1.0, 1.05, 1.1, 1.15, 1.2},
        legend pos=north west,
        ymajorgrids=true,
        grid style=dashed,
        cycle list name=exotic,
    ]

    \addplot[
        color=blue,
        mark=square,
        ]
        coordinates {
        (48,1.0) (96,1.0865) (144,1.1288) (192,1.1276)
        };
        \addlegendentry{Multiprocess (Hybrid MPI+OpenMP)}

    \addplot[
        color=red,
        mark=triangle,
        ]
        coordinates {
        (12,1.0) (24,0.9626) (36,0.9738) (48,0.9939)
        };
        \addlegendentry{Multithread (Single-Node OpenMP)}

    \end{axis}
    \end{tikzpicture}
    \caption{Weak scaling performance. The "Multiprocess" line shows hybrid MPI+OpenMP scaling across multiple nodes. The "Multithread" line shows pure OpenMP scaling on a single node. Ideal scaling corresponds to a constant normalized time of 1.0.}
    \label{fig:bench_weak_scaling_plot}
\end{figure}
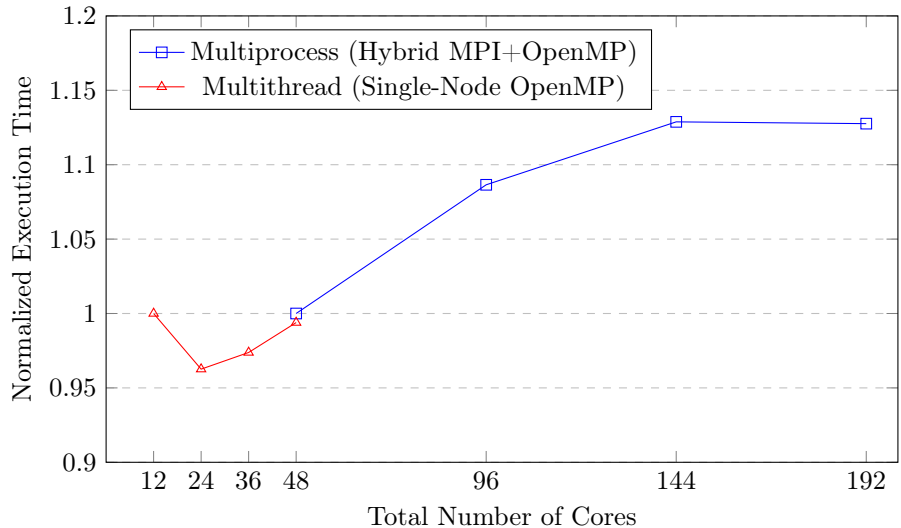  %
\subsection{Application: Wasserstein Barycenter}
\label{subsec:barycenter_application}

A powerful application of Optimal Transport lies in the ability to define and compute geometric averages of probability measures. This is achieved through the concept of the Wasserstein barycenter, a notion rigorously introduced and analyzed by Agueh and Carlier \cite{Agueh2011}. The barycenter provides a principled way to interpolate between and find the "mean" of a set of distributions, preserving their underlying geometric structure in a way that simple linear averaging cannot. This capability has found direct application in fields like computer graphics, where advanced techniques such as Wasserstein Blue Noise Sampling \cite{Qin2017} frame multi-class sampling itself as the computation of a constrained Wasserstein barycenter (see Section~\ref{subsec:blue_noise}).

However, the computation of these barycenters is exceptionally demanding.
Usually iterative methods are employed, requiring the solution of a full Optimal Transport problem from each source measure to the current barycenter estimate \emph{at every single iteration}. When dealing with complex, high-resolution source measures, such as those discretized by fine FE meshes, this nested optimization loop presents a formidable computational barrier, often rendering the problem intractable without significant acceleration.

We demonstrate the capability of our RSOT framework to efficiently compute Wasserstein barycenters by integrating it with an iterative optimization scheme. This showcases how the RSOT solver can serve as a core component in more complex, nested optimization problems while leveraging the multilevel acceleration strategies developed in this work.

The entropy-regularized Wasserstein barycenter $\nu^*$ of a set of measures $\{\mu_i\}_{i=1}^K$ with weights $\{\lambda_i\}_{i=1}^K$ (where $\sum \lambda_i = 1$) is the measure that minimizes the weighted sum of regularized transport costs:
\begin{equation}
    \label{eq:barycenter_objective}
    \nu^* = \argmin_{\nu \in \Prob(\R^d)} \sum_{i=1}^K \lambda_i \W_{\varepsilon,c}(\mu_i, \nu).
\end{equation}
We seek a discrete barycenter $\nu = \sum_{j=1}^{N_b} \nu_j \delta_{y_j}$ with fixed uniform weights $\nu_j = 1/N_b$, which transforms the problem into an optimization over the support point locations $Y = \{y_j\}_{j=1}^{N_b}$.

For the squared Euclidean cost, $c(x,y) = \frac{1}{2}\|x-y\|^2$, the optimality condition for a barycenter point $y_j^*$ leads to a fixed-point equation:
\begin{equation}
    \label{eq:barycenter_fixed_point}
    y_j^* = \sum_{i=1}^K \lambda_i T_{i,j}, \quad \text{where} \quad T_{i,j} = \frac{1}{\nu_j} \int_{\OmegaDomain} x \, \dd\pi_{i,j,\varepsilon}^*(x).
\end{equation}
Here, $\pi_{i,j,\varepsilon}^*$ is the component of the optimal transport plan $\pi_{i,\varepsilon}^*$ between $\mu_i$ and $\nu$ that transports mass to the location $y_j$ (see Remark~\ref{rem:semi_discrete_plans}), and $T_{i,j}$ is the corresponding conditional barycenter (centroid). This condition motivates a Lloyd-type fixed-point iteration, which we stabilize with a damping parameter $\theta$:
\begin{equation}
\label{eq:damped_lloyd}
y_j^{(t+1)} = (1-\theta)y_j^{(t)} + \theta \left(\sum_{i=1}^K \lambda_i T_{i,j}^{(t)}\right).
\end{equation}
The conditional barycenters $T_{i,j}^{(t)}$ are computed at each iteration using the optimal potentials $\bpsi_i^{(t)}$ obtained from solving the respective RSOT subproblems.

In the general case of Wasserstein barycenters, for which an update of the support point locations and discrete mass of the target measure needs to be performed, we refer to Algorithm~\ref{alg:wasserstein_barycenters}. In case $\Omega$ is an arbitrary Riemannian manifold, e.g. the sphere with the geodesic distance, the evaluation of the gradient of the support points $\mathbf{y}$ and their location update, should be carried out on the manifold, taking into account the metric information.

\begin{remark}[Lloyd's Algorithm vs. Gradient Descent]
The fixed-point iteration in Eq.~\eqref{eq:damped_lloyd} is a specialized and highly efficient method for the quadratic cost setting. More generally, one could minimize the barycenter objective $F_\varepsilon(Y)$ using a standard gradient descent scheme: $y_j^{(t+1)} = y_j^{(t)} - \alpha \nabla_{y_j} F_\varepsilon(Y^{(t)})$, where $\alpha$ is a learning rate. The required gradient is given by
\begin{equation}
\label{eq:barycenter_gradient_general}
\nabla_{y_j} \W_{\varepsilon,c}(\mu_i, \nu) =
\int_{\OmegaDomain} \nabla_{y_j} c(x, y_j) \, d\pi_{i,j,\varepsilon}^*(x).
\end{equation}
An analogous results holdf for $\nabla_\nu\mathcal{W}_{\varepsilon, c}$, and is a consequence of Danskin's Theorem, an envelope theorem for optimization problems. It provides a powerful result for computing the gradient of an optimal value with respect to a parameter, such as the support point location $y_j$. The theorem states that the gradient is simply the integral of the cost function's gradient, $\nabla_{y_j} c(x, y_j)$, evaluated under the optimal transport plan $\pi$. The indirect dependence of the optimal plan itself on $y_j$ does not contribute to the final expression. For the specific quadratic cost $c(x,y_j) = \frac{1}{2}\|x-y_j\|^2$, the cost gradient is $\nabla_{y_j} c(x,y_j) = y_j - x$. Setting the full barycenter gradient to zero yields the optimality condition in Eq.~\eqref{eq:barycenter_fixed_point}. The Lloyd's algorithm can thus be seen as an iterative method that directly seeks to satisfy this first-order optimality condition. While gradient descent is more general and applicable to any differentiable cost function, the Lloyd-type approach is often preferred for the quadratic case as it avoids the need to tune the learning rate.
A comparison between the two methods is presented in \href{https://github.com/SemiDiscreteOT/SemiDiscreteOT/tree/master/examples/tutorial_4}{Tutorial 4} of the \texttt{SemiDiscreteOT} library.
\end{remark}

\begin{algorithm}[ht]
    \caption{Wasserstein Barycenters}
    \label{alg:wasserstein_barycenters}
    \begin{algorithmic}[1]
    \Require Source meshes $\{\calT_h\}_{i=1}^{K}$, measures $\{\mu_{i}\}_{i=1}^{K}$ with densities $\{\rho_{h,i}\}_{i=1}^{K}$, quadrature $\{\{x_{q,i}, w_{q,i}\}_{q=1}^{N_{q,i}}\}_{i=1}^{K}$, initial guess $\mathbf{y}^{t=0}=\{y^{t=0}_j\}_{j=1}^N, \nu^{t=0}=\{\nu^{t=0}_j\}_{j=1}^N$, cost function $c(\cdot, \cdot)$, $\varepsilon > 0$, tolerances $\delta_{\text{tol}, \nu}$, $\delta_{\text{tol}, y}$, Wasserstein barycenter weights $\{\lambda_i\}_{i=1}^{K}$, steps $\alpha,\beta$.
    \Ensure Wasserstein barycenter described as discrete measure $\{y^*_j, \nu^*_j\}_{j=1}^N$.

    \Repeat (While $\lVert\delta\nu\rVert_2 > \delta_{\text{tol}, \nu}\lVert\nu\rVert_2$,  $\lVert\delta \mathbf{y}\rVert_2 > \delta_{\text{tol}, \mathbf{y}}\lVert\mathbf{y}\rVert_2$ and $t<T_{\text{max iterations}}$)
        \ForAll{$i$ in K}
            Solve regularized semi-discrete optimal transport between the source measures $\{\mu_{i}\}_{i=1}^{K}$ and the target $(\nu^t, \mathbf{y}^t)$, minimizing $\{J^h_{\varepsilon}(\bpsi_i)\}_{i=1}^K$ with L-BFGS.
        \EndFor
        \State Compute gradient of discrete weighted Wasserstein distances: $\delta\nu\leftarrow \sum_{i=1}^K\lambda_i \nabla_\nu\mathcal{W}_{\varepsilon, c}(\mu_i, \nu)$.
        \State Positivity preserving update: $\nu \leftarrow \nu^t\odot\exp(-\alpha \nu^t\odot\delta \nu^t)$.
        \State Normalization: $\nu^{t+1} \leftarrow \nu^t \oslash \lVert\nu^t\rVert_1$.
        \State Compute gradient of discrete weighted Wasserstein distances: $\delta\mathbf{y}\leftarrow \sum_{i=1}^K\lambda_i \nabla_{\mathbf{y}}\mathcal{W}_{\varepsilon, c}(\mu_i, \nu)$.
        \State Update support points: $\mathbf{y}^{t+1}\leftarrow\mathbf{y}^t-\beta\delta\mathbf{y}$.
        \State t = t + 1.
    \Until{convergence to $(\mathbf{y}^*, \nu^*)$.}
    \end{algorithmic}
\end{algorithm}

\begin{remark}[Optimal Quantization as a Barycenter Problem]
The Lloyd's algorithm is also the standard method for the optimal quantization of a single continuous measure $\mu$, which can be viewed as a barycenter problem with a single source ($K=1$). In this setting, the location update step involves computing a conditional barycenter for each point. While this operation corresponds to a simple centroid for the squared Euclidean cost, it requires solving a more complex optimization problem for general non-Euclidean costs. We demonstrate this general case in detail in Section~\ref{subsec:blue_noise}, where we apply the framework to blue noise sampling on a sphere using a geodesic distance cost.
\end{remark}

\subsubsection{Computational Strategy and Results}

We apply this methodology to compute the barycenter between two source probability measures, $\mu_1$ and $\mu_2$. Following the uniform density scenario outlined in Section~\ref{subsec:benchmarking}, each measure is defined by a constant density over its respective 3D vascular geometry domain, $\Omega_1$ and $\Omega_2$, such that $\dd\mu_i(x) = (1/\text{Vol}(\Omega_i)) \dd x$ for $i=1,2$. The domains, visualized in Figure~\ref{fig:source_geometries}, are discretized by finite element meshes $\mathcal{T}_h^1$ and $\mathcal{T}_h^2$ (see Table~\ref{tab:source_meshes} for details). We aim to compute a discrete barycenter supported on $N_b = \num{10000}$ points.

\begin{table}[ht]
    \centering
    \begin{tabular}{lccc}
        \hline
        \textbf{Source Mesh} & \textbf{Volume} & \textbf{Cells} & \textbf{DoFs ($P_1$)} \\
        \hline
        $\mathcal{T}_h^1$ & $\approx \SI{8.7e-6}{\cubic\meter}$ & \num{486904} & \num{83208} \\
        $\mathcal{T}_h^2$ & $\approx \SI{3.7e-5}{\cubic\meter}$ & \num{442502} & \num{74668} \\
        \hline
    \end{tabular}
    \caption{Geometric characteristics of the finest-level source meshes, $\mathcal{T}_h^1$ and $\mathcal{T}_h^2$, used to discretize the domains $\Omega_1$ and $\Omega_2$. The number of Degrees of Freedom (DoFs) corresponds to a $P_1$ finite element discretization.}
    \label{tab:source_meshes}
\end{table}

\begin{table}[htbp]
    \centering
    \begin{tabular}{c r r c r r}
        \toprule
        Level & \multicolumn{2}{c}{Source Mesh $\mathcal{T}_h^1$} & \phantom{abc} & \multicolumn{2}{c}{Source Mesh $\mathcal{T}_h^2$} \\
        \cmidrule(lr){2-3} \cmidrule(lr){5-6}
        Index ($l$) & Cells & DoFs ($P_1$) && Cells & DoFs ($P_1$) \\
        \midrule
        3 (Coarsest) & \num{915} & \num{347} && \num{1276} & \num{392} \\
        2 & \num{6932} & \num{1803} && \num{8272} & \num{2068} \\
        1 & \num{46499} & \num{10049} && \num{57107} & \num{11422} \\
        0 (Finest) & \num{486904} & \num{83208} && \num{442502} & \num{74668} \\
        \bottomrule
    \end{tabular}
    \caption{Details of the 4-level mesh hierarchies for the two source geometries, $\mathcal{T}_h^1$ and $\mathcal{T}_h^2$.}
    \label{tab:barycenter_hierarchies}
\end{table}

\begin{figure}[htbp]
    \centering
    \begin{subfigure}[b]{0.48\textwidth}
        \centering
        \includegraphics[height=4cm,trim={10cm 1cm 10cm 1cm},clip]{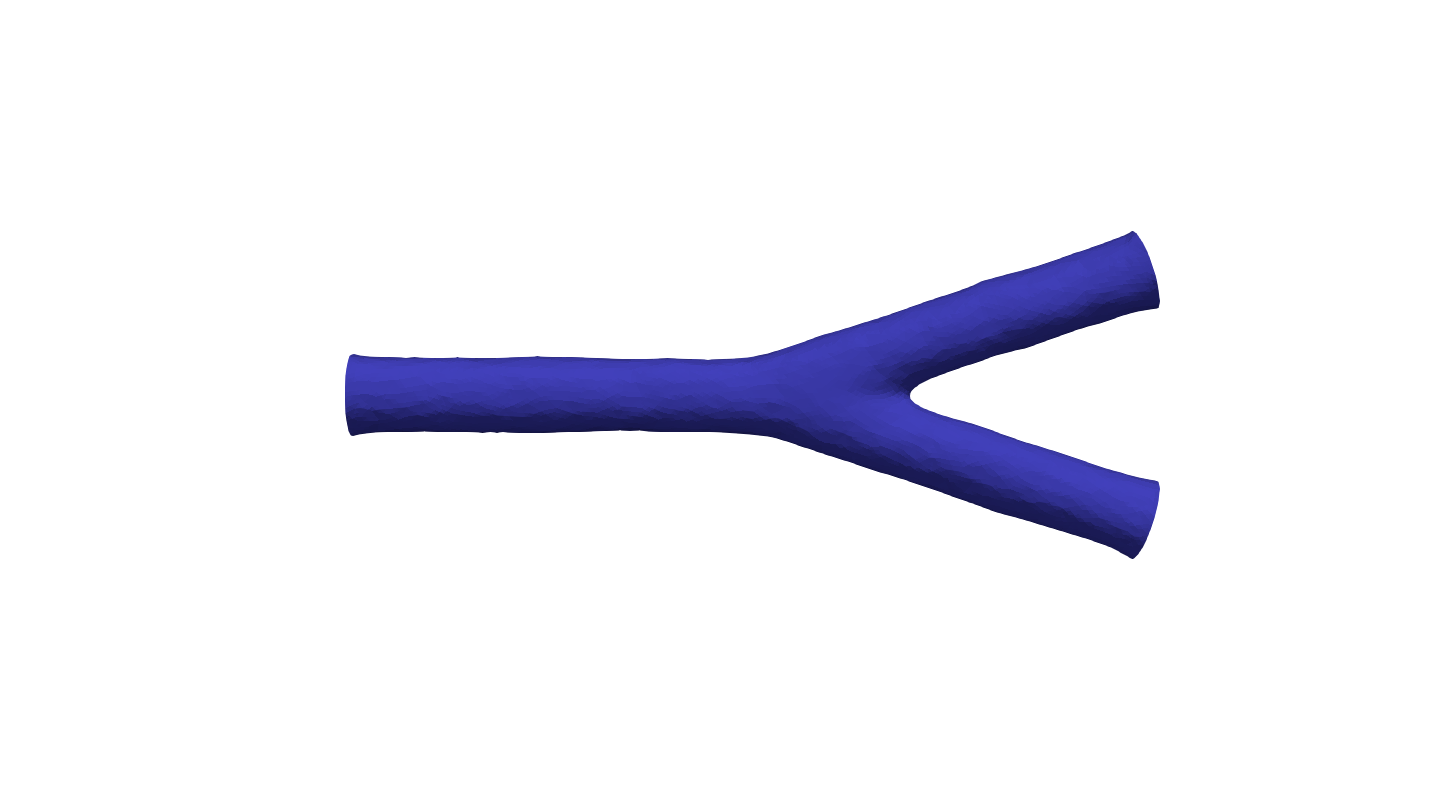}
        \caption{Source $\mu_1$: 2D projection ($xy$ plane)}
    \end{subfigure}
    \hfill
    \begin{subfigure}[b]{0.48\textwidth}
        \centering
        \includegraphics[height=4cm,trim={15cm 1cm 10cm 5cm},clip]{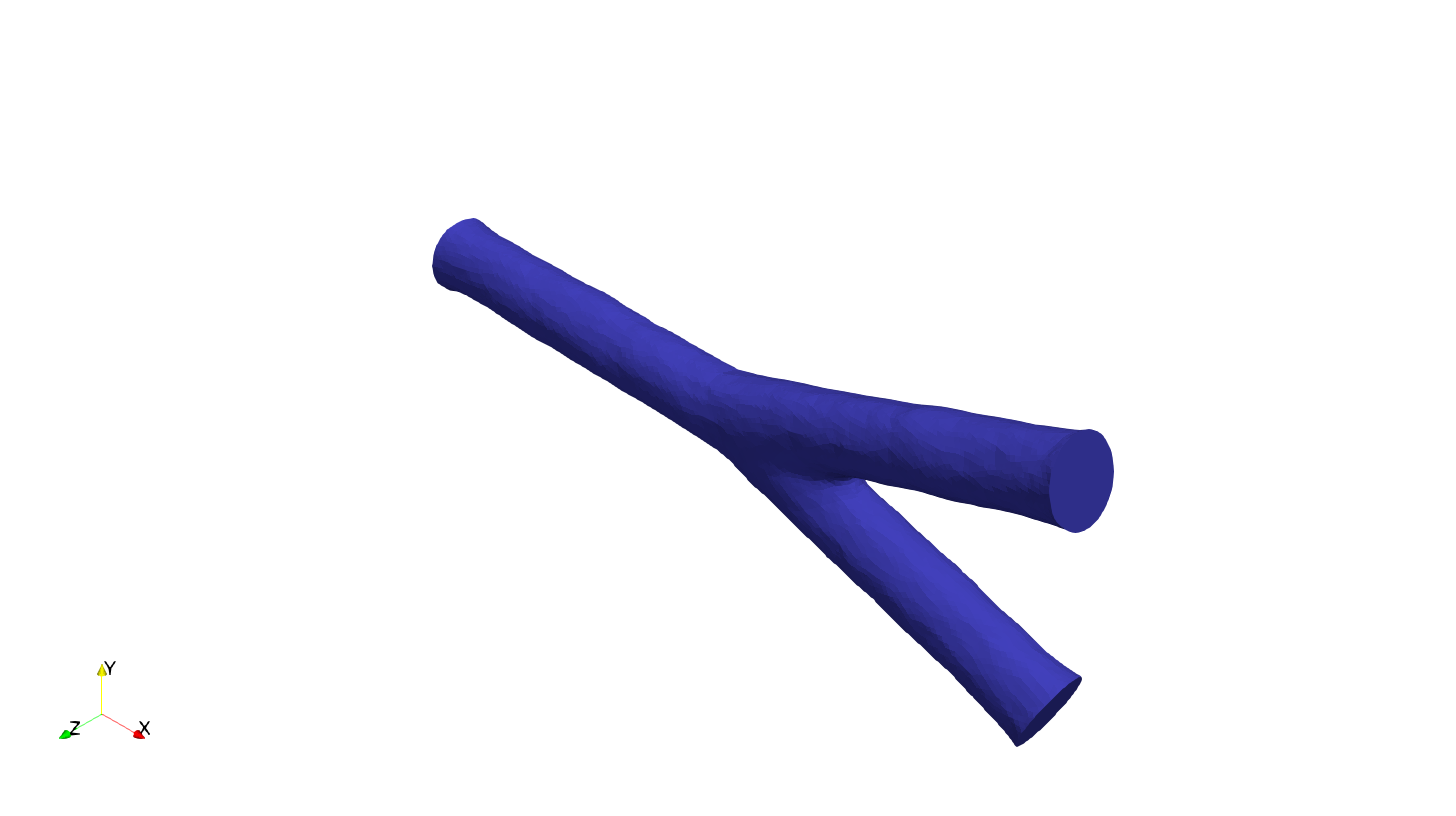}
        \caption{Source $\mu_1$: 3D visualization}
    \end{subfigure}
    \vspace{0.5cm}
    \begin{subfigure}[b]{0.48\textwidth}
        \centering
        \includegraphics[height=4cm,trim={10cm 1cm 10cm 1cm},clip]{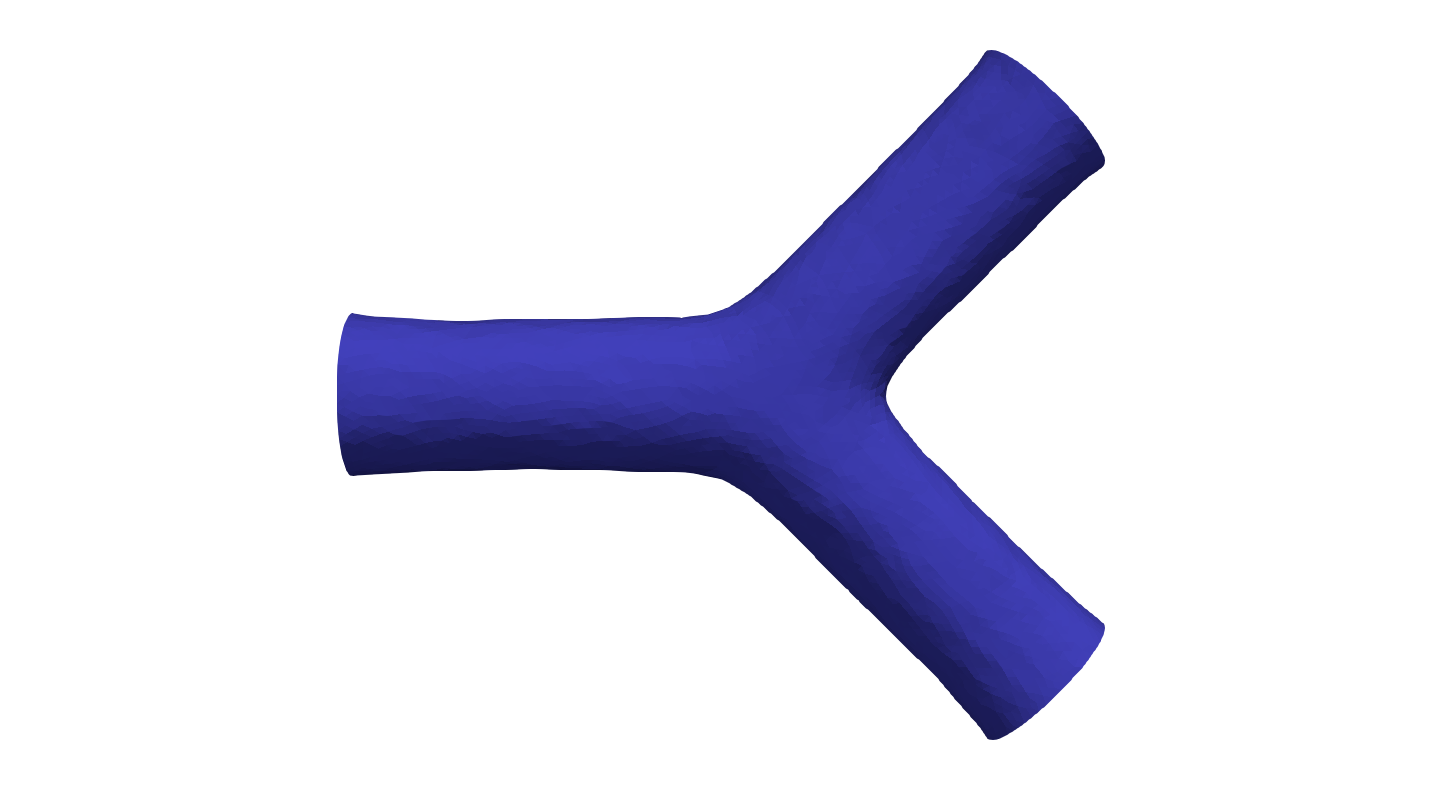}
        \caption{Source $\mu_2$: 2D projection ($xy$ plane)}
    \end{subfigure}
    \hfill
    \begin{subfigure}[b]{0.48\textwidth}
        \centering
        \includegraphics[height=4cm,trim={15cm 1cm 10cm 5cm},clip]{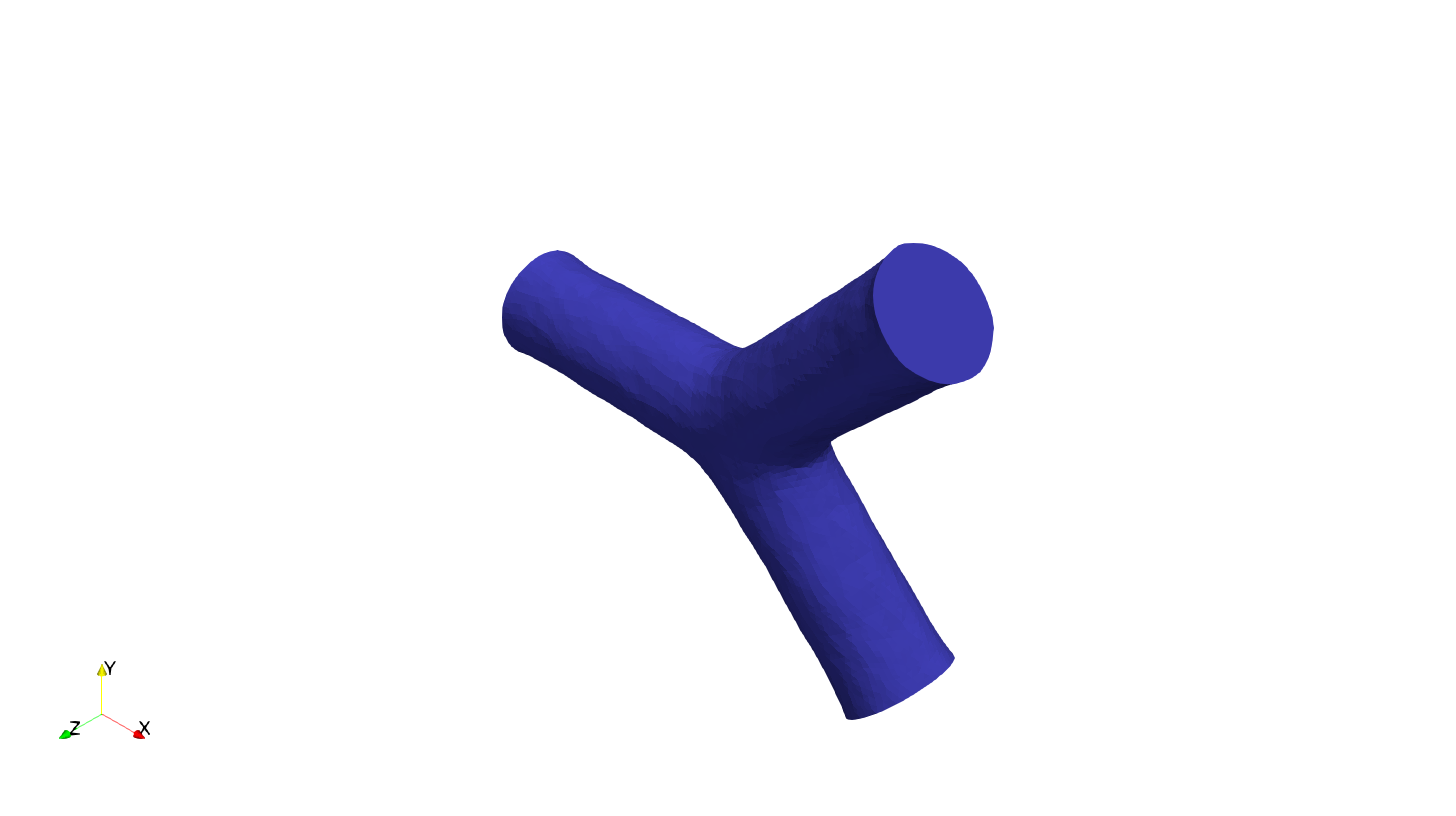}
        \caption{Source $\mu_2$: 3D visualization}
    \end{subfigure}
    \caption{Visualization of the two vascular junction geometries, corresponding to measures $\mu_1$ and $\mu_2$, used for the barycenter computation.}
    \label{fig:source_geometries}
\end{figure}

Our computational strategy is a two-stage hierarchical process designed for both efficiency and accuracy. We first perform a coarse initialization, which begins by sampling the $N_b$ barycenter support points from the domain $\Omega_1$. A barycenter is then computed using these points and only the coarsest mesh representations from each source hierarchy (level 3 in Table~\ref{tab:barycenter_hierarchies}). This stage, run to convergence using the damped Lloyd iteration (Eq.~\eqref{eq:damped_lloyd} with $\theta = 0.5$), rapidly determines a robust initial guess for the barycenter's overall geometric structure at low cost. This coarse solution then serves as a warm start for the second stage: high-resolution refinement. In each step of this main Lloyd's iteration, we must solve the two RSOT problems, $\mu_1 \to \nu^{(t)}$ and $\mu_2 \to \nu^{(t)}$, for the current barycenter estimate $\nu^{(t)}$.

For these inner solves, we employ the \emph{source-side multilevel strategy} (Section~\ref{sec:source_coarsening}). While the benchmark analysis in Section~\ref{subsubsec:multilevel_performance} showed that target-side or combined multilevel methods can offer superior performance, they are ill-suited for this barycenter problem.  The reason is that the barycenter support points $\{y_j\}_{j=1}^{N_b}$ are optimization variables that move in each Lloyd's step, a target-side approach would require continuous, costly re-clustering and hierarchy rebuilding. More importantly, our chosen approach confers a decisive advantage: we use the optimal dual potential vector $\bpsi^{(t)}$ from the RSOT solve at iteration $t$ as a warm start for the solve at iteration $t+1$, which substantially reduces inner solver iterations. This warm-start mechanism is invalidated by a changing target hierarchy, rendering the source-side multilevel strategy the most logical and efficient choice for this problem.

Numerical parameters for this experiment were set as follows: regularization $\varepsilon = \SI{5e-6}{}$, a value corresponding to approximately 1\% of the characteristic cost derived from the smaller source volume ($V_{\text{min}}^{2/d}$ for $d=3$); geometric truncation with relative error tolerance $\tau = \SI{1e-5}{}$; and an L-BFGS tolerance of $\delta_{\text{tol}} = \SI{1e-5}{}$ for the inner RSOT solves. The outer Lloyd's iteration was terminated when the Root Mean Square (RMS) Point Movement between iterations fell below $\SI{1e-5}{}$:
\begin{equation}
\label{eq:rms_movement}
\text{RMS}(t) = \sqrt{\frac{1}{N_b} \sum_{j=1}^{N_b} \|y_j^{(t+1)} - y_j^{(t)}\|^2} \le \SI{1e-5}{}.
\end{equation}

\begin{figure}[htbp]
    \centering
    \includegraphics[width=0.9\textwidth]{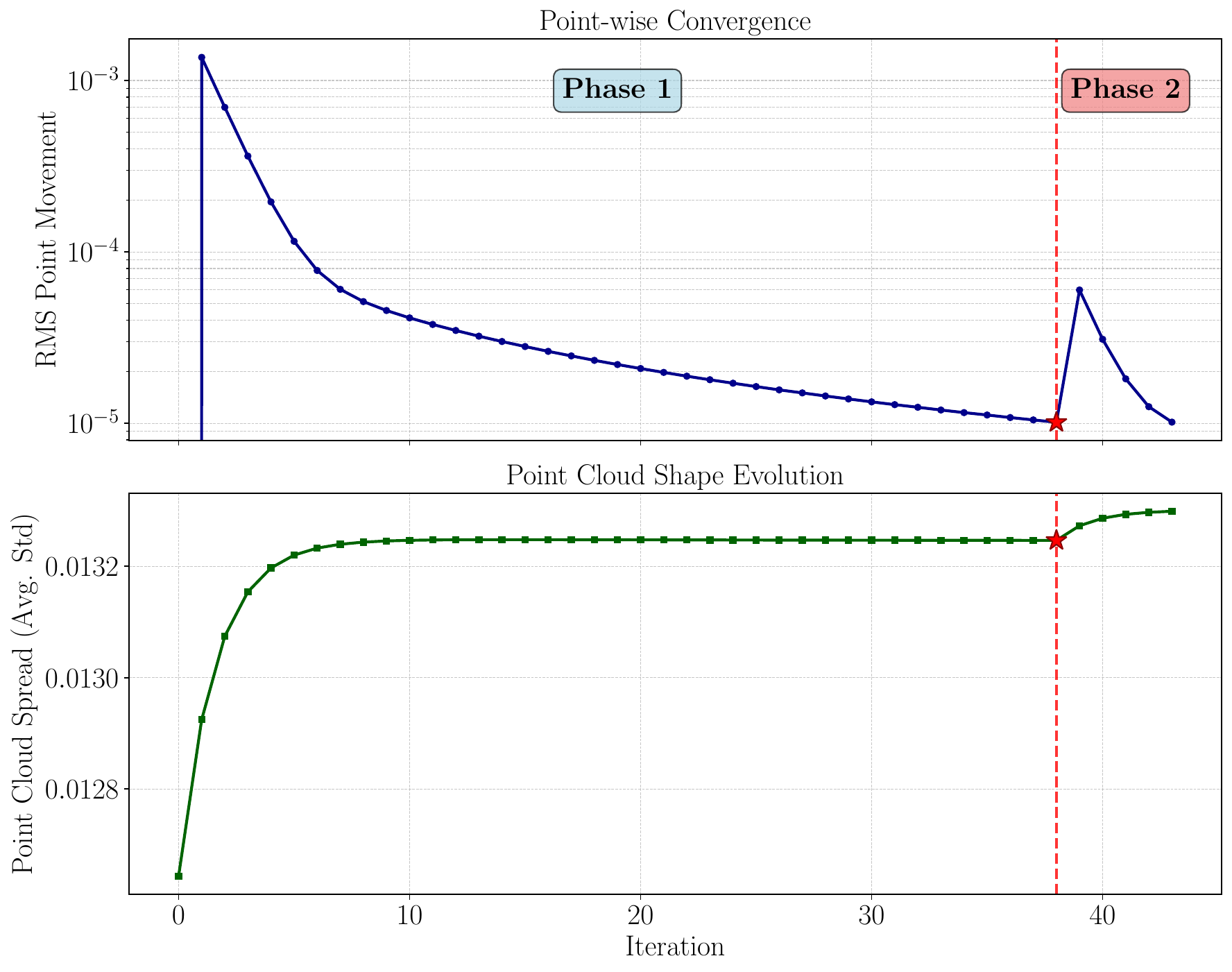}
    \caption{Convergence analysis of the two-stage barycenter computation. The top panel plots the RMS Point Movement (Eq.~\eqref{eq:rms_movement}) on a log scale, showing pointwise convergence. The bottom panel tracks the Point Cloud Spread (average standard deviation of coordinates), indicating the evolution of the barycenter's overall shape. The red dashed line and star marker indicate the transition from the coarse initialization phase (Phase 1) to the high-resolution refinement phase (Phase 2).}
    \label{fig:barycenter_convergence}
\end{figure}

The convergence of this two-stage process is detailed in Figure~\ref{fig:barycenter_convergence}. The top panel shows that in Phase 1, the RMS point movement starts high and decreases rapidly as the algorithm quickly finds the coarse geometric structure of the barycenter. The spike at the transition to Phase 2 is expected, as the algorithm begins its first high-resolution update from the coarse solution. In Phase 2, the RMS movement continues to decrease smoothly, indicating stable refinement of the point positions. The bottom panel shows that the overall shape of the point cloud, measured by its spatial spread, stabilizes very quickly during the coarse phase and remains consistent throughout the fine-tuning phase.

\begin{figure}[htbp]
    \centering
    \includegraphics[width=\textwidth]{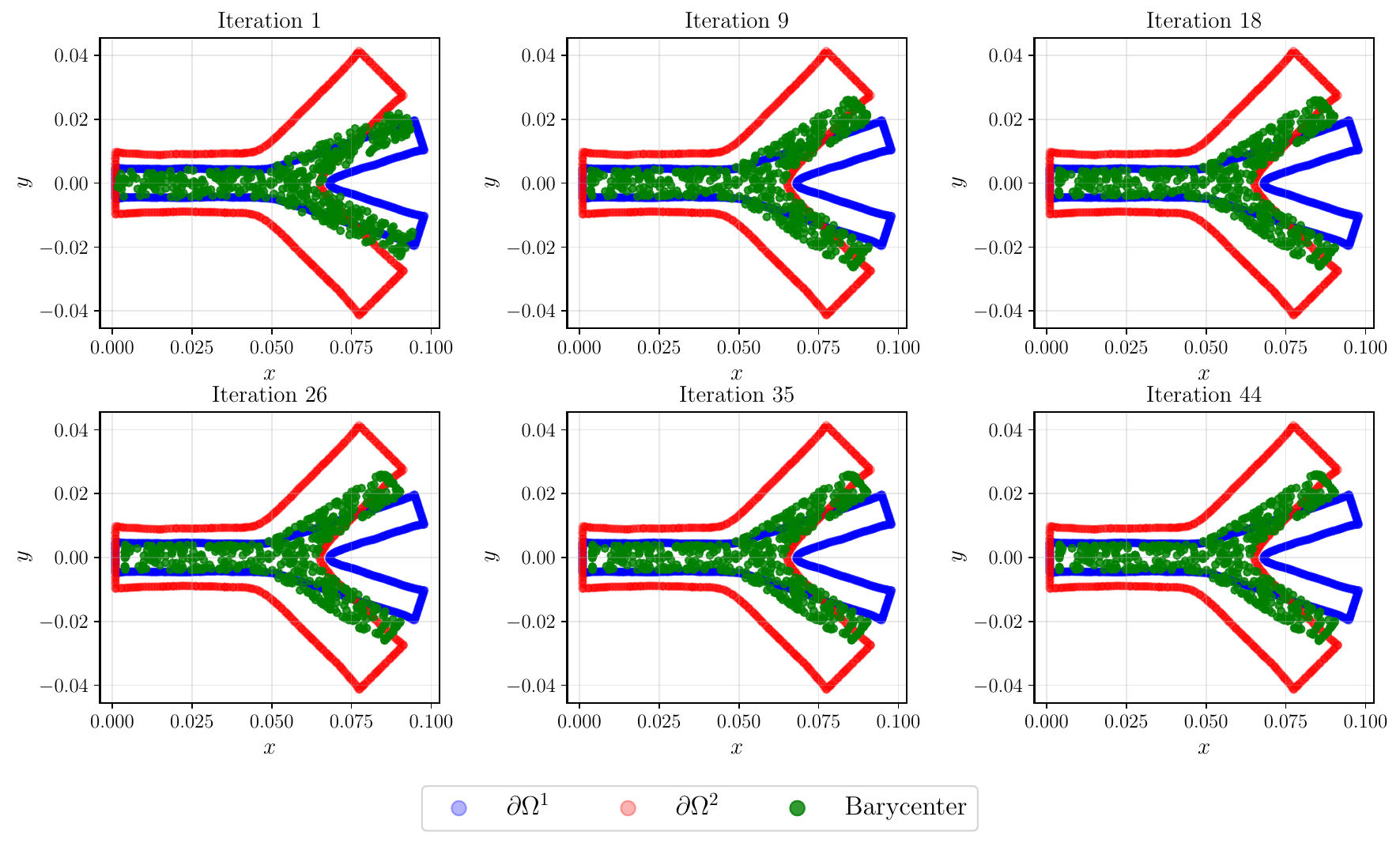}
    \caption{Evolution of the barycenter point cloud at selected iterations, projected onto the $xy$ plane (showing 500 randomly sampled points for visualization clarity from the full \num{10000} point cloud). The plots show the initial state followed by progressive refinement. The final iteration (44) shows the converged high-resolution barycenter, which smoothly interpolates the geometry of the two source measures ($\mu_1$ and $\mu_2$).}
    \label{fig:barycenter_snapshots}
\end{figure}

Figure~\ref{fig:barycenter_snapshots} provides a visual representation of the barycenter's evolution. The snapshots illustrate how the point cloud rapidly organizes from its initial configuration. The subsequent iterations refine the density and position of the points to form a smooth and well-defined geometric mean of the two sources. The final state (Iteration 44) represents the converged high-resolution barycenter, successfully capturing the interpolated geometry between the source structures (supports of $\mu_1$ and $\mu_2$) shown in Figure~\ref{fig:source_geometries}.

This experiment highlights a powerful application of our framework. By nesting the efficient, multilevel RSOT solver within a Lloyd's iteration, we can tackle computationally demanding problems like Wasserstein barycenter computation for complex 3D geometries. The two-stage strategy, combining a cheap coarse initialization with an accelerated high-resolution refinement, proves to be a highly effective and robust approach. %
\subsection{Application: Shape Registration}
\label{subsec:registration_application}

Beyond barycenter computation, we demonstrate the utility of our RSOT framework for the registration of complex 3D geometries, a fundamental task in medical imaging and computational anatomy~\cite{romor2025dataassimilationperformedrobust}. We consider the problem of finding a smooth, physically meaningful mapping between the two vascular junction geometries introduced in Section \ref{subsec:barycenter_application}. Specifically, we transport the measure $\mu_1$, defined on the domain $\Omega_1$, to a discrete measure $\nu_2$, supported on the vertices of the mesh discretizing domain $\Omega_2$. This experiment showcases the detailed analysis possible with the computed transport plan and highlights the benefits of the multilevel and acceleration strategies in a practical, large-scale setting.

\subsubsection{Problem Setup}

The registration problem is formulated using uniform measures to drive the geometric mapping. The source measure $\mu_1$ is defined by a constant density over the domain $\Omega_1$ (discretized by the mesh $\mathcal{T}_h^1$), and the target measure $\nu_2$ is a discrete measure supported on the $N_2 = \num{74668}$ vertices of the mesh $\mathcal{T}_h^2$ (discretizing $\Omega_2$), with uniform weights. To handle the geometric complexity and the large number of points, we employ the combined multilevel approach (Section \ref{subsec:combined_multilevel}). The source mesh hierarchy for $\mathcal{T}_h^1$ is identical to that used in the barycenter experiment (Table \ref{tab:barycenter_hierarchies}), while the target measure hierarchy for the support points of $\nu_2$ is constructed via k-means clustering. The details of both hierarchies are provided in Table \ref{tab:registration_hierarchies}.

The RSOT problem was solved with a regularization parameter of $\varepsilon = \SI{5e-6}{}$ and a geometric truncation tolerance of $\tau = \SI{1e-5}{}$ for the accelerated dual evaluation.

\begin{table}[htbp]
    \centering
    \caption{Hierarchies used for the vascular registration problem.}
    \label{tab:registration_hierarchies}
    \begin{tabular}{lrrrr}
        \toprule
        \textbf{Hierarchy Level} & \multicolumn{2}{c}{\textbf{Source Mesh ($\mathcal{T}_h^1$)}} & \multicolumn{2}{c}{\textbf{Target Points (from $\mathcal{T}_h^2$)}} \\
        \cmidrule(lr){2-3} \cmidrule(lr){4-5}
         & Vertices & Elements & \multicolumn{2}{r}{Points ($N_l$)} \\
        \midrule
        3 (Coarsest) & 347 & \num{1535} & \multicolumn{2}{r}{\num{2500}} \\
        2 & \num{1803} & \num{9428} & \multicolumn{2}{r}{\num{10000}} \\
        1 & \num{10049} & \num{56495} & \multicolumn{2}{r}{\num{40000}} \\
        0 (Finest) & \num{83208} & \num{486904} & \multicolumn{2}{r}{\num{74668}} \\
        \bottomrule
    \end{tabular}
\end{table}

\subsubsection{Analysis of the Transport Plan: Conditional Densities}

The optimal transport plan $\pi_\varepsilon^*$ provides rich information about the correspondence between the source and target geometries. As discussed in Remark \ref{rem:conditional_expectations}, the plan can be interpreted as the joint law of a pair of random variables $(X,Y)$ on $\Omega_1 \times \Y_2$. A key quantity for analysis is the conditional probability density of the source location $X$ given that mass is transported to a specific target point $Y=y_j$\footnote{For a detailed tutorial on conditional densities in the context of semi-discrete optimal transport, see \url{https://github.com/SemiDiscreteOT/SemiDiscreteOT/tree/master/examples/tutorial_2}.}. This density, which we denote $\mu_j(x)$, is obtained via Bayes' theorem:
\begin{equation}
\label{eq:conditional_density}
\mu_j(x) \defeq \frac{\dd\pi_j(x)}{\int_{\Omega_1} \dd\pi_j(x')} = \frac{p_j(x|\bpsi^*) \dd\mu_1(x)}{\nu_{2,j}},
\end{equation}
where $p_j(x|\bpsi^*)$ is the unnormalized plan density from Eq.~\eqref{eq:final_plan_density} and the equality follows from the target marginal constraint $\int_{\Omega_1} p_j(x|\bpsi^*) \dd\mu_1(x) = \nu_{2,j}$. By construction, $\mu_j(x)$ is a probability measure on $\Omega_1$ for each target point $j$.

Figure \ref{fig:reg_cond_density} visualizes this conditional density $\mu_j(x)$ for four distinct points on the target geometry: one on the axis before the bifurcation, one at the bifurcation, and one in each of the two outflow branches. Each subplot shows the source geometry colored by the density of mass that is transported to the indicated target point.

\begin{figure}[htbp]
    \centering
    \begin{subfigure}[b]{0.49\textwidth}
        \centering
        \includegraphics[width=\textwidth,trim={12cm 1cm 12cm 1cm},clip]{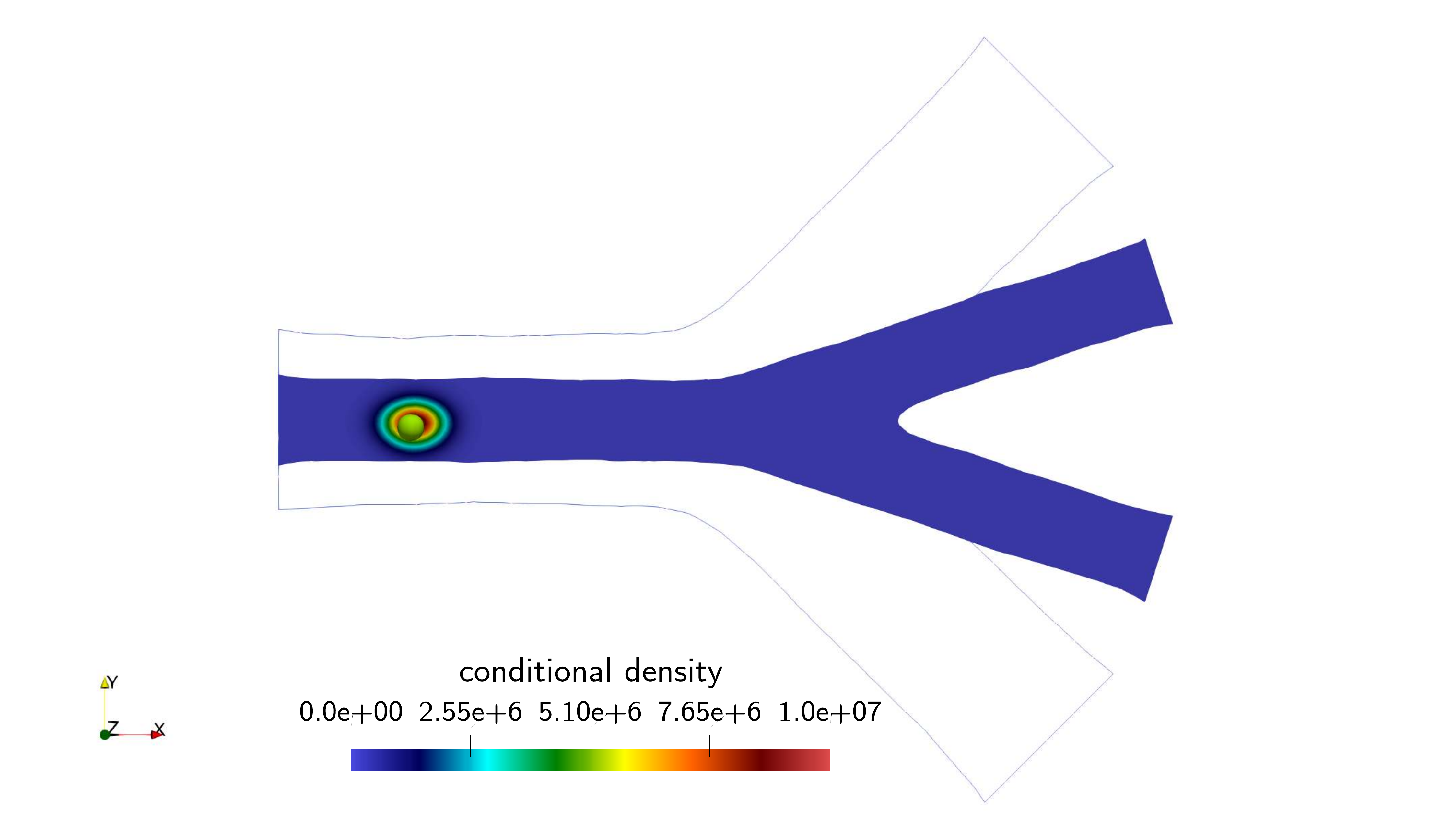}
        \caption{Target point on axis (before junction)}
    \end{subfigure}
    \hfill
    \begin{subfigure}[b]{0.49\textwidth}
        \centering
        \includegraphics[width=\textwidth,trim={12cm 1cm 12cm 1cm},clip]{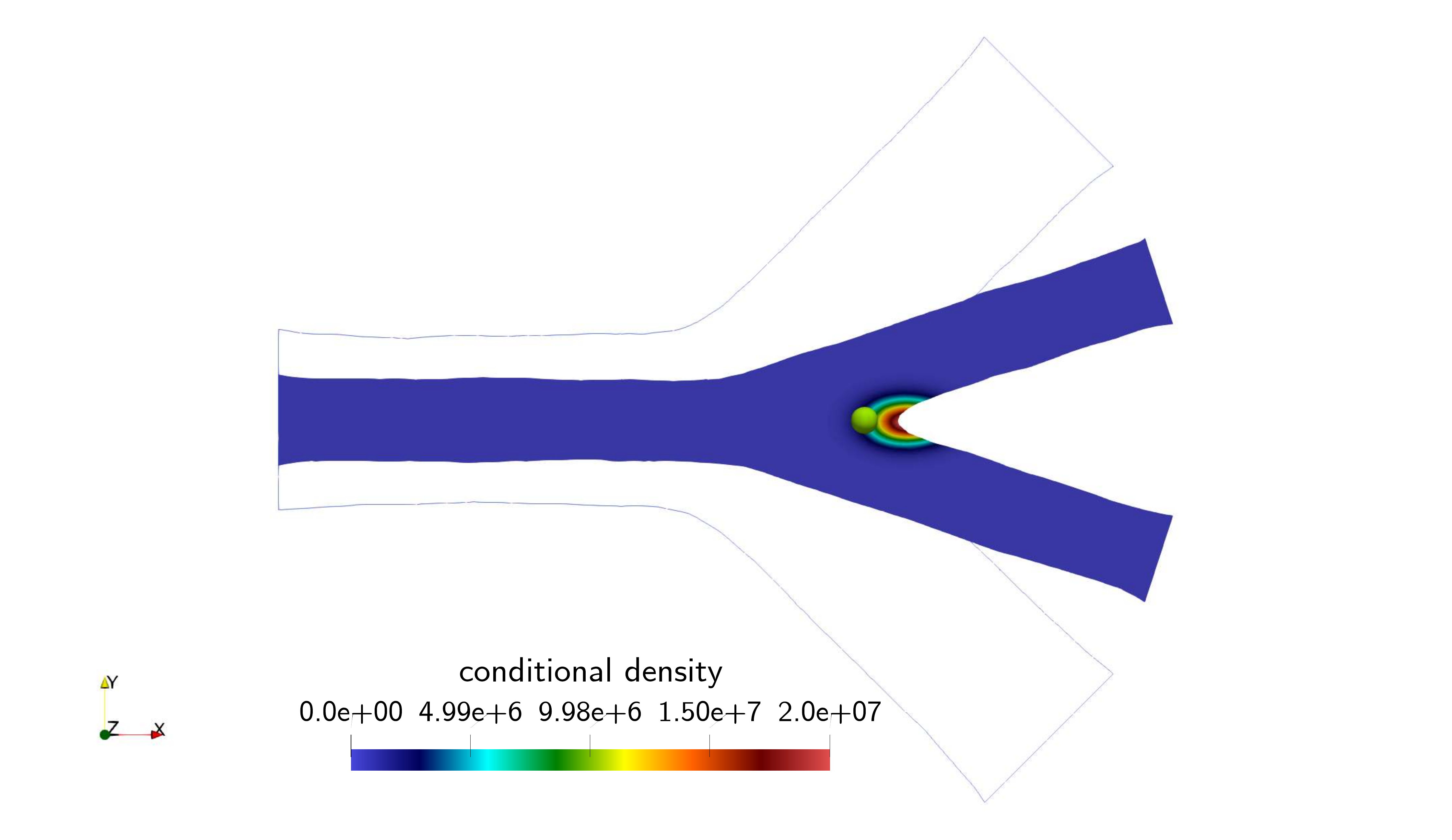}
        \caption{Target point at junction}
    \end{subfigure}
    \vspace{0.2cm}
    \begin{subfigure}[b]{0.49\textwidth}
        \centering
        \includegraphics[width=\textwidth,trim={12cm 1cm 12cm 1cm},clip]{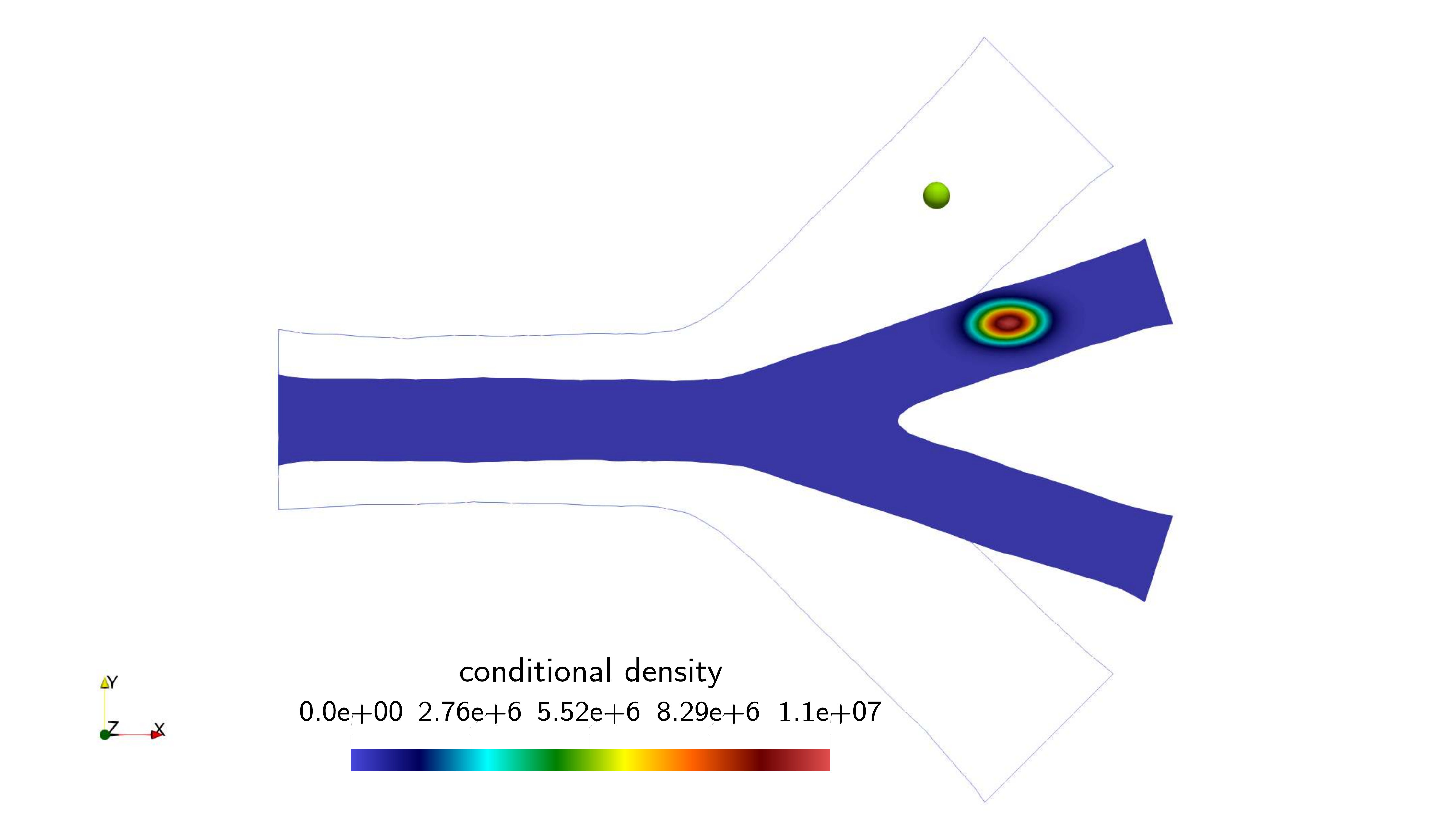}
        \caption{Target point on upper branch}
    \end{subfigure}
    \hfill
    \begin{subfigure}[b]{0.49\textwidth}
        \centering
        \includegraphics[width=\textwidth,trim={12cm 1cm 12cm 1cm},clip]{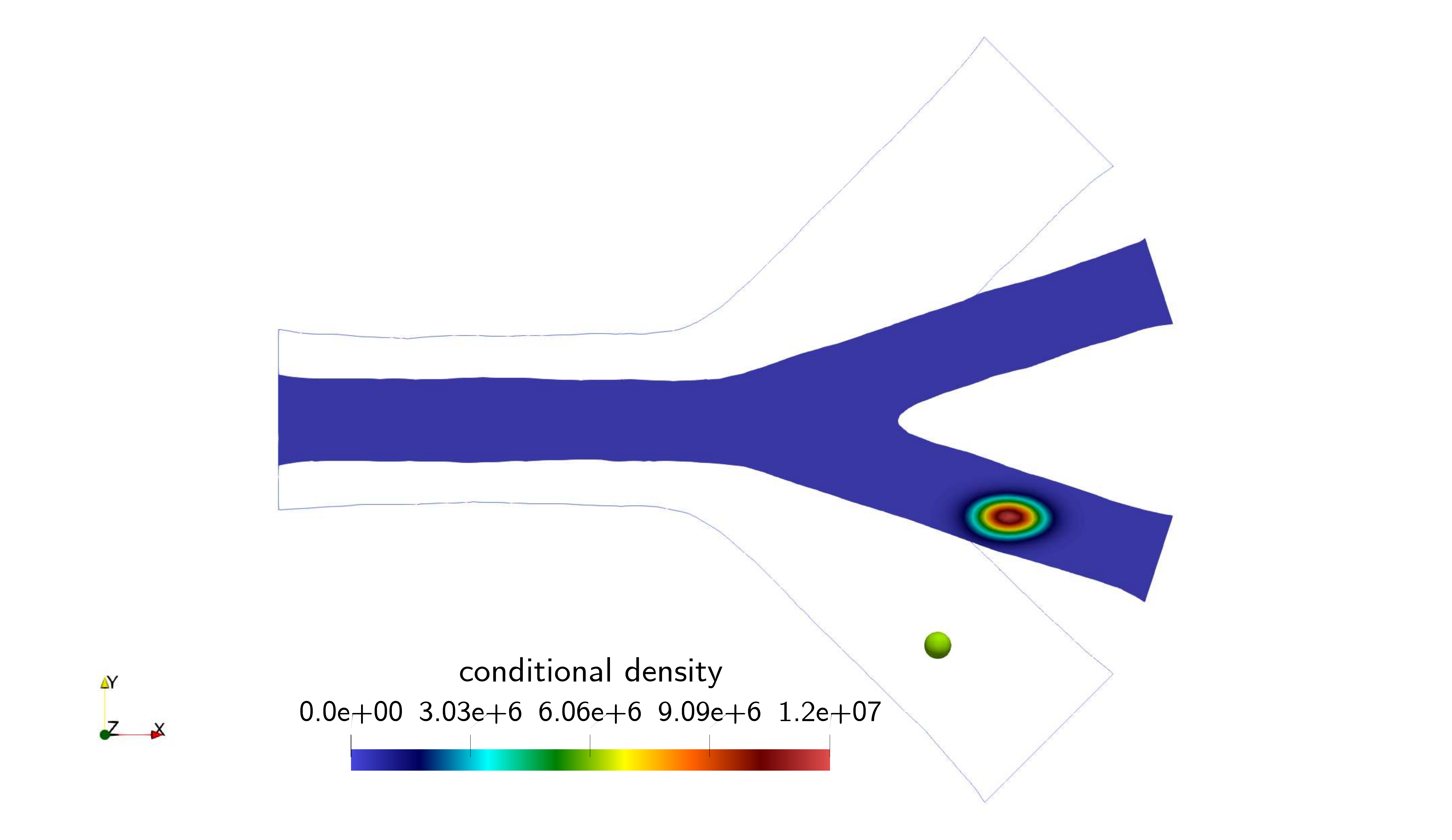}
        \caption{Target point on lower branch}
    \end{subfigure}
    \caption{Conditional density of the source measure for four selected target points. The color map (blue to red) indicates low to high density, showing which parts of the source geometry contribute mass to each specific target location (represented by the green dots). The plots demonstrate that the transport plan correctly identifies homologous regions.}
    \label{fig:reg_cond_density}
\end{figure}

The results clearly demonstrate that the transport plan captures the geometric and topological correspondence between the two structures. Mass arriving at the target's axis (a) originates almost exclusively from the source's axis. Similarly, the target's junction (b) is fed by the source's junction, and the upper (c) and lower (d) branches are fed by their respective counterparts in the source geometry. The concentrated, high-density (red) regions confirm that for a small $\varepsilon$, the transport plan is highly localized and not overly diffuse, correctly mapping homologous anatomical features.

\subsubsection{Transport Map Recovery and Displacement Analysis}

Once the optimization process converges to the optimal potentials $\bpsi^*$, the resulting plan density $p_j(x|\bpsi^*)$ from Eq.~\eqref{eq:final_plan_density} describes a probabilistic mapping from each source point $x$ to the set of target points $\Y_2$. To obtain a deterministic transport map $T: \Omega_1 \to \R^3$, we must summarize this conditional distribution for each $x$. We consider two canonical choices based on statistical estimators: the conditional expectation (mean) and the conditional mode (maximum likelihood).

The first map is the \emph{barycentric map}, $T_{\text{bary}}(x)$, which computes the expected target location for a source point $x$. This is analogous to finding the center of mass of the target points, weighted by their conditional probabilities:
\begin{equation}
\label{eq:barycentric_map_registration}
T_{\text{bary}}(x) \defeq \mathbb{E}[Y | X=x] = \sum_{j=1}^{N_2} p_j(x | \bpsi^*) \, y_{2,j} = \frac{\sum_{j=1}^{N_2} \nu_{2,j} \exp\left( (\psi^*_j - c(x, y_{2,j})) / \varepsilon \right) y_{2,j}}{\sum_{l=1}^{N_2} \nu_{2,l} \exp\left( (\psi^*_l - c(x, y_{2,l})) / \varepsilon \right)}.
\end{equation}
Because it is a weighted average over all target points, the barycentric map produces a smooth deformation field. The resulting point $T_{\text{bary}}(x)$ is a new point in $\R^3$ that does not necessarily belong to the original set of target vertices $\Y_2$.

The second map is the \emph{modal map}, $T_{\text{modal}}(x)$, which represents a maximum a posteriori estimate. It assigns each source point $x$ to the single, most probable target vertex $y_{2,j}$. This operation corresponds to finding the mode of the conditional distribution for each source point $x$:
\begin{equation}
\label{eq:modal_map_registration}
T_{\text{modal}}(x) \defeq y_{2, j^*(x)}, \quad \text{where} \quad j^*(x) = \argmax_{j \in \{1, \dots, N_2\}} p_j(x | \bpsi^*).
\end{equation}
Maximizing the plan density $p_j(x|\bpsi^*)$ is equivalent to maximizing the term $(\psi^*_j - c(x, y_{2,j}) + \varepsilon \ln \nu_{2,j})$. This partitions the source domain $\Omega_1$ into a Power diagram, where each cell consists of points mapping to the same target vertex. Unlike the barycentric map, the modal map is piecewise constant and its image is a subset of the original target vertices $\Y_2$.

\begin{remark}[Conditional Expectations in Barycenter and Registration Problems]
\label{rem:conditional_expectations}
There is a noteworthy duality between the computation of the barycentric map in this registration context and the update step for the Wasserstein barycenter in Section \ref{subsec:barycenter_application}. Let $(X, Y)$ be a pair of random variables on $\Omega_1 \times \Y_2$ (where $\Y_2$ is the support of $\nu_2$) with joint law given by the optimal plan $\pi_\varepsilon^*$ between $\mu_1$ and $\nu_2$. The barycentric map $T_{\text{bary}}(x)$ defined in Eq. \eqref{eq:barycentric_map_registration} is precisely the conditional expectation of the target location $Y$ given the source location $X=x$:
\[
T_{\text{bary}}(x) = \mathbb{E}_{\pi_\varepsilon^*}[Y | X=x].
\]
Conversely, the conditional barycenter $T_{i,j}$ used in the Lloyd's algorithm for the Wasserstein barycenter (Eq. \eqref{eq:barycenter_fixed_point}) is the conditional expectation of the source location $X$ given that the mass is transported to the target point $Y=y_j$:
\[
T_{i,j} = \mathbb{E}_{\pi_{i,\varepsilon}^*}[X | Y=y_j] = \frac{1}{\nu_j} \int_{\Omega_i} x \, \dd\pi_{i,j,\varepsilon}^*(x).
\]
Thus, while both applications leverage the structure of the optimal plan to compute centroids, they do so by conditioning on complementary variables, reflecting the different goals of finding a map versus finding a mean geometry.
\end{remark}

Figure \ref{fig:displacement_magnitude} shows the displacement magnitude, $\|T(x) - x\|$, for both maps. The barycentric map (a) yields a smooth displacement field, while the modal map (b) shows sharper transitions corresponding to the boundaries of the Power cells. Both maps correctly identify that the largest deformations occur at the ends of the branches, where the geometric differences between the source and target are most significant.

\begin{figure}[htbp]
    \centering
    \begin{subfigure}[b]{0.49\textwidth}
        \centering
        \includegraphics[width=\textwidth]{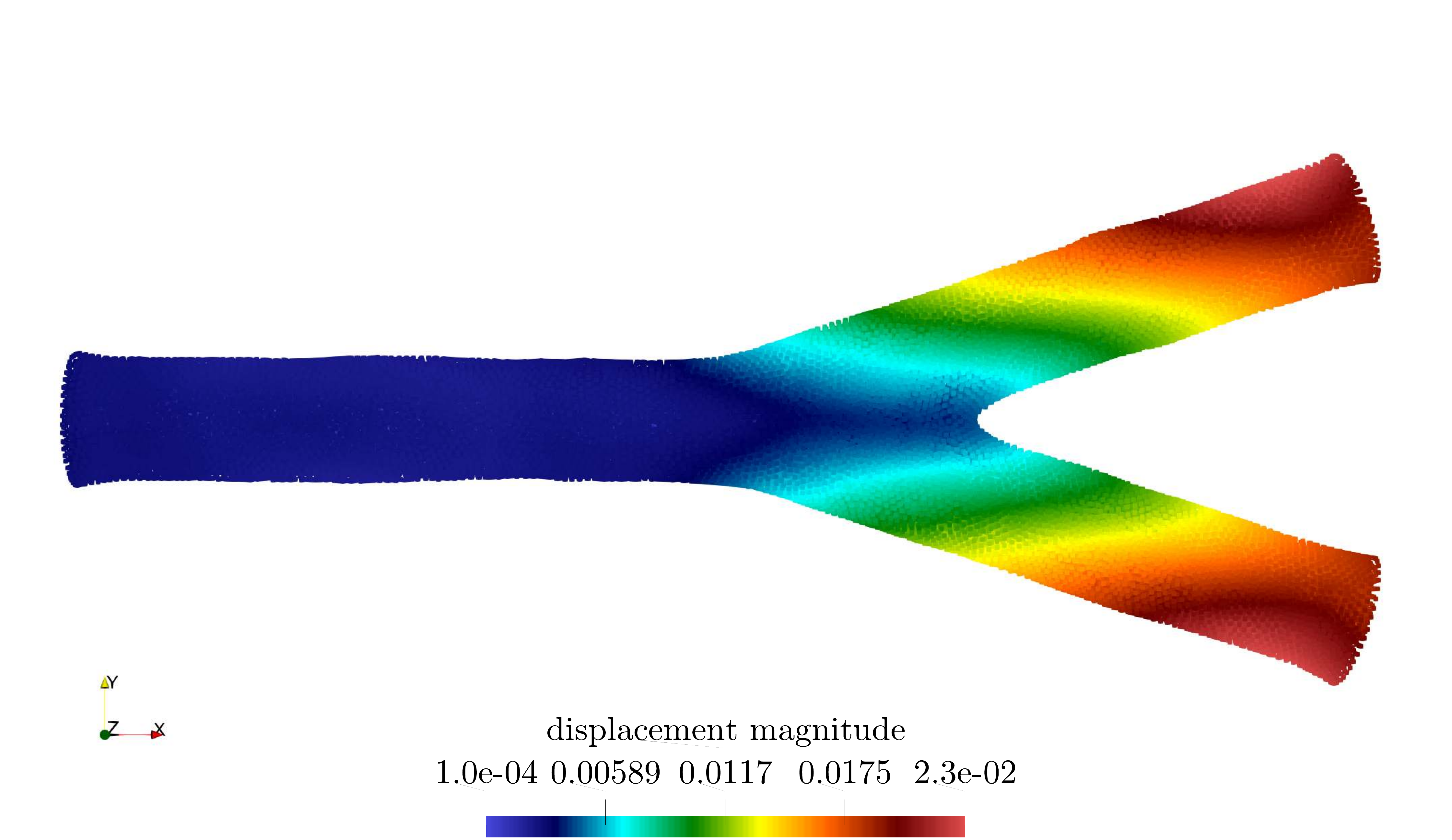}
        \caption{Displacement Magnitude (Barycentric Map)}
    \end{subfigure}
    \hfill
    \begin{subfigure}[b]{0.49\textwidth}
        \centering
        \includegraphics[width=\textwidth]{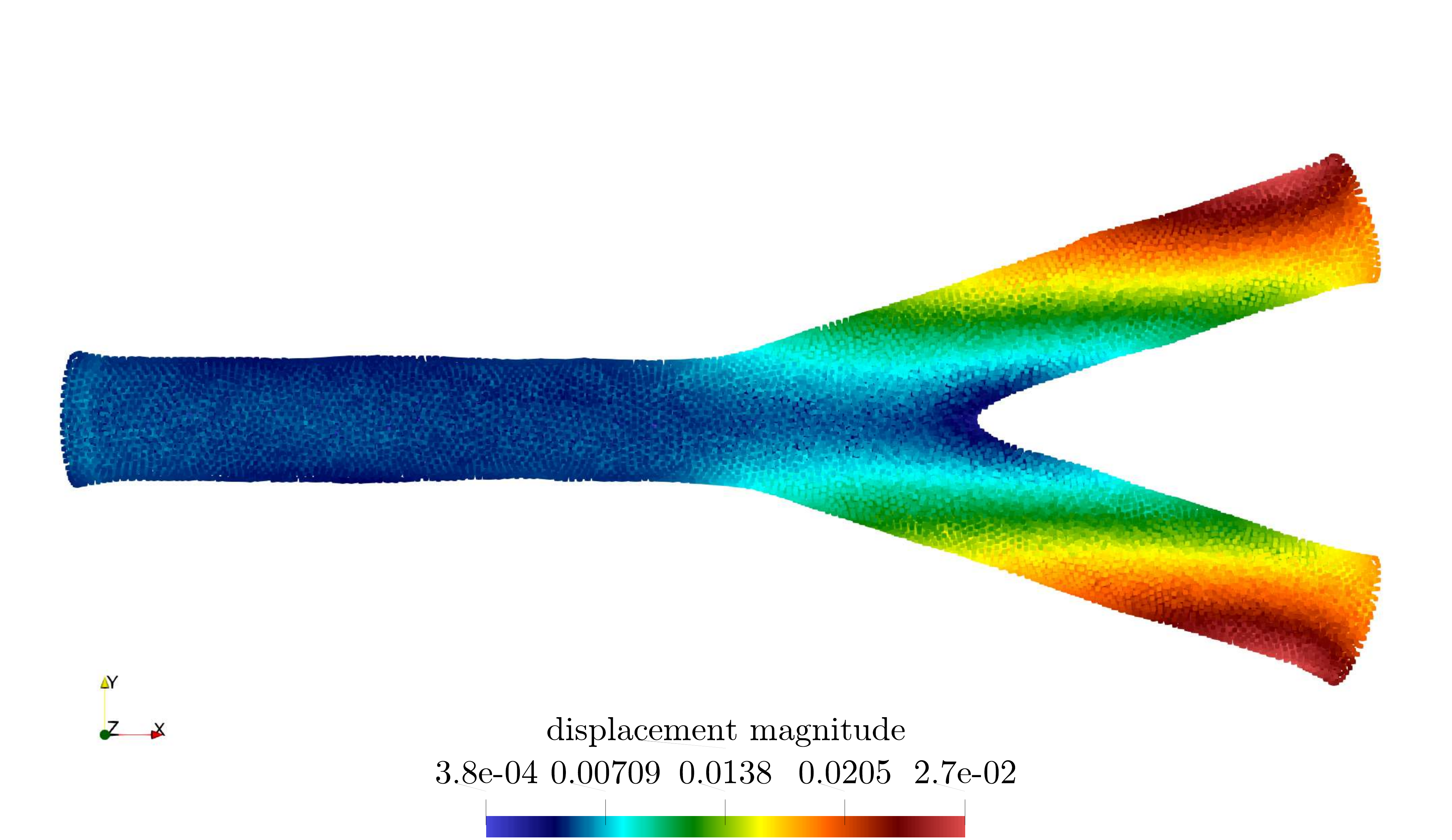}
        \caption{Displacement Magnitude (Modal Map)}
    \end{subfigure}
    \caption{Displacement magnitude $\|T(x)-x\|$ on the source geometry vertices for the barycentric and modal maps. The color scale indicates the magnitude of movement required to map the source geometry of $\Omega_1$ to the target.}
    \label{fig:displacement_magnitude}
\end{figure}

To visualize the deformation path, we perform a displacement interpolation, as proposed by McCann \cite{McCann1997}, between the original source points $x$ and their mapped positions $T_{\text{bary}}(x)$. For an interpolation parameter $\alpha \in [0,1]$, the interpolated point $x_\alpha$ is given by:
\begin{equation}
\label{eq:mccann_interpolation}
x_\alpha = (1-\alpha)x + \alpha T_{\text{bary}}(x).
\end{equation}
This concept is central to generating intermediate states along geodesic paths in the Wasserstein space and can been leveraged for data augmentation \cite{Khamlich2025}. Figure \ref{fig:transport_interpolation} shows snapshots of this process, confirming that the barycentric map provides a smooth, non-intersecting deformation from the source to the target geometry. The cloud of source points (green) progressively deforms to fill the shape defined by the target surface (red outline), providing a compelling visual validation of the registration quality.

\begin{figure}[htbp]
    \centering
    \includegraphics[width=\textwidth]{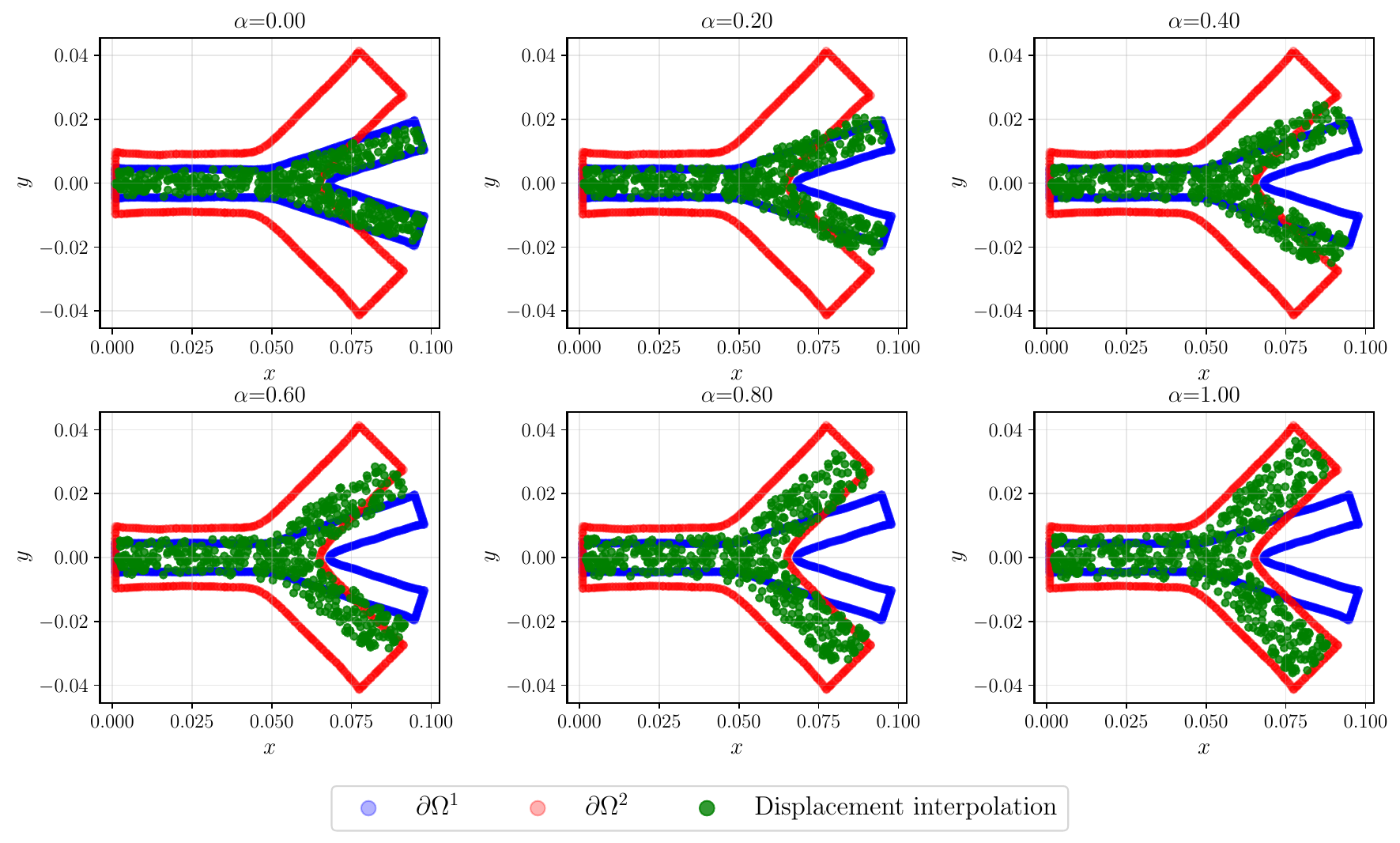}
    \caption{Displacement interpolation via $(1-\alpha)x + \alpha T_{\text{bary}}(x)$ (showing 500 randomly sampled source points for visualization clarity from the full \num{83208} vertices). The source points (green) deform smoothly from the source surface (blue outline) to the target surface (red outline) as $\alpha$ goes from 0 to 1, illustrating the smooth registration path found by the barycentric map.}
    \label{fig:transport_interpolation}
\end{figure}

\subsubsection{Application: Transporting Scalar Fields}

A powerful application of registration is the transfer of physical data from one geometry to another, employed for example in surrogate modelling~\cite{romor2025dataassimilationperformedrobust}. Using the computed map $T(x)$, which was derived from uniform measures, we can transport a scalar field defined on the source domain $\Omega_1$. Let $P_1: \Omega_1 \to \R$ be a scalar field, such as the pressure obtained from the incompressible Navier-Stokes equations. We define the transported field $P_{\text{transported}}$ on the deformed geometry $T(\Omega_1)$ by setting the value at a mapped point $y=T(x)$ to be equal to the value at the original point $x$:
\begin{equation}
\label{eq:pressure_transport}
P_{\text{transported}}(y) = P_{\text{transported}}(T(x)) \defeq P_1(x).
\end{equation}

Figure \ref{fig:field_transfer} compares the original pressure field with the fields transported via the barycentric and modal maps. The barycentric map (b) produces a smooth, continuous deformation of the pressure field onto the new geometry. Because this map averages the locations of the target vertices, the envelope of the transported points does not perfectly coincide with the boundary of the original target mesh. In contrast, the modal map (c) transports source points to a subset of the original target vertices. The transported points therefore lie exactly on the target geometry, including its boundary. This task is similar to color transfer in image processing \cite{Bonneel2023}; however, since the modal map can transport multiple source points to the same target vertex, for visualization purposes, the plot shows the source value of one of these points. An alternative approach would be to take an average of the field values that get mapped to the target point. This highlights the different nature of the two maps: the barycentric map is ideal for smooth field transfer and deformation analysis, while the modal map is more suited to segmentation or label transfer tasks. This example demonstrates the versatility of the RSOT framework in not only solving for geometric correspondence but also enabling subsequent physical analysis on the registered domains.

\begin{figure}[htbp]
    \centering
    \begin{subfigure}[b]{0.5\textwidth}
        \centering
        \includegraphics[width=\textwidth]{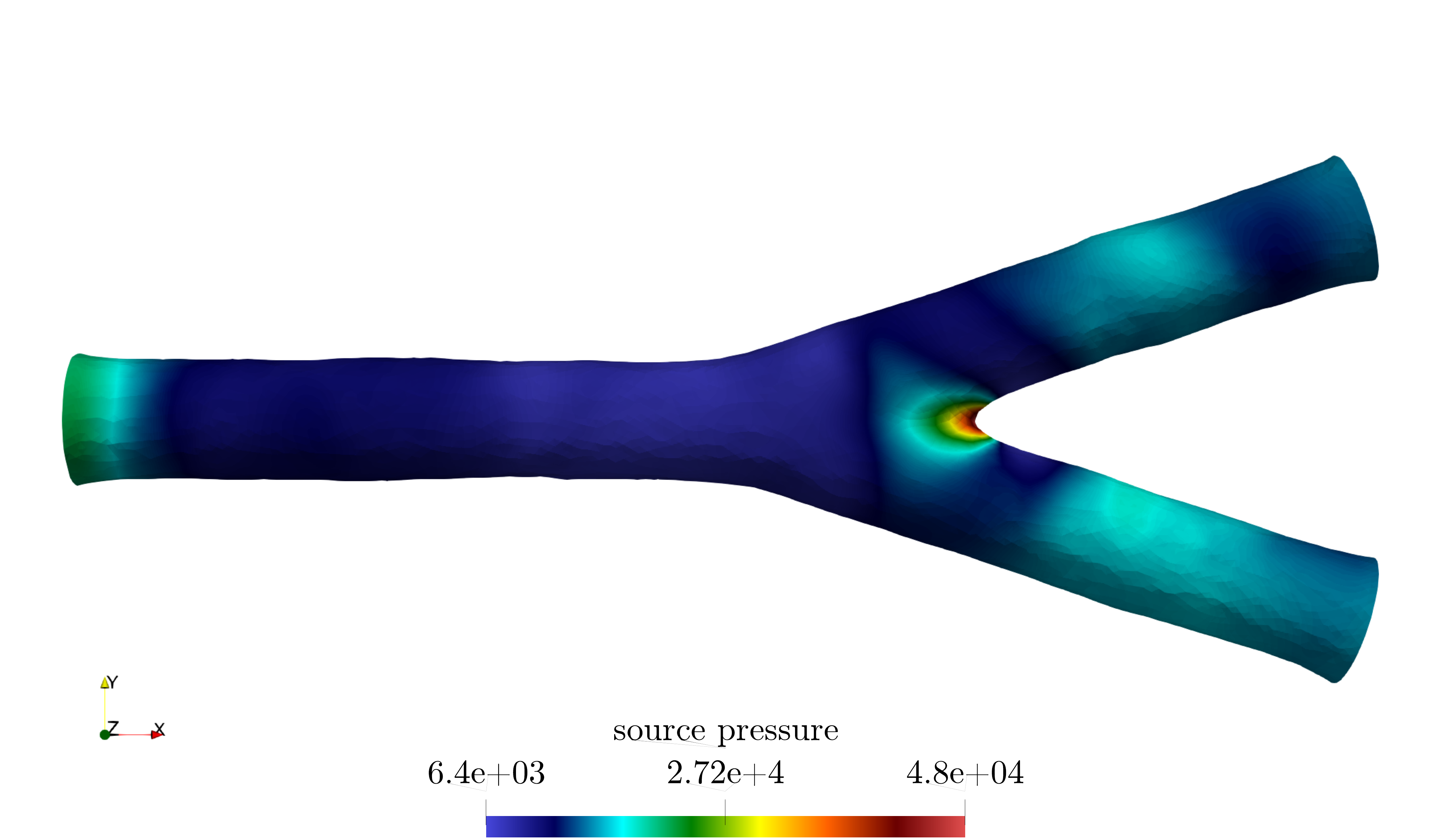}
        \caption{Original pressure field on the source domain $\Omega_1$.}
    \end{subfigure}

    \vspace{0.5cm}

    \begin{subfigure}[b]{0.49\textwidth}
        \centering
        \includegraphics[width=\textwidth]{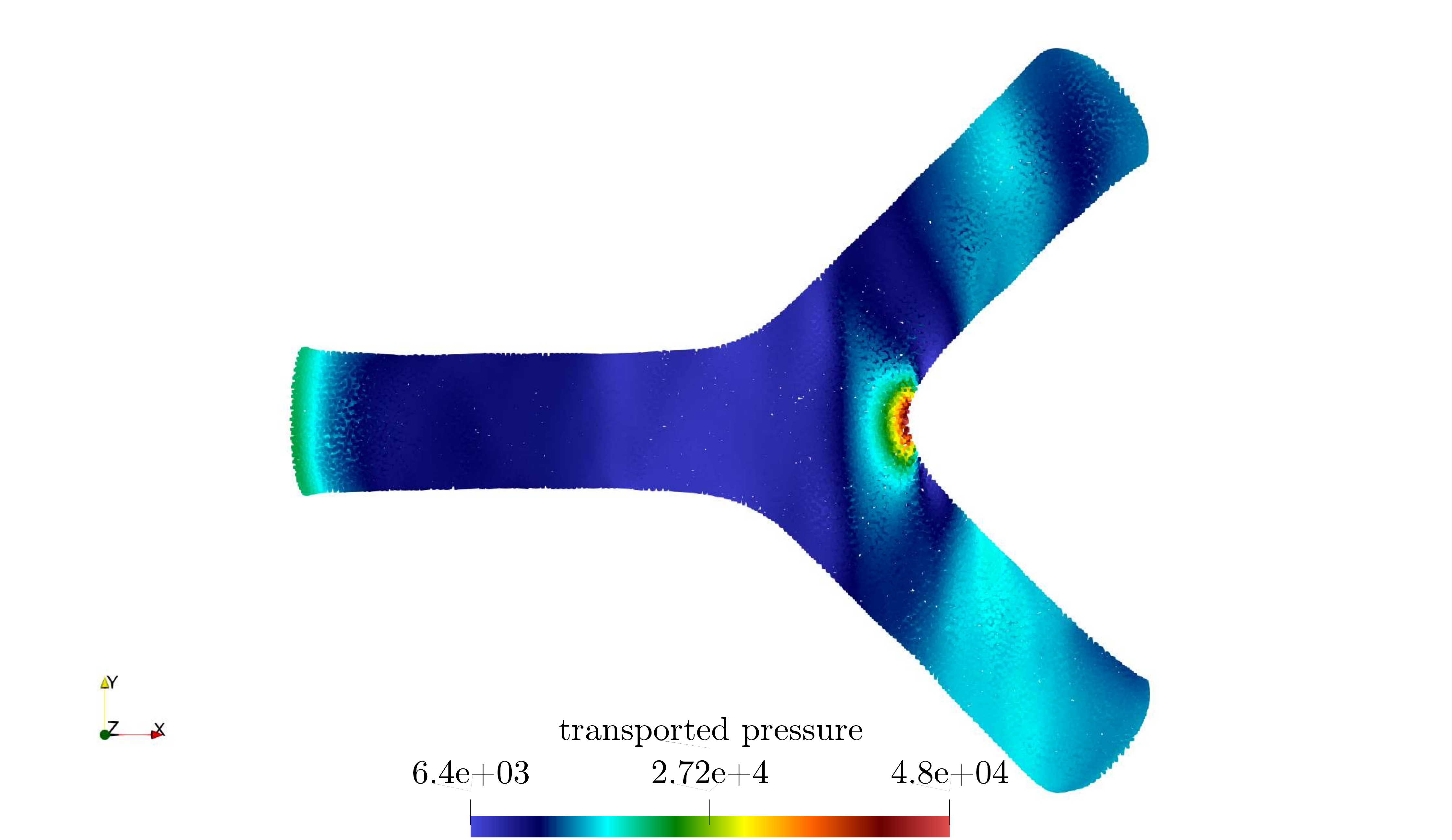}
        \caption{Transported Pressure (Barycentric)}
    \end{subfigure}
    \hfill
    \begin{subfigure}[b]{0.49\textwidth}
        \centering
        \includegraphics[width=\textwidth]{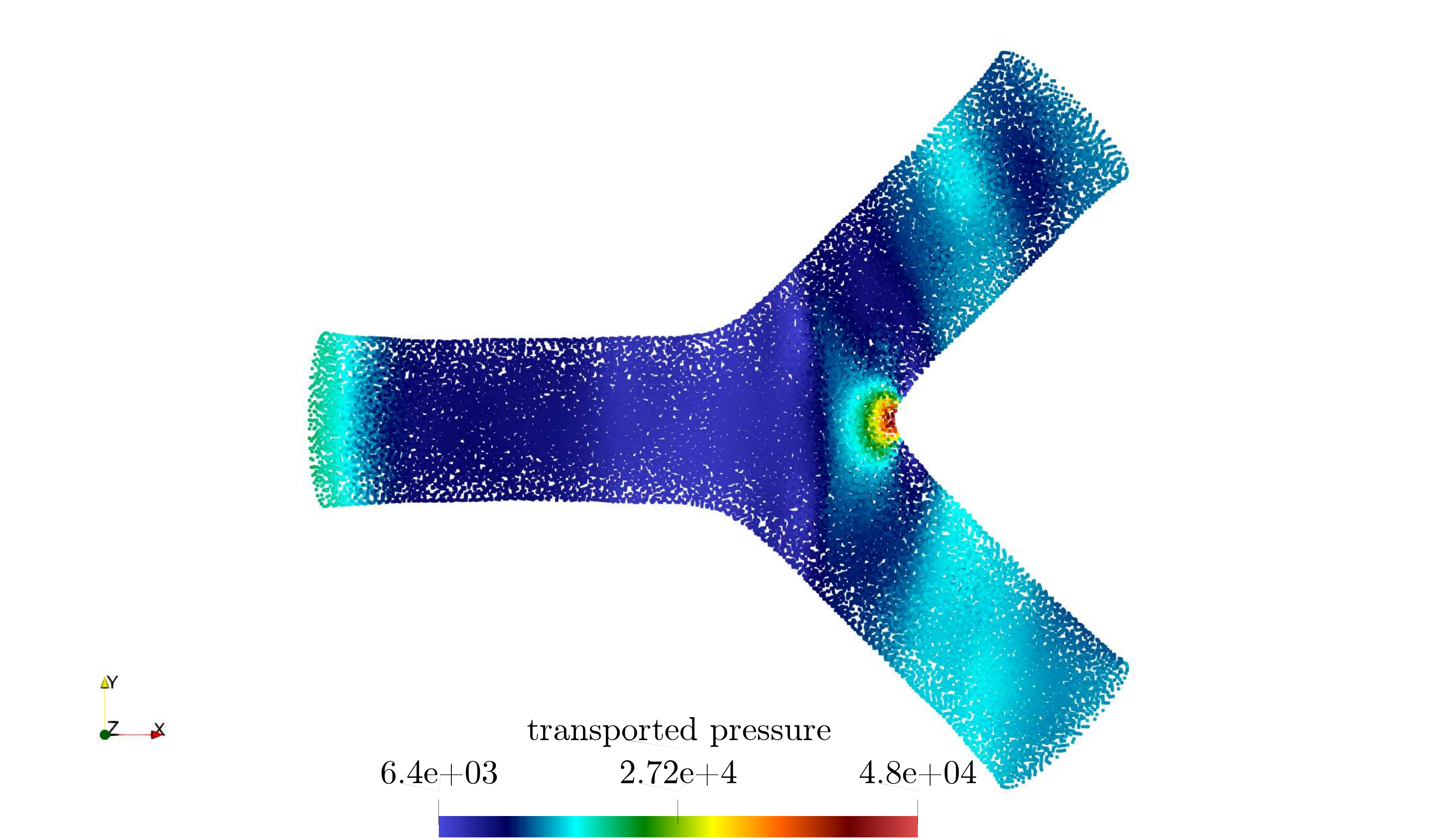}
        \caption{Transported Pressure (Modal)}
    \end{subfigure}
    \caption{Transfer of the pressure field from $\Omega_1$ to the target geometry. (a) The original pressure field on the source domain $\Omega_1$. (b) The smoothly deformed pressure field using the barycentric map. (c) The pressure field resulting from the modal map, where points are mapped directly to a subset of the target vertices.}
    \label{fig:field_transfer}
\end{figure} %
\subsection{Application: Blue Noise Sampling}
\label{subsec:blue_noise}

Blue noise sampling aims to generate a set of $N$ discrete points that are spatially well-distributed (i.e., uniform yet random-looking, avoiding clumps and voids) with respect to a given target density. This problem is fundamental in computer graphics for tasks like object placement, dithering, and Monte Carlo integration \cite{Lagae2007,Yuksel2015}. While our RSOT framework, as formulated in Section~\ref{sec:formulation}, assumes fixed target point locations, it can be effectively utilized as a core component within an iterative scheme to optimize these locations. This approach, conceptually similar to Lloyd-type algorithms \cite{degoes2012, Qin2017}, allows us to conclude by demonstrating the versatility of our framework with a generic, differentiable cost function $c(x,y)$, showcasing its power beyond the squared Euclidean metric.

The optimization problem is to find the locations $Y = \{y_1, \dots, y_N\}$ of a discrete measure $\nu(Y) = \frac{1}{N}\sum_{k=1}^N \delta_{y_k}$ that best represents a continuous source measure $\mu$ by minimizing the regularized Wasserstein cost:
\begin{equation}
    \min_{Y \subset \Y^N} F(Y), \quad \text{where} \quad F(Y) \defeq \mathcal{W}_{\varepsilon,c}(\mu, \nu(Y)).
    \label{eq:quantization_problem}
\end{equation}
The first-order necessary condition for a set of locations $Y^*$ to be optimal is that the gradient of the objective function $F(Y)$ with respect to each location $y_k$ must vanish. Using the Envelope Theorem, this condition is:
\begin{equation}
    \nabla_{y_k} F(Y^*) = \int_{\Omega} \nabla_{y_k} c(x, y_k^*) \dd\pi_k^*(Y^*)(x) = \mathbf{0},
    \label{eq:general_optimality_condition}
\end{equation}
where $\pi_k^*(Y^*)$ is the component of the optimal plan associated with site $y_k^*$. This system of equations is implicit and generally cannot be solved directly.

This structure motivates an iterative solution via the Lloyd's algorithm (Algorithm~\ref{alg:blue_noise}). The algorithm alternates between two steps: first, partitioning the source mass by solving the RSOT problem for fixed locations; and second, updating the locations by solving a subproblem for each site:
\begin{equation}
    y_k^{(t+1)} = \argmin_{y \in \Y} \int_{\Omega} c(x, y) \dd\pi_k^{(t)}(x).
    \label{eq:barycenter_subproblem_general}
\end{equation}
The solution to this subproblem is the point that minimizes the expected cost with respect to the mass it receives. We can define this as the generalized, \textit{cost-minimizing barycenter}. The Lloyd's algorithm is thus a fixed-point iteration where each site is updated to become the cost-minimizing barycenter of its assigned mass. When the iteration converges, the locations satisfy the optimality condition of Eq.~\eqref{eq:general_optimality_condition}.

While this framework is general, it is particularly powerful when the cost function is a squared metric distance, $c(x,y) = d(x,y)^2$. For this class of costs, the barycenter subproblem is often well-posed. In the specific case of a Riemannian manifold with the squared geodesic distance, $c(x,y) = d_g(x,y)^2$, the cost-minimizing barycenter is precisely the well-known \textit{Riemannian barycenter}, which can be computed via manifold-based gradient descent (Algorithm~\ref{alg:riemannian_barycenters}).

\begin{algorithm}[ht]
    \caption{Lloyd algorithm on Riemannian manifolds}
    \label{alg:blue_noise}
    \begin{algorithmic}[1]
    \Require Source meshe $\calT_h$, measure $\mu$ with density $\rho$, quadrature $\{x_{q}, w_{q}\}_{q=1}^{N_{q,i}}$, initial guess $\mathbf{y}^{t=0}=\{y^{t=0}_j\}_{j=1}^N, \nu=\{\nu_j\}_{j=1}^N$, cost function $c(\cdot, \cdot)$, $\varepsilon > 0$, tolerance $\delta_{\text{tol},\mathbf{y}}$, Riemannian manifold $\Omega$.
    \Ensure Sampled target support points locations $\mathbf{y}^*$.

    \Repeat (While $\lVert\delta \mathbf{y}\rVert_2 > \delta_{\text{tol},\mathbf{y}}\lVert\mathbf{y}\rVert_2$ and $t<T_{\text{max iterations}}$)
        \State Solve regularized semi-discrete optimal transport between the source measure $\mu$ and the target measure $(\nu, \mathbf{y}^t)$, minimizing $J^h_{\varepsilon}(\bpsi^t)$ with L-BFGS.
        \State Compute Riemannian barycenters with Algorithm~\ref{alg:riemannian_barycenters}, $$\mathbf{y}^{t+1}\leftarrow \textbf{RiemannianBarycenter}(\Omega, \{\mu_j(\cdot|\bpsi^t)\}_{j=1}^N).$$
        \State Compute increment $\delta\mathbf{y}\leftarrow \mathbf{y}^{t+1}-\mathbf{y}^t$.
        \State t = t + 1.
    \Until{convergence to $\mathbf{y}^*$.}
    \end{algorithmic}
\end{algorithm}

\begin{algorithm}[ht]
    \caption{Riemannian barycenters}
    \label{alg:riemannian_barycenters}
    \begin{algorithmic}[1]
    \Require Riemannian manifold with associated geodesic distance and Riemannian metric $(\Omega, d, g)$, conditioned praobability density $\{\mu_j(\cdot|\bpsi)\}_{j=1}^N$, tolerance $\delta_{\text{tol},\mathbf{y}}$, step $\alpha$.
    \Ensure Riemannian barycenters $\{y_j^*\}_{j=1}^{N}\subset\Omega$.

    \Repeat (While $g(\delta \mathbf{y}) > \delta_{\text{tol},\mathbf{y}}$ and $t<T_{\text{max iterations}}$)
        \State Compute gradient $\delta y_j$ of functional to minimize: $$\forall j, T_{y_j^t}\Omega\ni\delta y_j\leftarrow \nabla_{y_j^t} \int_{\Omega} d^2(y_j^t, x)\mu_j(x|\bpsi)\,dx = \int_{\Omega} -2\log_{y_j^t}(x)\mu_j(x|\bpsi)\,dx.$$
        \State Evolve the points along the manifold: $\forall j, \Omega\ni y_j^{(t+1)}\leftarrow \exp(-\alpha\delta y_j)$.
        \State t = t + 1.
    \Until{convergence to $\mathbf{y}^*$.}
    \end{algorithmic}
\end{algorithm}

We illustrate this process on the unit sphere, using the squared geodesic distance as the cost. The source measure is defined by a non-uniform density given by the fifth eigenfunction of the Laplace-Beltrami operator. Figure~\ref{fig:blue_noise} shows the converged point sets for an increasing number of target points ($N=4$ to $9$), with $\varepsilon=\SI{1e-2}{}$, $\delta_{\text{tol},\mathbf{y}}=\SI{1e-3}{}$. Starting from random initial positions sampled on the unit sphere, the algorithm correctly repositions the points to form a well-distributed blue noise pattern that respects both the source density and the manifold's curvature. The points naturally concentrate in the high-density regions (red/yellow) and are sparse in the low-density regions (blue), demonstrating the successful application of the general framework to a non-Euclidean problem.

\begin{figure}[htp!]
    \centering
    \includegraphics[width=0.7\linewidth]{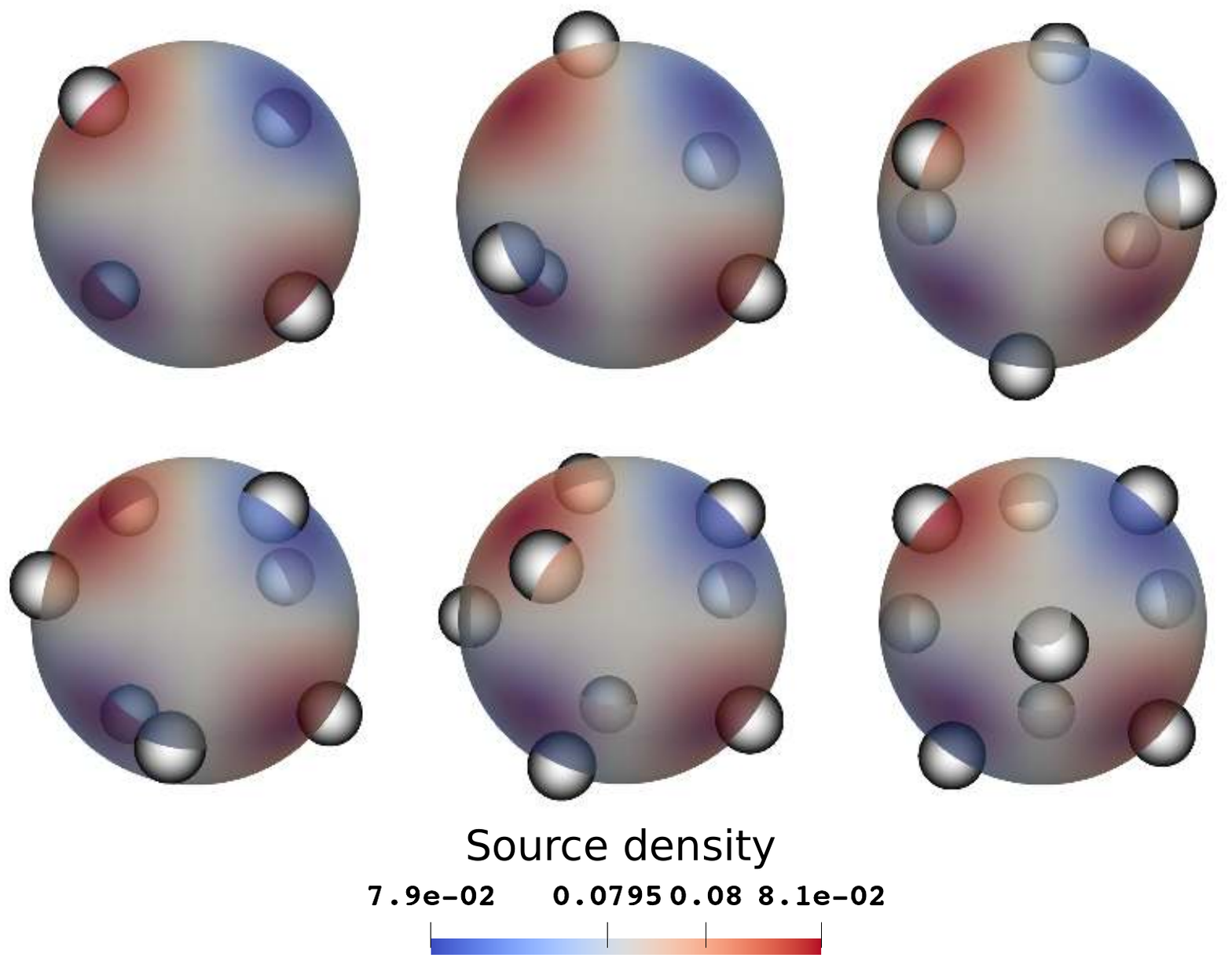}
        \caption{Blue noise sampling on the unit sphere via optimal quantization. The locations are optimized using the iterative Lloyd's algorithm (Algorithm~\ref{alg:blue_noise}) with a squared geodesic distance cost. From top to bottom, left to write an increasing number of target points with the same mass are sampled $N\in\{4,\dots, 9\}$.}
    \label{fig:blue_noise}
\end{figure}
\section{Conclusion}
\label{sec:conclusion}

In this work we have introduced a computational framework for solving large-scale regularized semi-discrete optimal transport (RSOT) problems, particularly those arising from complex geometries and PDE-based data discretized with the finite element method. Our main contribution is the integration of multiple acceleration strategies. By combining adaptive, error-controlled truncation of the Gibbs kernel with fast spatial queries using R-trees, we reduce the per-iteration cost of the dual optimization. This core acceleration is enhanced by a multilevel strategy, which leverages hierarchies of both the source mesh and target measure to provide improved initial guesses. The practical utility of this approach was demonstrated through numerical studies, including the computation of Wasserstein barycenters for 3D vascular geometries and a geometric registration application. The entire methodology is made available through the open-source C++ library \texttt{SemiDiscreteOT}.

Beyond the performance improvements demonstrated, this work opens several avenues for future research and application. A natural extension is the development of comprehensive GPU support. The core computational kernel, the evaluation of the dual functional and its gradient, is well-suited to parallelization and could benefit from the single-instruction, multiple-data (SIMD) architecture of modern GPUs. Porting this kernel would provide additional speedups, making real-time applications more feasible. Preliminary results in this regard are reported in the appendix~\ref{appendix:gpu}.

This work paves the way for applications at the intersection of optimal transport and numerical modeling with PDEs. Examples include: fluid-particle interactions~\cite{diez2024multicellularsimulationsshapevolume}; data assimilation in computational fluid dynamics~\cite{deGoes2015} with applications in hemodynamics and climatology; anisotropic material modeling through variations in non-uniform continuous source densities~\cite{Buze2024}; non-parametric variational inference~\cite{Ambrogioni2018} where optimal transport structures latent spaces; and computational optics~\cite{Rubinstein2017}, where the unregularized limit connects to the Monge-Ampère equation for reflector and lens design.

\section*{Acknowledgments}
\textbf{MK} and \textbf{GR} acknowledge the support provided by the European Union - NextGenerationEU, in the framework of the iNEST - Interconnected Nord-Est Innovation Ecosystem (iNEST ECS00000043 – CUP G93C22000610007) consortium.
The authors also like to acknowledge INdAM-GNCS for its support.

\printbibliography
\newpage
\appendix

\section{Proof of convergence of minimizers of regularized SOT discrete functional}
\label{appendix:proof}
\begin{proof}
    Standard error analysis results in the following estimates:
    \begin{align*}
        |J_\varepsilon(\bpsi) - J_\varepsilon^{h,k,r}(\bpsi)| &= |J_{\varepsilon}(\bpsi)-\int_{\Omega} I_k[f(x)]\rho_h(x)\,dx + \int_{\Omega} I_k[f(x)]\rho_h(x)\,dx-J_\varepsilon^{h,k,r}(\bpsi)|\\
        &\lesssim h^{k+1} ( \lVert \rho\lVert_1\lVert f\rVert_{H^{k+1}(\Omega})+\lVert f\lVert_1\lVert \rho\rVert_{H^{k+1}(\Omega)}+ h^{2r-k} \lVert f\rVert_{\mathcal{C}^{2r-k}}\lVert \rho\rVert_1\overset{h \to 0}{\rightarrow} 0,
    \end{align*}
    where the first term corresponds to the interpolation error and the second to the quadrature error (ultimately linked to the interpolation error with polynomials as Gauss quadrature rules of order $r$ are exact for polynomials of order $2r-1$). For the first term we have used $f\in\mathcal{C}^{k+1}$ and for the second $f\in\mathcal{C}^{2r-k}$. We represented with $I_k$ the approximation of $f$ with finite elements of order $k$. We remark that $\lVert f\rVert_{H^{k+1}(\Omega)}$ depends on $\bpsi$. Similarly,
    \begin{equation*}
        \|\nabla J_\varepsilon(\bpsi) - \nabla J_\varepsilon^{h,k,r}(\bpsi)\|_2^2 = \sum_{i=1}^N |(\nabla J_\varepsilon(\bpsi) - \nabla J_\varepsilon^{h,k,r}(\bpsi))_i|^2,
    \end{equation*}
    so that for every $i\in\{1,\dots N\}:$
    \begin{align*}
        |(\nabla J_\varepsilon(\bpsi) - \nabla J_\varepsilon^{h,k,r}(\bpsi))_i|&=|\nabla J_{\varepsilon}(\bpsi)-\int_{\Omega} I_k[g_i(x)]\rho_h(x)\,dx + \int_{\Omega} I_k[g_i(x)]\rho_h(x)\,dx-\nabla J_\varepsilon^{h,k,r}(\bpsi)|\\
        &\lesssim h^{k+1} ( \lVert \rho\lVert_1\lVert g_i\rVert_{H^{k+1}(\Omega)}+\lVert g_i\lVert_1\lVert \rho\rVert_{H^{k+1}(\Omega)}+ h^{2r-k} \lVert g_i\rVert_{\mathcal{C}^{2r-k}}\lVert \rho\rVert_1\overset{h \to 0}{\rightarrow} 0.
    \end{align*}
    We remark that $\lVert g_i\rVert_{\mathcal{C}^{2r-k}}$ depends on $\bpsi$, for all $i$.

    For the convergence of the minimizers $\bpsi^{h,k,r}_{\varepsilon}\overset{h \to 0}{\rightarrow}\bpsi_{\varepsilon}$, we only need the pointwise convergence of the convex and continuous functionals $J_\varepsilon$ and  $J_\varepsilon^{h,k,r}$.

    Since $J_{\varepsilon}(\bpsi + c\mathbf{1}) = J_{\varepsilon}(\bpsi)$ and $J^{h,k,r}_{\varepsilon}(\bpsi + c\mathbf{1}) = J^{h,k,r}_{\varepsilon}(\bpsi)$ for any $c \in \R$, we restrict to $V = \{\bpsi \in \R^N \mid \sum_{j=1}^N \psi_j \nu_j = 0\}$ for uniqueness.

    We prove equi-coercivity of $J^{h,k,r}:V\subset\mathbb{R}^N\rightarrow\mathbb{R}$ with respect to $h$: for all $\bpsi\in V$
    \begin{align*}
        J_\varepsilon^{h,k,r}(\bpsi) = &\sum_{K \in \calT_h} \sum_{q=1}^{N_{q,K}} \varepsilon \log\left(\sum_{j=1}^N \nu_j \exp\left(\frac{\psi_j - c(x_q^K, y_j)}{\varepsilon}\right)\right) \rho_h(x_q^K) w_q^K\\
        &\geq \sum_{K \in \calT_h} \sum_{q=1}^{N_{q,K}}\max_j \left(\psi_j - c(x^K_q, y_j) + \varepsilon\log \nu_j\right)\rho_h(x_q^K) w_q^K\\
        &\geq \min_{x\in\Omega}\max_j \left(\psi_j - c(x, y_j) + \varepsilon\log \nu_j\right)\left(\sum_{K \in \calT_h} \sum_{q=1}^{N_{q,K}}\rho_h(x_q^K) w_q^K\right)\\
        &= \min_{x\in\Omega}\max_j \left(\psi_j - c(x, y_j) + \varepsilon\log \nu_j\right)\geq\max_j \left(\psi_j - \max_{x\in\Omega}c(x, y_j) + \varepsilon\log \nu_j\right),
    \end{align*}
    and since $\Omega$ is a bounded domain, $\max_{x\in\Omega}c(x, y_j)$ is bounded and at least one component $\psi_j\to+\infty$ (otherwise $\bpsi\notin V$), we have that for all $h$ uniformly:
    \begin{equation*}
        \lVert \bpsi\rVert_2\to+\infty \Rightarrow J_\varepsilon^{h,k,r}(\bpsi)\to+\infty.
    \end{equation*}

    Thanks to equi-coercivity the sequence $\{\bpsi^{h,k,r}_{\varepsilon}\}_h$ is bounded and, due to compactness in $\mathbb{R}^N$, it  admits a minimizer $\overline{\bpsi}$. Convexity and equi-coercivity of $J_\varepsilon^{h,k,r}$ imply equi-continuity that together with uniform boundedness (thanks to equi-coercivity) imply uniform convergence (Ascoli-Arzelà theorem) of $J_\varepsilon^{h,k,r}$ to $J_\varepsilon$ up to subsequences (due to pointwise convergence). Using uniform convergence of $J_\varepsilon^{h,k,r}$ on compact subsets, we have
    \begin{equation*}
        J_\varepsilon(\overline{\bpsi}) = \lim_{h\to 0} J_\varepsilon^{h,k,r}(\bpsi^{h,k,r}_{\varepsilon}),
    \end{equation*}
    that implies,
    \begin{equation*}
        J_\varepsilon(\overline{\bpsi}) = \lim_{h\to 0} J_\varepsilon^{h,k,r}(\bpsi^{h,k,r}_{\varepsilon})\leq \limsup_{h\to 0} J_\varepsilon^{h,k,r}(\bpsi)=J_{\varepsilon}(\bpsi),\qquad \forall\bpsi\in\mathbb{R}^N
    \end{equation*}
    which proves that $\overline{\bpsi}$ is the minimizer of $J_{\varepsilon}$.
\end{proof}
\section{Derivation of Truncation Cutoffs for Global Error Control}
\label{app:truncation_derivation}

We derive the truncation bounds $C_{int}$ and $C_{geom}$ discussed in Section~\ref{subsubsec:truncation_revised} that ensure a prescribed relative error $\tau > 0$ on the continuous dual functional $J_\varepsilon(\bpsi)$ for the semi-discrete problem. We assume a general cost function $c(x,y)$ that satisfies the usual regularity conditions.

The exact dual functional is:
\[
J_\varepsilon(\bpsi) = \int_{\OmegaDomain} \varepsilon \ln\left(\sum_{j=1}^N \nu_j e^{(\psi_j - c(x, y_j))/\varepsilon}\right) \dd\mu(x) - \sum_{j=1}^N \psi_j \nu_j.
\]
Let $S(x, \bpsi)$ be defined as
\[
S(x, \bpsi) = \sum_{j=1}^N \nu_j \exp\left(\frac{\psi_j - c(x, y_j)}{\varepsilon}\right).
\]
We approximate this with a truncated sum $S_{\text{trunc}}(x, \bpsi; C)$,
\[
S_{\text{trunc}}(x, \bpsi; C) = \sum_{j \in I_C(x)} \nu_j \exp\left(\frac{\psi_j - c(x, y_j)}{\varepsilon}\right),
\]
where the index set $I_C(x)$ includes all indices $j$ for which the cost is below a threshold $C$:
\[
I_C(x) = \{j \in \{1,\dots,N\} : c(x, y_j) < C\}.
\]

The truncated functional is:
\[
\widetilde{J}_\varepsilon(\bpsi) = \int_{\OmegaDomain} \varepsilon \ln(S_{\text{trunc}}(x, \bpsi; C)) \dd\mu(x) - \sum_{j=1}^N \psi_j \nu_j.
\]
The error introduced in the functional is $E = J_\varepsilon(\bpsi) - \widetilde{J}_\varepsilon(\bpsi)$:
\[
E = \int_{\OmegaDomain} \varepsilon \ln\left(\frac{S(x, \bpsi)}{S_{\text{trunc}}(x, \bpsi; C)}\right) \dd\mu(x).
\]
Let $a_j(x) = \exp((\psi_j - c(x, y_j))/\varepsilon)$. Then $S(x, \bpsi)$ can be written as
\[
S(x, \bpsi) = S_{\text{trunc}}(x, \bpsi; C) + \sum_{j \notin I_C(x)} \nu_j a_j(x),
\]
so
\[
E = \int_{\OmegaDomain} \varepsilon \ln\left(1 + \frac{\sum_{j \notin I_C(x)} \nu_j a_j(x)}{S_{\text{trunc}}(x, \bpsi; C)}\right) \dd\mu(x).
\]
Using the inequality $\ln(1+z) \le z$ for $z \ge 0$, we get:
\[
E \le \int_{\OmegaDomain} \varepsilon \frac{\sum_{j \notin I_C(x)} \nu_j a_j(x)}{S_{\text{trunc}}(x, \bpsi; C)} \dd\mu(x).
\]
Assume bounds $m \le \psi_j \le M$ for all $j$. For $j \notin I_C(x)$, we have $c(x, y_j) \ge C$. The individual neglected terms are bounded by:
\[
a_j(x) = \exp\left(\frac{\psi_j - c(x, y_j)}{\varepsilon}\right) \le \exp\left(\frac{M - C}{\varepsilon}\right).
\]
The sum over neglected terms is bounded:
\[
\sum_{j \notin I_C(x)} \nu_j a_j(x) \le \left(\sum_{j \notin I_C(x)} \nu_j\right) e^{(M - C)/\varepsilon} \le \left(\sum_{j=1}^N \nu_j\right) e^{(M - C)/\varepsilon}.
\]
We assume the target measure is normalized, $\nu_{tot} = \sum_{j=1}^N \nu_j = 1$. Then
\[
E \le \int_{\OmegaDomain} \varepsilon \frac{e^{(M - C)/\varepsilon}}{S_{\text{trunc}}(x, \bpsi; C)} \dd\mu(x) = \varepsilon e^{(M - C)/\varepsilon} \int_{\OmegaDomain} \frac{\dd\mu(x)}{S_{\text{trunc}}(x, \bpsi; C)}.
\]
Define the integrated quantity $D(\bpsi, C)$ as
\[
D(\bpsi, C) = \int_{\OmegaDomain} [S_{\text{trunc}}(x, \bpsi; C)]^{-1} \dd\mu(x).
\]
We want to ensure the relative error is bounded: $E \le \tau |J_\varepsilon(\bpsi)|$. It suffices to satisfy:
\[
\varepsilon e^{(M - C)/\varepsilon} D(\bpsi, C) \le \tau |J_\varepsilon(\bpsi)|.
\]
Taking logarithms on both sides yields:
\[
\ln(\varepsilon D(\bpsi, C)) + \frac{M - C}{\varepsilon} \le \ln(\tau |J_\varepsilon(\bpsi)|).
\]
Rearranging for $C$, we get:
\[
C \ge M + \varepsilon \ln\left(\frac{\varepsilon D(\bpsi, C)}{\tau |J_\varepsilon(\bpsi)|}\right).
\]
This provides the condition for the \emph{integrated bound} $C_{int}$:
\[
C_{int}(\bpsi, \varepsilon, \tau, D) \ge M + \varepsilon \ln\left( \frac{\varepsilon D(\bpsi, C)}{\tau |J_\varepsilon(\bpsi)|} \right).
\]
This matches equation \eqref{eq:truncation_cutoff_int} in the main text.

To derive the \emph{geometric bound} $C_{geom}$, we estimate $D(\bpsi, C)$ using worst-case analysis. Let $C_0$ be the covering cost:
\[
C_0 = \max_{x \in \OmegaDomain} \min_{1 \le j \le N} c(x, y_j).
\]
For any $x \in \OmegaDomain$, there exists at least one index $j_0$ such that $c(x, y_{j_0}) \le C_0$. If we choose our truncation cutoff $C \ge C_0$, then this index $j_0$ will always be included in the sum $S_{\text{trunc}}(x, \bpsi; C)$, i.e., $j_0 \in I_C(x)$. Let $\underline{\nu} = \min_{1 \le j \le N} \nu_j$ and $m = \min_k \psi_k$.
Then, for any $x \in \OmegaDomain$:
\[
S_{\text{trunc}}(x, \bpsi; C) \ge \nu_{j_0} a_{j_0}(x) = \nu_{j_0} \exp\left(\frac{\psi_{j_0} - c(x, y_{j_0})}{\varepsilon}\right) \ge \underline{\nu} \exp\left(\frac{m - C_0}{\varepsilon}\right).
\]
This provides a uniform lower bound for $S_{\text{trunc}}(x, \bpsi; C)$. Therefore,
\[
\frac{1}{S_{\text{trunc}}(x, \bpsi; C)} \le \frac{1}{\underline{\nu}} \exp\left(-\frac{m - C_0}{\varepsilon}\right).
\]
Now we can bound $D(\bpsi, C)$:
\begin{align*}
D(\bpsi, C) &= \int_{\OmegaDomain} \frac{\dd\mu(x)}{S_{\text{trunc}}(x, \bpsi; C)} \\
&\le \int_{\OmegaDomain} \left( \frac{1}{\underline{\nu}} e^{-(m - C_0)/\varepsilon} \right) \dd\mu(x) \\
&= \frac{1}{\underline{\nu}} e^{-(m - C_0)/\varepsilon} \int_{\OmegaDomain} \dd\mu(x).
\end{align*}
Assuming the source measure is also normalized, $\mu(\OmegaDomain) = \int_{\OmegaDomain} \dd\mu(x) = 1$:
\[
D(\bpsi, C) \le \frac{1}{\underline{\nu}} e^{-(m - C_0)/\varepsilon}.
\]
Substituting this upper bound for $D(\bpsi, C)$ into the inequality derived previously for $C$:
\[
C \ge M + \varepsilon \ln\left(\frac{\varepsilon}{\tau |J_\varepsilon(\bpsi)|} \cdot \frac{1}{\underline{\nu}} e^{-(m - C_0)/\varepsilon} \right),
\]
which simplifies to
\[
C \ge M + \varepsilon \ln\left(\frac{\varepsilon}{\underline{\nu} \tau |J_\varepsilon(\bpsi)|}\right) + \varepsilon \left(-\frac{m - C_0}{\varepsilon}\right),
\]
so that
\[
C \ge M - m + C_0 + \varepsilon \ln\left(\frac{\varepsilon}{\underline{\nu} \tau |J_\varepsilon(\bpsi)|}\right).
\]
Using the potential range $\Gamma(\bpsi) = M-m$, we obtain the condition for the geometric truncation bound $C_{geom}$:
\[
C_{geom}(\bpsi, \varepsilon, \tau) \ge C_0 + \Gamma(\bpsi) + \varepsilon \ln\left(\frac{\varepsilon}{\underline{\nu} \tau |J_\varepsilon(\bpsi)|}\right).
\]
This justifies the formula \eqref{eq:truncation_cutoff_geom_revised} presented in the main text. Note that this bound requires $C \ge C_0$. If the calculated $C_{geom}$ satisfying the inequality is less than $C_0$, one should use $C=C_0$ instead.

\section{GPU acceleration of regularized semi-discrete optimal transport}
\label{appendix:gpu}
As computing paradigms shift from distributed-memory CPU architectures using MPI toward data center environments optimized for GPU acceleration, we demonstrate the performance benefits of regularized semi-discrete optimal transport when implemented on GPUs in Figure~\ref{fig:gpu}. GPU acceleration is enabled through the external open-source library \texttt{Kokkos}~\cite{9485033}. The test case setting is similar to the one in subsection~\ref{subsec:benchmarking}. The source and target measures are uniform. We consider different refinement levels of the source and target meshes and vary the Gauss quadrature order from $1$ to $3$ (even though the source probability density is a constant), the resulting dimensions of the regularized semi-discrete optimal transport problem are shown in Table~\ref{tab:gpu_data}. We remark that the dimension of the potential vector $\bpsi$ is equal to the number of target points $N$. The regularization parameter is $\varepsilon=\SI{1e-2}{}$ and the relative tolerance $\delta_{\text{tol}}=\SI{1e-3}{}$.

The tests were performed on a ProLiant XL675d Gen10 Plus 2xAMD EPYC 7402, 24 Core 2800 MHz, 1024 GB RAM CPU with 1 GPU NVIDIA A100-SXM4 40GB. Number of threads for \texttt{OpenMP} is set to $96$. Regarding the refinement level $5$, only the GPU implementation with Gauss quadrature order $1$ converged in less than $24$ hours: for the other cases, only the average L-BFGS iteration wall time is reported. We remark that, in this preliminary stage, no thresholding of log-sum-exp terms, and no source or target multi-level strategy has been performed. To achieve convergence in reasonable time for the missing data (e.g. GPU with $5$ refinements and Gauss quadrature order $3$), the same approaches introduced in section~\ref{sec:multilevel} can be applied. Additionally, MPI-based distributed-memory implementations with GPU support are readily available through the integration of \texttt{deal.II} and \texttt{Kokkos}.

\begin{table}[htbp!]
  \centering
  \caption{RSOT with increasingly refined meshes and Gauss quadrature order.}
  \label{tab:gpu_data}
  \begin{tabular}{lrrrrr}
      \toprule
      \textbf{N of Refinements} & \multicolumn{3}{c}{\textbf{Source Mesh ($\mathcal{T}_h^1$)}} & \multicolumn{2}{c}{\textbf{Target Points (from $\mathcal{T}_h^2$)}} \\
      \cmidrule(l{2pt}r{2pt}){2-4} \cmidrule(l{2pt}r{2pt}){5-6}
        & Gauss 1 & Gauss 2 & Gauss 3 & \multicolumn{2}{c}{Points ($N$)} \\
      \midrule
      2 (Coarsest) & \num{512} & \num{4096} & \num{13824} & \multicolumn{2}{c}{\num{507}} \\
      3 & \num{4096} & \num{32768} & \num{110592} & \multicolumn{2}{c}{\num{3817}} \\
      4& \num{32768} & \num{262144} & \num{884736} & \multicolumn{2}{c}{\num{29521}} \\
      5 (Finest) & \num{262144} & \num{2097152} & \num{7077888} & \multicolumn{2}{c}{\num{232609}} \\
      \bottomrule
  \end{tabular}
\end{table}

\begin{figure}[htp!]
    \centering
    \includegraphics[width=1\linewidth]{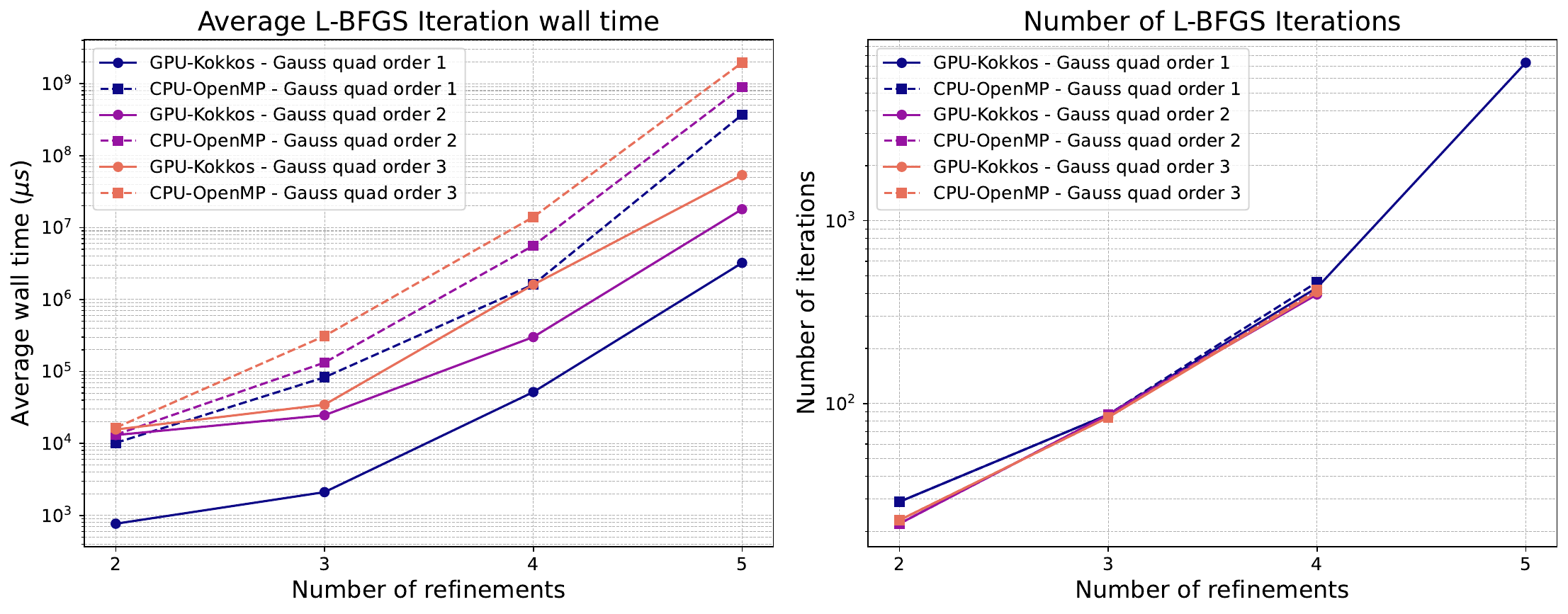}
    \caption{Comparison between accelerators for solving the regularized semi-discrete optimal transport problem with increasingly refined meshes and higher Gauss quadrature orders: OpenMP with $96$ CPU threads against $1$ GPU.}
    \label{fig:gpu}
\end{figure}
 \end{document}